\documentclass[twoside,11pt]{article}

\usepackage[preprint]{jmlr2e}

\usepackage{JMLR-essay-def}
\usepackage{appendix}
\usepackage{mathrsfs}
\usepackage{rotating}
\usepackage{makecell}
\usepackage{cases}
\usepackage{bm}
\usepackage{algorithm}
\usepackage{algorithmic}
\usepackage{multirow}
\usepackage{tikz}
\usetikzlibrary{shapes,arrows,matrix}

\jmlrheading{1}{2020}{1-48}{4/00}{10/00}{meila00a}{Yifei Wang and Wuchen Li}

\ShortHeadings{Information Newton's flow}{Wang and Li}
\firstpageno{1}

\begin{document}

\title{Information Newton's flow: second-order optimization method in probability space}

\author{\name Yifei Wang \email wangyf18@stanford.edu \\
       \addr Department of Electrical Engineering\\
       Stanford University\\
       Stanford, CA 94305-9505, USA
       \AND
       \name Wuchen Li \email wcli@math.ucla.edu \\
       \addr Department of Mathematics\\
       University of California\\
       Los Angeles, CA 90095-1555, USA}

\editor{}

\maketitle

\begin{abstract}
We introduce a framework for Newton's flows in probability space with information metrics, named information Newton's flows. Here two information metrics are considered, including both the Fisher-Rao metric and the Wasserstein-2 metric. A known fact is that overdamped Langevin dynamics correspond to Wasserstein gradient flows of Kullback-Leibler (KL) divergence. Extending this fact to Wasserstein Newton's flows, we derive Newton's Langevin dynamics. We provide examples of Newton's Langevin dynamics in both one-dimensional space and Gaussian families. 
For the numerical implementation, we design sampling efficient variational methods in affine models and reproducing kernel Hilbert space (RKHS) to approximate Wasserstein Newton's directions. We also establish convergence results of the proposed information Newton's method with approximated directions. Several numerical examples from Bayesian sampling problems are shown to demonstrate the effectiveness of the proposed method.
\end{abstract}


\section{Introduction}
Optimization problems in probability space are of great interest in inverse problems, information science, physics, and scientific computing, with applications in machine learning \citep{igaia, ivabp, Liu2017_steina, ngwei, tiot}. One typical problem here comes from Bayesian inference, which provides an optimal probability formulation for learning models from observed data. Given a prior distribution, the problem is to generate samples from a (target) posterior distribution \citep{ivabp}. From an optimization perspective, such a problem often refers to minimizing an objective function, such as the Kullback-Leibler (KL) divergence, in the probability space. The update relates to finding a sampling representation for the evolution of the probability. 

In practice, one often needs to transfer probability optimization problems into sampling-based formulations, and then design efficient updates in the form of samples. Here first-order methods, such as gradient descent methods, play essential roles. 
We notice that gradient directions for samples rely on the metric over the probability space, which reflects the change of objective/loss functions.
In practice, there are several important metrics, often named information metrics from information geometry and optimal transport, including the Fisher-Rao metric \citep{ngwei} and the Wasserstein-$2$ metric (in short, Wasserstein metric) \citep{tdmac,tgode}. In literature, along with a given information metric, the probability space can be viewed as a Riemannian manifold, named density manifold \citep{tdmac}.

For the Fisher-Rao metric, its gradient flow, known as birth-death dynamics, are important in modeling population games and designing evolutionary dynamics \citep{igaia}. It is also important for optimization problems in discrete probability 
\citep{cowig} and machine learning \citep{igoaa}. Recently, the Fisher-Rao gradient has also been applied for accelerating Bayesian sampling problems in continuous sample space \citep{alswb}. The Fisher-Rao gradient direction also inspires the design of learning algorithms for probability models. Several optimization methods in machine learning approximate the Fisher-Rao gradient direction, including the Kronecker-factored Approximate Curvature (K-FAC) \citep{onnwk} method and adaptive estimates of lower-order moments (Adam) method \citep{adam}. 

For the Wasserstein metric, its gradient direction deeply connects with stochastic differential equations and the associated Markov chain Monte Carlo methods (MCMC). An important fact is that the Wasserstein gradient of KL divergence forms the Kolmogorov forward generator of overdamped Langevin dynamics \citep{tvfot}. Hence, many MCMC methods can be viewed as Wasserstein gradient descent methods. In recent years, there are also several generalized Wasserstein metrics, such as Stein metric \citep{SVGD, Liu2017_steina}, Hessian transport (mobility) metrics \citep{CARRILLO20101273, Dolbeault2009, Li2019} and Kalman-Wasserstein metric \citep{ildgs}. These metrics introduce various first-order methods with sampling efficient properties. For instance, the Stein variational gradient descent \citep[SVGD]{SVGD} introduces a kernelized interacting Langevin dynamics. The Kalman-Wasserstein metric introduces a particular mean-field interacting Langevin dynamics \citep{ildgs}, known as ensemble Kalman sampling. On the other hand, many approaches design fast algorithms on modified Langevin dynamics. These methods can also be viewed and analyzed by the modified Wasserstein gradient descent, see details in \citep{itaao,sqnmc,DBLP:journals/corr/abs-1907-12546}. By viewing sampling as optimization problems in the probability space, many efficient sampling algorithms are inspired by classical optimization methods. E.g., \citet{lmcaj,plarc} apply the operator splitting technique to improve the unadjusted Langevin algorithm. \citet{afomo,affpd,aigf} study Nesterov's accelerated gradient methods in probability space.

In optimization, the Newton's method is a fundamental second-order method to accelerate optimization computations. For optimization problems in probability space, several natural questions arise: \textit{Can we systematically design Newton's methods to accelerate sampling related optimization problems? What is the Newton's flow in probability space under information metrics? Focusing on the Wasserstein metric, can we extend the relation between Wasserstein gradient flow of KL divergence and Langevin dynamics? In other words, what is the Wasserstein Newton's flow of KL divergence and which Langevin dynamics does it corresponds to?} 

In this paper, following \citep{gopsv, aigf}, we complete these questions. We derive Newton's flows in probability space with general information metrics. By studying these Newton's flows, we provide the convergence analysis.
Focusing on Wasserstein Newton's flows of KL divergence, we derive several analytical examples in one-dimensional space and Gaussian families. 
Besides, we design two algorithms as particle implementations of Wasserstein Newton's flows in high dimensional sample space.
This is to restrict the dual variable (cotangent vector) associated with Newton's direction into either finite-dimensional affine function space or RKHS. 
A hybrid update of Newton's direction and gradient direction is also introduced. For the concreteness of presentation, we demonstrate the Wasserstein Newton's flow of KL divergence in Theorem \ref{thm:wnf_kl}. 
\begin{theorem}[Wasserstein Newton's flow of KL divergence]
\label{thm:wnf_kl}
For a density $\rho^*(x)\propto \exp(-f(x))$ , where $f$ is a given function, denote the KL divergence between $\rho$ and $\rho^*$ by
\begin{equation}
\mathrm{D}_{\textrm{KL}}(\rho\|\rho^*) = \int\rho \log\frac{\rho}{e^{-f}}dx-\log Z,
\end{equation}
where $Z=\int \exp(-f(x))dx$. Then the Wasserstein Newton's flow of KL divergence follows
\begin{equation}\label{WN}
\p_t\rho_t+\nabla\cdot(\rho_t\nabla\Phi_t^{\operatorname{Newton}}) = 0,
\end{equation}
where $\Phi_t^{\operatorname{Newton}}$ satisfies the following equation
\begin{equation}\label{equ:wnf_kl}
    \nabla^2:(\rho_t\nabla^2\Phi_t)-\nabla\cdot(\rho_t\nabla^2f\nabla \Phi_t) - \nabla\cdot(\rho_t\nabla f)-\Delta \rho_t=0.
\end{equation}
\end{theorem}
Here we notice that $\Phi_t^{\operatorname{Newton}}$ is the solution to the Wasserstein Newton's direction equation \eqref{equ:wnf_kl}. In Figure \ref{figure1}, we provide a sampling (particle) formulation of Wasserstein Newton's flows. We compare formulations among Wasserstein Newton's flows, Wasserstein gradient flows and overdamped Langevin dynamics.
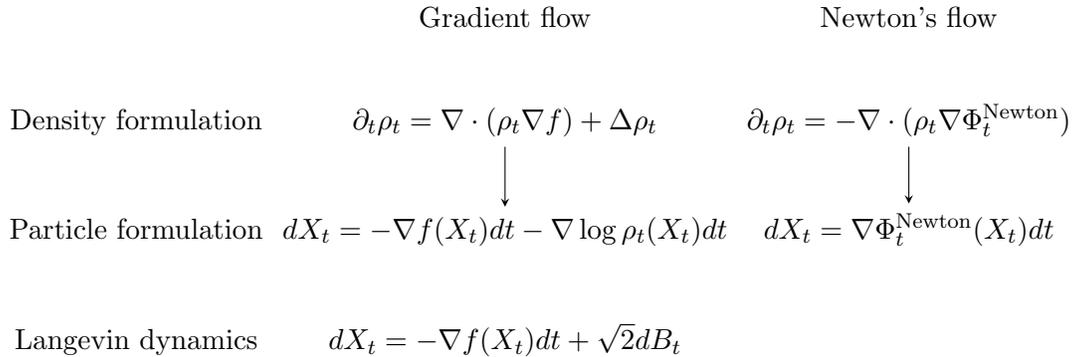
\begin{figure}[ht]
\centering
\begin{tikzpicture}
  \matrix (m) [matrix of math nodes,row sep=2em,column sep=0em,minimum width=1em]
  {
    &\text{Gradient flow} & \text{Newton's flow}\\
    \text{Density formulation}& \p_t\rho_t =\nabla\cdot(\rho_t\nabla f)+\Delta \rho_t & \p_t\rho_t = -\nabla\cdot(\rho_t\nabla\Phi_t^{\operatorname{Newton}})\\
    \text{Particle formulation}& dX_t = -\nabla f(X_t)dt-\nabla\log \rho_t(X_t)dt & dX_t = \nabla\Phi_t^{\operatorname{Newton}}(X_t)dt \\
    \text{Langevin dynamics}&dX_t = -\nabla f(X_t)dt+\sqrt{2}dB_t&\\};
  \path[-stealth]
    (m-2-2) edge node [below] {} (m-3-2)
    (m-2-3) edge node [below] {} (m-3-3);
\end{tikzpicture}
\caption{The relation among Wasserstein gradient flow, Newton's flow and Langevin dynamics. Our approach derive the particle formulation of Wasserstein Newton's flow of KL divergence.}
\label{figure1}
\end{figure}

In literature, second-order methods are developed for optimization problems on Riemannian manifold, see \citep{otorm, gcoao}. Here we are interested in density manifolds, i.e., probability space with information metrics. Compared to known results in Riemannian optimization, we not only develop methods in probability space but also find efficient sampling representations of the algorithms. In discrete probability simplex with the Fisher-Rao metric and exponential family models, the Newton's method has also been studied by \citet{cowig}, known as the second order method in information geometry. Also, \citet{asvnm,psvna} design second-order methods for the Stein variational gradient descent direction. Our approach generalizes these results to information metrics, especially for the Wasserstein metric. 
On the other hand, the Newton-type MCMC method has been studied in \citep{sqnmc}, known as Hessian Approximated MCMC (HAMCMC) method. 
The differences between HAMCMC and our proposed Newton's Langevin dynamics can be observed from evolutions in probability space. HAMCMC utilizes the Hessian matrix of logarithm of target density function and derives the associated drift-diffusion process. In density space, it is still a linear local partial differential equation (PDE). 
Newton's Langevin dynamics apply the Hessian operator of KL divergence based on the Wasserstein metric. In density space, the Wasserstein Newton's flow is a nonlocal PDE. 
A careful comparison of all related Langevin dynamics in analytical (Appendix \ref{NLD1D}) and numerical examples are provided. 

We organize this paper as follows. In section \ref{section2}, we briefly review information metrics and corresponding gradient operators in probability space. We introduce properties of Hessian operators and derive 
information Newton's flows in section \ref{section3}. Focusing on Wasserstein Newton's flows of KL divergence, we 
derive Newton's Langevin dynamics in section \ref{sec:nld}. Two sampling efficient numerical algorithms of Wasserstein Newton's method are presented in section \ref{sec:par}. In section \ref{sec:conv}, we prove the asymptotic convergence rate of information Newton's method with approximated Newton's direction. Several numerical examples for sampling problems are provided in section \ref{sec:num}.

\section{Review on Newton's flows and information metrics}\label{section2}
In this section, we briefly review Newton's methods and Newton's flows in Euclidean spaces and Riemannian manifolds. Then, we focus on a probability space, in which we introduce information metrics with the associated gradient and Hessian operators. Based on them, we will derive the Newton's flow under information metrics later on. Throughout this paper, we use $\la\cdot,\cdot \ra $ and $\|\cdot\|$ to denote the Euclidean inner product and norm in $\mbR^d$. 

\subsection{Finite dimensional Newton's flow}
We first briefly review Newton's methods and Newton's flows in Euclidean spaces. Given an objective function $f\colon \mathbb{R}^d\rightarrow \mathbb{R}$, consider an optimization problem:
$$
\min_{x\in \mbR^d} f(x).
$$
The update rule of the (damped) Newton's method follows
$$
x_{k+1} = x_k+\alpha_kp_k,\quad p_k=- \pp{\nabla^2f(x_k)}^{-1}\nabla f(x_k).
$$
Here $\alpha_k>0$ is a step size and $p_k$ is called the Newton's direction. With $\alpha_k=1$, we recover the classical Newton's methods. By taking a limit $\alpha_k\to 0$, the Newton's method in continuous-time, namely Newton's flow, writes
\begin{equation}\label{equ:new_eud}
    \dot x = -\pp{\nabla^2 f(x)}^{-1}\nabla f(x).\tag{Euclidean Newton's flow}
\end{equation}

We next consider an optimization problem on a Riemannian manifold $\mcM\subset \mbR^d$. Given an objective function $f\colon \mcM\rightarrow \mathbb{R}$, consider 
$$
\min_{x\in \mcM} f(x).
$$
The tangent space $T_x\mcM$ and the cotangent space $T_x^*\mcM$ at $x$ are identical to a linear subspace of $\mbR^d$. For $p,q\in T_x\mcM$, let $\lra{p,q}_x=p^T\mcG(x)q$ denote an inner product in tangent space $T_x\mcM$ at $x$. Here $\mcG(x)$ is called the metric tensor, which corresponds to a symmetric semi-positive definite matrix in $\mbR^{d\times d}$. For the Euclidean case, we can view $T_x\mcM=T_x^*\mcM=\mbR^d$ and $\mcG(x)=I$, where $I$ is an identity matrix. The Riemannian gradient of $f$ at $x$ is the unique tangent vector $v$ such that the following equality holds for all $p\in T_x\mcM$.
$$
\lra{\grad f(x),p}_x = \lim_{\epsilon\to0}\frac{f(x+\epsilon p)-f(x)}{\epsilon}.
$$
The Riemannian Hessian of $f$ at $x$ is a linear mapping from $T_x\mcM$ to $T_x\mcM$ defined by
$$
\Hess f(x) p = \nabla_p \grad f(x),\quad \forall p\in T_x\mcM.
$$
Here $\nabla_p \grad f(x)$ is the covariant derivative of $\grad f(x)$ w.r.t. the tangent vector $p$. Detailed definitions of gradient and Hessian operators on a Riemannian manifold can be found in \citep[Chapter 1]{oaorm}. The update rule of the Newton's method writes
$$
x_{k+1} = R_{x_k}(\alpha_kp_k), \quad p_k = -(\Hess f(x_k))^{-1}\grad f(x_k).
$$
Here $R_{x_k}$ can be the exponential mapping or the retraction (first-order approximation of the exponential mapping) at $x_k$. Based on the Riemannian metric of $\mcM$, the exponential mapping uniquely maps a tangent vector to a point in $\mcM$ along the geodesic curve. Different from the Euclidean case, the update of $x_{k+1}$ is based on the (approximated) geodesic curve of $\mcM$. In continuous time, the Newton's flow follows
\begin{equation}\label{equ:new_rie}
    \dot x = -(\Hess f(x))^{-1}\grad f(x). \tag{Riemannian Newton's flow}
\end{equation}
From now on, we consider optimization problems in probability space. Suppose that sample space $\Omega$ is a region in $\mbR^d$. Let $\mcF(\Omega)$ represent the set of smooth functions on $\Omega$. Denote the set of probability density
\begin{equation*}
\mcP(\Omega) =\Big\{\rho\in \mathcal{F}(\Omega)\colon \int_\Omega \rho dx=1,\quad \rho\geq0\Big\}.
\end{equation*}
The optimization problem in $\mcP(\Omega)$ takes the form:
$$
\min_{\rho\in\mcP(\Omega)} E(\rho).
$$ 
Here $E(\rho)$ is the objective or loss functional. It evaluates certain divergence or metric functional between $\rho$ and a target density $\rho^*\in\mcP(\Omega)$. In machine learning problems, typical examples of $E(\rho)$ include the KL divergence, Maximum mean discrepancy (MMD), cross entropy, etc. Similar to \eqref{equ:new_eud} and \eqref{equ:new_rie}, the Newton's flow in probability space (density manifold) takes the form
\begin{equation}\label{equ:new_prob}
    \p_t\rho_t = -(\Hess E(\rho_t))^{-1}\grad E(\rho_t).\tag{Information Newton's flow}
\end{equation}
Here $\grad$ and $\Hess$ represent the gradient and the Hessian operator with respect to certain information metric, respectively. To understand \eqref{equ:new_prob}, we briefly review the information metrics with the associated gradient operators.

\subsection{Information metrics}
We first define the tangent space and the cotangent space in probability space. The tangent space at $\rho\in\mcP(\Omega)$ is defined by 
$$
T_\rho\mathcal{P}(\Omega)=\left\{\sigma\in\mcF(\Omega):\int \sigma dx=0\right\}.
$$
The cotangent space $T^*_\rho\mcP(\Omega)$ is equivalent to $\mcF(\Omega)/\mbR$, which represents the set of functions in $\mcF(\Omega)$ defined up to addition of constants. 

\begin{definition}[Metric in probability space]
For a given $\rho\in \mcP(\Omega)$, a metric tensor $\mcG(\rho): T_\rho\mcP(\Omega)\to T^*_\rho\mcP(\Omega)$ is an invertible mapping from the tangent space $ T_\rho\mcP(\Omega)$ to the cotangent space $ T^*_\rho\mcP(\Omega)$. This metric tensor defines the metric (inner product) on the tangent space $T_\rho\mcP(\Omega) $. Namely, for $\sigma_1, \sigma_2\in T_\rho\mcP(\Omega)$, we define the inner product $g_\rho \colon T_\rho\mcP(\Omega)\times T_\rho\mcP(\Omega)\rightarrow\mathbb{R}$ by 
$$
g_\rho(\sigma_1,\sigma_2) =\int   \sigma_1\mcG(\rho)\sigma_2dx=\int   \Phi_1\mcG(\rho)^{-1}\Phi_2dx,
$$
where $\Phi_i$ is the solution to $\sigma_i= \mcG(\rho)^{-1}\Phi_i$, $i=1,2$. 
\end{definition}

We present two essential examples of metrics in probability space $\mcP(\Omega)$:  Fisher-Rao metric and Wasserstein metric. 

\begin{example}[Fisher-Rao metric]
The inverse of the Fisher-Rao metric tensor follows
$$
\mcG^F(\rho)^{-1}\Phi = \rho\lp\Phi -\int   \Phi\rho dx\rp, \quad \Phi\in T_\rho^*\mcP(\Omega).
$$
The Fisher-Rao metric is defined by
$$
g^F_\rho( \sigma_1,\sigma_2)=\int \Phi_1\Phi_2\rho dx-\lp\int \Phi_1\rho dx\rp\lp\int \Phi_2\rho dx\rp, \quad \sigma_1, \sigma_2\in T_\rho\mcP(\Omega),
$$
where $\Phi_i$ satisfies $\sigma_i=\rho\lp\Phi_i -\int   \Phi_i\rho dx\rp,\,i=1,2$. 
\end{example}
\begin{example}[Wasserstein metric]
The inverse of the Wasserstein metric tensor satisfies
$$
\mcG^W(\rho)^{-1}\Phi = -\nabla\cdot(\rho\nabla\Phi), \quad \Phi\in T_\rho^*\mcP(\Omega).
$$
The Wasserstein metric is given by
$$
g^W_\rho(\sigma_1,\sigma_2)=\int   \rho\la \nabla \Phi_1,\nabla \Phi_2\ra dx, \quad \sigma_1, \sigma_2\in T_\rho\mcP(\Omega),
$$
where $\Phi_i$ is the solution to $\sigma_i=-\nabla \cdot(\rho\nabla \Phi_i),\,i=1,2$. 
\end{example}
\subsection{Gradient operators}
The gradient operator for the objective functional $E(\rho)$ in $(\mcP(\Omega), \mcG(\rho))$ satisfies
\begin{equation*}
\grad E(\rho) = -\mcG(\rho)^{-1}\frac{\delta E}{\delta \rho}.
\end{equation*}
Here $\frac{\delta E}{\delta \rho}$ is the $L^2$ first variation w.r.t. $\rho$. The gradient flow follows
$$
\p_t\rho_t = -\grad E(\rho_t)=-\mcG(\rho)^{-1}\frac{\delta E}{\delta \rho_t}.
$$
We present gradient operators under either Fisher-Rao metric or Wasserstein metric.
\begin{example}[Fisher-Rao gradient operator]
The Fisher-Rao gradient operator satisfies
\begin{equation*}
\grad^F E(\rho)= \rho\lp \frac{\delta E}{\delta \rho}-\int    \frac{\delta E}{\delta \rho}\rho dx\rp.
\end{equation*}
\end{example}
\begin{example}[Wasserstein gradient operator]
The Wasserstein gradient operator writes
\begin{equation*}
\grad^W E(\rho) =-\nabla \cdot \lp\rho \nabla \frac{\delta E}{\delta \rho}\rp.
\end{equation*}
\end{example}

\section{Information Newton's flow}\label{section3}
In this section, we introduce and discuss properties of Hessian operators in probability space. Then, we 
formulate Newton's flows under information metrics. This is based on the previous definition of gradient operators and the inverse of Hessian operators.

\subsection{Information Hessian operators}
In this subsection, we review the definition of Hessian operators in probability space and provide the exact formulations of Hessian operators. 

For $\sigma\in T_\rho \mcP(\Omega)$, there exists a unique geodesic curve $\hat \rho_s$, which satisfies $\hat\rho_s|_{s=0} = \rho$ and $\hat\p_s\rho_s|_{s=0} = \sigma$. The Hessian operator of $E(\rho)$ w.r.t. metric tensor $\mcG(\rho)$ is a mapping $ \Hess E(\rho):T_\rho \mcP(\Omega) \to T_\rho \mcP(\Omega)$, which is defined by
\begin{equation*}
g_\rho(\Hess E(\rho) \sigma, \sigma) = g_\rho(\sigma, \Hess E(\rho) \sigma) = \left.\frac{d^2}{ds^2} E(\hat\rho_s)\right|_{s=0}.
\end{equation*}
Combining with the metric tensor, the Hessian operator uniquely defines a self-adjoint mapping $\mcH_E(\rho):T_\rho^*\mcP(\Omega)\to T_\rho \mcP(\Omega)$, which satisfies
\begin{equation*}
\int \Phi \mcH_E(\rho) \Phi dx = g_\rho(\sigma, \Hess E(\rho) \sigma),\quad \Phi = \mcG(\rho)\sigma.
\end{equation*}
In Proposition \ref{prop:php}, we give an exact formulation of $\int \Phi \mcH_E(\rho) \Phi dx$ and a relationship between $\mcH_E(\rho)$ and $\Hess E(\rho)$. 
\begin{proposition}\label{prop:php}
The quantity $g_\rho(\sigma, \Hess E(\rho) \sigma)$ is a bi-linear form of $\Phi$:
\begin{equation}\label{equ:php}
\begin{aligned}
\int \Phi \mcH_E(\rho) \Phi dx=&g_\rho(\sigma, \Hess E(\rho) \sigma)\\
=&-\frac{1}{2}\int \mcA(\rho)(\Phi,\Phi)\mcG(\rho)^{-1}\frac{\delta E}{\delta \rho} dx+\int   \mcA(\rho)\pp{\Phi,\frac{\delta E}{\delta \rho}} \mcG(\rho)^{-1}\Phi dx\\
 &+\int \int \pp{\mcG(\rho)^{-1}\Phi}(y)\frac{\delta^2 E}{\delta \rho^2}(x,y)dy \pp{\mcG(\rho)^{-1}\Phi}(x) dx.
\end{aligned}
\end{equation}
Here $\frac{\delta^2 E}{\delta \rho^2} (x,y)$ is defined by
$$
\frac{\delta^2 E}{\delta \rho^2} (x,y) = \frac{\delta}{\delta \rho}\pp{ \int \frac{\delta E}{\delta \rho}(y)\delta(x-y)dy},
$$
where $\delta(x)$ is the Dirac delta function. Here $\mcA(\rho):T_\rho^*\mcP(\Omega)\times T_\rho^*\mcP(\Omega)\to T^*_\rho\mcP(\Omega)$ is a bi-linear operator which satisfies
\begin{equation*}
   \mcA(\rho)(\Phi_1,\Phi_2) = \frac{\delta}{\delta \rho}\int \Phi_1\mcG(\rho)^{-1}\Phi_2 dx,\quad \forall \Phi_1,\Phi_2\in T_\rho^*\mcP(\Omega).
\end{equation*}
Moreover, the operator $\mcH_E(\rho)$ satisfies
\begin{equation}\label{equ:hehess}
    \mcH_E(\rho) = \Hess E(\rho)\mcG(\rho)^{-1}.
\end{equation}
\end{proposition}

Now, we are ready to present the information Newton's flow in probability space.
\begin{proposition}[Information Newton's flow]The Newton's flow of $E(\rho)$ in $(\mcP(\Omega), \mcG(\rho))$ satisfies
\begin{equation*}
\p_t\rho_t+(\Hess E(\rho_t))^{-1} \mcG(\rho_t)^{-1}\frac{\delta E}{\delta{\rho_t}} = 0.
\end{equation*}
This is equivalent to
\begin{equation}
\lbb{
&\p_t\rho_t-\mcG(\rho_t)^{-1}\Phi_t = 0,\\
&\mcH_E(\rho_t)\Phi_t + \mcG(\rho_t)^{-1}\frac{\delta}{\delta{\rho_t}}E(\rho_t)=0.
}
\end{equation}
\end{proposition}

In particular, we focus on Wasserstein Newton's flow of KL divergence. Other examples of Newton's flows of different objective functions under either Fisher-Rao metric or Wasserstein metric are presented in Appendix \ref{ssec:fnf} and \ref{ssec:wnf}.

\begin{example}[Wasserstein Newton's flow of KL divergence]
\label{eg:wnf}
In this example we prove Theorem \ref{thm:wnf_kl}. As a known fact in \citep{goaib} and Gamma calculus \citep{diff_hyper, gopsv}, the Hessian operator of KL divergence under the Wasserstein metric follows
\begin{equation*}
g^W_{\rho}(\sigma,\Hess^W E(\rho)\sigma) = \int \Big(\|\nabla^2\Phi\|_F^2+(\nabla \Phi)^T\nabla^2 f\nabla \Phi)\Big)\rho dx,
\end{equation*}
where $\sigma = -\nabla\cdot(\rho\nabla \Phi)$ and $\|\cdot\|_F$ is the Frobenius norm of a matrix in $\mbR^{n\times n}$.
Via integration by parts, we validate that the operator $\mcH_E^W(\rho)$ follows
\begin{equation}\label{equ:he_kl}
\mcH^W_E(\rho)\Phi = \nabla^2:(\rho\nabla^2\Phi)-\nabla\cdot(\rho\nabla^2f\nabla \Phi).
\end{equation}
\end{example}

We also present the Wasserstein Newton's flow of KL divergence in Gaussian families. Proposition \ref{prop:exist} ensures the existence of information Newton's flows in Gaussian families.
\begin{proposition}\label{prop:exist}
Suppose that $\rho_0,\rho^*$ are Gaussian distributions with zero means and their covariance matrices are $\Sigma_0$ and $\Sigma^*$. $E(\Sigma)$ evaluates the KL divergence from $\rho$ to $\rho^*$:
\begin{equation}\label{esigma}
    E(\Sigma) = \frac{1}{2}\pp{\tr(\Sigma (\Sigma^*)^{-1})-d-\log\det\pp{\Sigma (\Sigma^*)^{-1}}}.
\end{equation}
Let $(\Sigma_t,S_t)$ satisfy
\begin{equation}
\left\{  \begin{aligned}
    &\dot \Sigma_t -2(S\Sigma_t+\Sigma S_t)=0,\\
    &2\Sigma_t S_t(\Sigma^*)^{-1}+2(\Sigma^*)^{-1}S_t\Sigma_t+4S_t = -(\Sigma_t (\Sigma^*)^{-1}+(\Sigma^*)^{-1}\Sigma_t-2I).
    \end{aligned}\right.
\end{equation}
with initial values $\Sigma_t|_{t=0}=\Sigma_0$ and $S_t|_{t=0}=0$. Thus, for any $t\geq 0$, $\Sigma_t$ is well-defined and stays positive definite. We denote
$$
\rho_t(x)=\frac{(2\pi)^{-n/2}}{\sqrt{\det(\Sigma_t)}}\exp\lp-\frac{1}{2}x^T\Sigma_t^{-1}x\rp, \quad \Phi_t(x)=x^TS_tx+C(t),
$$
where $C(t)=-t+\frac{1}{2}\int_0^t\log\det(\Sigma_s(\Sigma^*)^{-1})ds$. Then, $\rho_t$ and $\Phi_t$ follow the information Newton's flow \eqref{equ:wnf_kl} with initial values $\rho_t|_{t=0}=\rho_0$ and $\Phi_t|_{t=0}=0$.
\end{proposition}


\section{Newton's Langevin dynamics}\label{sec:nld}
In this section, we primarily focus on 
the Wasserstein Newton's flow of KL divergence. 
We formulate it into the Newton's Langevin dynamics for Bayesian sampling problems. The connection and difference with

Let the objective functional $E(\rho)=\mathrm{D}_{\textrm{KL}}(\rho\|\rho^*)$ evaluate the KL divergence from $\rho$ to a target density $\rho^*(x)\propto \exp(-f(x))$ with $\int \exp(-f(x)) dx<\infty$. This specific optimization problem is important since it corresponds to sampling from the target density $\rho^*$. Classical Langevin MCMC algorithms evolves samples following 
overdamped Langevin dynamics (OLD), which satisfies
\begin{equation*}
dX_t = -\nabla f(X_t)dt+\sqrt{2}dB_t,
\end{equation*}
where $B_t$ is the standard Brownian motion. Denote $\rho_t$ as the density function of the distribution of $X_t$. The evolution of $\rho_t$ satisfies the Fokker-Planck equation 
\begin{equation*}
\partial_t\rho_t=\nabla\cdot(\rho_t \nabla f)+\Delta\rho_t.     
\end{equation*}
A known fact is that the Fokker-Planck equation is the Wasserstein gradient flow (WGF) of KL divergence, i.e. 
\begin{equation}\label{equ:wgf}
\begin{split}
    \p_t\rho_t=&-\grad^W\mathrm{D}_{\textrm{KL}}(\rho_t\|\rho^*)\\
    =&\mcG^W(\rho_t)^{-1}\frac{\delta}{\delta\rho_t} \mathrm{D}_{\textrm{KL}}(\rho_t\|\rho^*)\\
    =& \nabla \cdot(\rho_t \nabla (f+\log\rho_t+1))\\
    =&\nabla\cdot(\rho_t\nabla f)+\Delta\rho_t.    
    \end{split}
\end{equation}
where we use the fact that $\frac{\delta}{\delta\rho}\mathrm{D}_{\textrm{KL}}(\rho_t\|\rho^*)=\log\rho+t+f+1$ and $\rho\nabla\log\rho=\nabla\rho$.

It is worth mentioning that OLD can be viewed as particle implementations of WGF \eqref{equ:wgf}. From the viewpoint of fluid dynamics, WGF also has a Lagrangian formulation
$$
dX_t = -\nabla f(X_t)dt-\nabla \log \rho_t(X_t)dt.
$$
We name above dynamics by the Lagrangian Langevin Dynamics (LLD). Here `Lagrangian' refers to the Lagrangian coordinates (flow map) in fluid dynamics \citep{otoan}.

Overall, many sampling algorithms follow OLD or LLD. 
The evolution of corresponding density follows the Wasserstein gradient flow \eqref{equ:wgf}. E.g. the classical Langevin MCMC (unadjusted Langevin algorithm) is the time discretization of OLD. The Particle-based Variational Inference methods (ParVI), \citep{uaapb} can be viewed as the discrete-time approximation of LLD. 

In short, we notice that the Langevein dynamics can be viewed as first-order methods for Bayesian sampling problems. Analogously, the Wasserstein Newton's flow of KL divergence derived in Example \ref{eg:wnf} corresponds to certain Langevin dynamics of particle systems, named {\em Newton's Langevin dynamics.}
\begin{theorem}
Consider the Newton's Langevin dynamics
\begin{equation}
    dX_t= \nabla \Phi_t^{\operatorname{Newton}}(X_t)dt,
\end{equation}
where $\Phi_t^{\operatorname{Newton}}(x)$ is the solution to Wasserstein Newton's direction equation \eqref{equ:wnf_kl}:
\begin{equation*}
    \nabla^2:(\rho_t\nabla^2\Phi_t)-\nabla\cdot(\rho_t\nabla^2f\nabla \Phi_t) - \nabla\cdot(\rho_t\nabla f)-\Delta \rho_t=0.
\end{equation*}
Here $X_0$ follows an initial distribution $\rho^0$ and $\rho_t$ is the distribution of $X_t$. Then, $\rho_t$ is the solution to Wasserstein Newton's flow with an initial value $\rho_0=\rho^0$. 
\end{theorem}
\begin{proof}
Note that $\rho_t$ is the distribution of $X_t$. The dynamics of $X_t$ implies
$$
\p_t\rho_t+\nabla\cdot(\rho_t\nabla \Phi_t^{\operatorname{Newton}})=0.
$$
Because $\Phi_t$ satisfies the Wasserstein Newton's direction equation \eqref{equ:wnf_kl}, $\rho_t$ is the solution to Wasserstein Newton's flow.
\end{proof}
\begin{remark}
We notice that the Newton's Langevien dynamics is different from HAMCMC \citep{sqnmc}. Detailed comparisons can be found in Appendix \ref{ssec:cmp_hamcmc}. 
\end{remark}

The following proposition provide a closed-form formula for NLD in 1D Gaussian family.
\begin{proposition}\label{prop:gauss1d}
Assume that $f(x)=(2\Sigma^*)^{-1}(x-\mu^*)^2$, where $\Sigma^*>0$ and $\mu^*$ are given. Suppose that the particle system $X_0$ follows the Gaussian distribution. Then $X_t$ follows a Gaussian distribution with mean $\mu_t$ and variance $\Sigma_t$. The corresponding NLD satisfies 
$$
dX_t = \pp{\frac{\Sigma^*-\Sigma}{\Sigma^*+\Sigma_t}X_t-\frac{2\Sigma^*}{\Sigma^*+\Sigma_t}\mu_t+\mu^*}dt.
$$
And the evolution of $\mu_t$ and $\Sigma_t$ satisfies
$$
d\mu_t =(-\mu_t+\mu^*)dt,\quad d\Sigma_t =2\frac{\Sigma^*-\Sigma_t}{\Sigma^*+\Sigma_t}\Sigma_tdt.
$$
The explicit solutions of $\mu_t$ and $\Sigma_t$ satisfy 
$$
\mu_t=e^{-t}(\mu_0-\mu^*)+\mu^*,\quad \Sigma_t=\Sigma^*+(\Sigma_0-\Sigma^*)e^{-t}\sqrt{\frac{e^{-2t}(\Sigma_0-\Sigma^*)^2}{4\Sigma_0^2}+\frac{1}{\Sigma_0\Sigma^*}}.
$$
\end{proposition}
We present discrete-time particle implementations of Newton's Langevin dynamics in section \ref{sec:par} and numerical examples in section \ref{sec:num}. 

\section{Particle implementation of Wasserstein Newton's method}\label{sec:par}
In this section, 
we design sampling efficient implementations of Wasserstein Newton's meth- od. Focusing on Wasserstein Newton's flow of KL divergence, we introduce a variational formulation for computing the Wasserstein Newton's direction. By restricting the domain of the variational problem in a linear subspace or reproducing kernel Hilbert space (RKHS), we derive sampling efficient algorithms. 
Besides, a hybrid method between Newton's Langevin dynamics and overdamped Langevin dynamics is provided.

We briefly review update rules of Newton's methods and hybrid methods in Euclidean space. In each iteration, the update rule of Newton's method follows
$$
x_{k+1}=x_k+\alpha_kp_k,\quad p_k=-\nabla^2f(x_k)^{-1}\nabla f(x).
$$
Suppose that $f(x)$ is strictly convex. Namely, $\nabla^2f(x)$ is positive definite for all $x\in\mbR^d$. To compute the Newton's direction $p_k$, it is equivalent to solve the following variational problem
$$
\min_{p\in\mathbb{R}^n}~p^T\nabla^2f(x_k)p+2\nabla f(x_k)^Tp.
$$
In practice, the Newton's direction may not lead to the decrease in the objective function, especially when $f(x)$ is non-convex. Nevertheless, the Newton's method often converges when the update is close to the minimizer. One way to overcome this problem is the hybrid method. Consider a hybrid update of the Newton's direction and the gradient's direction
$$
x_{k+1}=x_k+\alpha_kp_k-\alpha_k\gamma \nabla f(x_k),
$$
where $\gamma>0$ is a parameter. 

Following above ideas in Euclidean space, we present a particle implementation of information Newton's method. Here we connect density $\rho_k\in \mcP(\Omega)$ with a particle system $\{x_k^n\}_{i=1}^N$. Namely, we assume that the distribution $\{x_k^n\}_{n=1}^N$ follows $\rho_k(x)$. 
We update each particle by 
$$
x_{k+1}^n= x_k^n+\alpha_k\nabla\hat \Phi_k(x_k^n),\quad i=1,2\dots N.
$$
Here $\hat \Phi_k$ is an approximated solution to the Wasserstein Newton's direction equation \eqref{equ:wnf_kl}. The details on obtaining $\hat \Phi_k$ is left in subsection \ref{ssec:var}. 

In practice, the Wasserstein Newton's direction may not be a descent direction if the update is far away from the target distribution. To overcome this issue, we propose a hybrid update of the Wasserstein Newton's direction and the Wasserstein gradient direction. 

Let $\gamma\geq0$ be a parameter. Here we recall that there are two choices for using the gradient direction. Namely, if we use overdamped Langevin dynamics as the gradient direction, the hybrid update rule follows
\begin{equation}\label{equ:hybrid1}
    x_{k+1}^n = x_k^n+\alpha_k\nabla\hat\Phi_k(x_k^n)-\gamma\alpha_k\nabla f(x_k^n)+\sqrt{2\gamma\alpha_k}z_k,
\end{equation}
where $z_k\sim\mcN(0,I)$. If we use Lagrangian Langevin dynamics as the gradient direction, the hybrid update rule satisfies
\begin{equation}\label{equ:hybrid2}
    x_{k+1}^n = x_k^n+\alpha_k \nabla\hat\Phi_k(x_k^n)-\gamma\alpha_k(\nabla f(x_k^n)+\xi_k(x_k^n)).
\end{equation}
Here $\xi_k$ is an approximation of $\nabla \log \rho_k$. 
For general $\rho_k$ and $\rho^*$, we can approximate $\nabla\log \rho_k$ via kernel density estimation (KDE) \citep{aktst}. Namely, we approximate $\nabla\log \rho_k$ by
$$
\xi_k(x) =\frac{\sum_{n=1}^N\nabla_y k(x,x_k^n)}{\sum_{n=1}^N\nabla k(x,x_k^n)}.
$$
Here $k(x,y)$ is a given positive kernel. A typical choice of $k(x,y)$ is a Gaussian kernel with a bandwidth $h>0$, such that
$$
k(x,y) = (2\pi h)^{-n/2}\exp\pp{-\frac{\|x-y\|^2}{2h}}.
$$

The overall algorithm is summarized in Algorithm \ref{alg:wnewton}.
\begin{algorithm}[ht]
\caption{Wasserstein Newton's method with hybrid update}\label{alg:wnewton}
\begin{algorithmic}[1]
\REQUIRE initial positions $\{x_0^n\}_{n=1}^N$, $\epsilon\geq 0$, step sizes $\alpha_k$, parameters $\lambda_k\geq 0$, maximum iteration $K$.
\STATE Set $k=0$.
\WHILE{$k< K$ and the convergence criterion is not met}
\STATE Compute an approximate solution $\Phi_k$ to \eqref{equ:wnf_kl}.
\STATE Update particle positions by \eqref{equ:hybrid1} or \eqref{equ:hybrid2}.
\STATE Set $k=k+1$.
\ENDWHILE
\end{algorithmic}
\end{algorithm}

\begin{remark}
It worths mentioning that our algorithm corresponds to the following hybrid Langvien dynamics
\begin{equation*}
dX_t=(\nabla\Phi_t-\gamma\nabla f)dt+ \sqrt{2\gamma}dB_t,
\end{equation*}
where $B_t$ is the standrad Brownian motion, $\gamma\geq 0$ is a parameter and $\Phi_t$ satisfies \eqref{equ:wnf_kl}.
\end{remark}

\subsection{Variational formulation for Wasserstein Newton's direction}
\label{ssec:var}
Similar to the Euclidean case, we derive a variational formulation for estimating Wasserstein Newton's direction, and provide the associated particle formulations.
\begin{proposition}\label{prop:var}
Suppose that $\mcH: T^*_{\rho}\mcP(\Omega)\to T_\rho \mcP$ is a linear self-adjoint operator and $\mcH$ is positive definite. Let $u\in T_\rho \mcP$. Then the minimizer of variational problem
$$
\min_{\Phi\in T^*_{\rho}\mcP(\Omega)}J(\Phi)=\int \pp{\Phi\mcH\Phi-2u\Phi} dx,
$$
satisfies $\mcH\Phi=u$, where $\Phi\in T^*_{\rho}\mcP(\Omega)$. 
\end{proposition}
\begin{proof}
Since $\mcH$ is linear and self-adjoint, the optimal solution of satisfies
$$
0 = \frac{\delta J}{\delta \Phi}=2\mcH\Phi -2u. 
$$
Hence, $\Phi$ satisfies $\mcH\Phi=u$. On the other hand, let $\Phi$ satisfy $\mcH\Phi=u$. Then, for any $\Psi\in  T^*_{\rho}\mcP(\Omega)$, it follows
$$
\begin{aligned}
&J(\Phi+\Psi) = \int \pp{(\Phi+\Psi)\mcH(\Phi+\Psi)-2u(\Phi+\Psi)} dx\\
=&\int \pp{\Phi\mcH\Phi-2u\Phi} dx+\int \pp{\Psi\mcH\Psi-2u\Psi-2\Psi\mcH\Phi} dx\\
=&J(\Phi)+\int \Psi\mcH\Psi dx\geq J(\Phi).
\end{aligned}
$$
The last inequality is based on the fact that $\mcH$ is positive definite. Hence, $\Phi$ is the optimal solution to the proposed variational problem. This completes the proof. 
\end{proof}

Suppose that $f$ is strongly convex, or equivalent, $\nabla^2 f(x)$ is positive definite for $x\in \Omega$. Then, the operator $\mcH_E(\rho)$ defined in \eqref{equ:he_kl} is positive definite. In this case, proposition \ref{prop:var} indicates that solving Wasserstein Newton's direction equation \eqref{equ:wnf_kl} is equivalent to optimizing the following variational problem. 
\begin{equation*}
    \min_{\Phi\in T^*_{\rho_k}\mcP(\Omega)}J(\Phi)=\int \pp{\|\nabla^2\Phi\|_F^2+\|\nabla \Phi\|^2_{\nabla^2 f}+2\lra{\nabla f+\nabla\log\rho_k, \nabla \Phi}}\rho_k dx.
\end{equation*}
Here we denote $\|v\|_A^2 = v^TAv$. For possibly non-convex $f$, we consider a regularized problem
\begin{equation}\label{equ:opt_var_reg}
    \min_{\Phi\in T^*_{\rho_k}\mcP(\Omega)}J^\epsilon(\Phi)=\int \pp{\|\nabla^2\Phi\|_F^2+\|\nabla \Phi\|^2_{\nabla^2 f+\epsilon I}+2\lra{\nabla f+\nabla\log\rho_k, \nabla \Phi}}\rho_k dx.
\end{equation}
Here $\epsilon\geq 0$ is a regularization parameter to ensure that $\nabla^2f(x)+\epsilon I$ is positive definite for $x\in \Omega$.

\begin{remark}
Namely, we penalize the objective function by adding the squared norm of $\Phi$ induced by the Wasserstein metric. In other words, 
$$
\min_{\Phi\in T^*_{\rho_k}\mcP(\Omega)}J(\Phi)+\epsilon \int \|\nabla\Phi\|^2\rho_k dx.
$$
\end{remark}
In terms of samples, we can rewrite \eqref{equ:opt_var_reg} into
\begin{equation}\label{equ:opt_var_reg_smp}
\begin{aligned}
\min_{\Phi\in T^*_{\rho_k}\mcP(\Omega)}J^\epsilon(\Phi)=\frac{1}{N}\sum_{n=1}^N\Big(&\|\nabla^2\Phi(x_k^n)\|_F^2+\|\nabla \Phi(x_k^n)\|^2_{\nabla^2 f(x_k^n)+\epsilon I}\\
&+2\lra{\nabla f(x_k^n)+\nabla\log\rho_k(x_k^n), \nabla \Phi(x_k^n)}\Big).
\end{aligned}
\end{equation}
In high dimensional sample space, directly solving \eqref{equ:opt_var_reg_smp} for $\Phi\in T_{\rho_k}^*\mcP(\Omega)$ can be difficult. To deal with this issue, we restrict the functional space of $\Phi$ into a linear subspace $\mcS\subseteq T^*_{\rho_k}\mcP(\Omega)$. An appropriately chosen  $\mcS$ can lead to a closed-form solution to \eqref{equ:opt_var_reg}. For the rest of this section, we discuss two choices of $\mcS$, including finite dimensional affine subspace and reproducing kernel Hilbert space (RKHS).

\subsection{Affine models}
Consider $\mcS=\operatorname{span}\{\psi_i\}_{i=1}^m$, where $\psi_i:\Omega\to R$ are given basis functions. Namely, we assume that $\Phi(x)$ is a linear combination of $\psi_1,\dots \psi_m$, such that
$$
\Phi(x) = \lra{\mfa,\psi(x)}=\sum_{i=1}^ma_i\psi_i(x),
$$
where $\mfa\in\mbR^m$ and $\psi(x)=[\psi_1(x), \psi_2(x),\dots \psi_m(x)]$. 

\begin{proposition}
Suppose that $\Phi(x) = \lra{\mfa,\psi(x)}$. Then, the optimization problem \eqref{equ:opt_var_reg_smp} with the constraint $\Phi\in \mcS$ is equivalent to 
$$
\min_{\mfa\in\mbR^m}J^\epsilon(\mfa)=\mfa^T(\mfB_k+\mfD_k)\mfa+2\mfc_k^T\mfa,
$$
where $\mfB_{k},\mfD_k\in \mbR^{m\times m}$ and $\mfc_{k}\in \mbR^m$. The detailed formulations of $\mfB_{k},
\mfD_k$ and $\mfc_k$ are provided as follows. 
$$
\begin{aligned}
\mfB_{k} = &\frac{1}{N}\sum_{n=1}^N\nabla\psi(x_k^n)(\nabla^2f(x_k^n)+\epsilon I)(\nabla\psi(x_k^n))^T,\\
\mfD(x)_{j_1,j_2} = &\frac{1}{N}\sum_{n=1}^N\tr (\nabla^2\psi_{j_1}(x_k^n)\nabla^2\psi_{j_2}(x_k^n)),\\
\mfc(x)= &\frac{1}{N}\sum_{n=1}^N\nabla \psi(x_k^n)(\nabla f(x_k^n)+\xi_k(x_k^n)).
\end{aligned}
$$
If $\mfB_k+\mfD_k$ is positive definite, the optimal solution follows $\mfa = -(\mfB_k+\mfD_k)^{-1}c_k$. The optimal solution $\hat \Phi$ follows $\hat \Phi(x) = \lra{ \mfa, \psi(x)}.$
\end{proposition}
\begin{proof}
We denote the Jacobian $\nabla\psi(x)\in \mbR^{n\times m}$. As a result, $J(\mfa)$ turns to be
$$
J^\epsilon(\mfa) = \left\{\frac{1}{N}\sum_{n=1}^N\left\|\sum_{j=1}^ma_j\nabla^2\psi_j(x_k^n)\right\|_F^2+\mfa^T\mfB(x
_k^n)\mfa+2\mfa^T \mfc(x_k^n)\right\}.
$$
We can further compute that
\begin{equation*}
\left\|\sum_{j=1}^m\mfa_j\nabla^2\psi_j(x_k^n)\right\|_F^2=\sum_{j_1=1}^m\sum_{j_2=1}^m \mfa_{j_1}\nabla^2\psi_{j_1}(x_k^n)\nabla^2\psi_{j_2}(x_k^n)\mfa_{j_2}= \mfa^T\mfD(x_k^n)\mfa.
\end{equation*}
This completes the proof.
\end{proof}


This affine approximation technique has been used in approximating natural gradient direction in \citep{anpl}. Hence, we call our method affine information Newton's method.

In particular, we set $m=2d$ and consider the basis
$$
\psi_i(x)=x_i,\quad \psi_{i+d}(x)=x_i^2,\quad 1\leq i\leq d.
$$
In other words, we assume that $\Phi(x)$ takes the form $\Phi(x) = \frac{1}{2}x\diag(s)x+b^Tx$, where $s,b\in \mbR^d$. For simplicity, we denote $v_k^n=\nabla f(x_k^n)+\xi_k(x_k^n)$. 

$$
\begin{aligned}
&J^\epsilon(s,b) 
=&\bmbm{s\\b}^T\mfH_k\bmbm{s\\b}+2\bmbm{s\\b}^Tu_k.
\end{aligned}
$$
where we denote $\mfH_k\in \mbR^{2d\times 2d}$ via
$$
\mfH_k=\bmbm{I+\frac{1}{N}\sum_{n=1}^N\diag(x_k^n)(\nabla^2 f(x_k^n)+\epsilon I)\diag(x_k^n) &\frac{1}{N}\sum_{n=1}^N\diag(x_k^n)(\nabla^2 f(x_k^n)+\epsilon I)\\ \frac{1}{N}\sum_{n=1}^N(\nabla^2 f(x_k^n)+\epsilon I)\diag(x_k^n)&\frac{1}{N}\sum_{n=1}^N(\nabla^2 f(x_k^n)+\epsilon I)  },
$$
and $u_k\in \mbR^{2d}$ via
$$
u_k=\bmbm{\frac{1}{N}\sum_{n=1}^N\diag(x_k^n)v_k^n\\\frac{1}{N}\sum_{n=1}^Nv_k^n}.
$$
Hence, the optimal solution for minimizing $J(s,b)$ follows 
$$
\bmbm{s_k\\b_k} = -(\mfH_k)^{-1}u_k.
$$
Hence, the approximate solution $\hat \Phi_k$ computed via the affine method follows
\begin{equation}
\nabla \hat\Phi_k(x) = \diag(s_k) x+b_k.
\end{equation}

The overall algorithm are summarized in Algorithm \ref{alg:aff_newton}. For simplicity, we do not mention the hybrid update.
\begin{algorithm}[ht]
\caption{Wasserstein Newton's method with affine models. }\label{alg:aff_newton}
\begin{algorithmic}[1]
\REQUIRE initial positions $\{x_0^i\}_{i=1}^N$, $\epsilon\geq 0$, step sizes $\alpha_k$, maximum iteration $K$.
\STATE Set $k=0$.
\WHILE{$k< K$ and the convergence criterion is not met}
\STATE Compute $v_k^n=\nabla f(x_k^n)+\xi_k(x_k^n)$. Here $\xi_k$ is an approximation of $\nabla \log \rho_k$. 
\STATE Calculate $\mfH_k$ by
{\footnotesize
$$
\mfH_k=\bmbm{I+\frac{1}{N}\sum_{n=1}^N\diag(x_k^n)(\nabla^2 f(x_k^n)+\epsilon I)\diag(x_k^n) &\frac{1}{N}\sum_{n=1}^N\diag(x_k^n)(\nabla^2 f(x_k^n)+\epsilon I)\\ \frac{1}{N}\sum_{n=1}^N(\nabla^2 f(x_k^n)+\epsilon I)\diag(x_k^n)&\frac{1}{N}\sum_{n=1}^N(\nabla^2 f(x_k^n)+\epsilon I)  },
$$
}
and formulate $u_k$ by
$$
u_k=\bmbm{\frac{1}{N}\sum_{n=1}^N\diag(x_k^n)v_k^n\\\frac{1}{N}\sum_{n=1}^Nv_k^n}.
$$
\STATE Compute $s_k$ and $b_k$ by
$$
\bmbm{s_k\\b_k} = -(\mfH_k)^{-1}u_k.
$$
\STATE Update particle positions by $$x_{k+1}^n = x_k^n+\alpha_k (\diag(s_k) x_k^n+b_k).$$
\STATE Set $k=k+1$.
\ENDWHILE
\end{algorithmic}
\end{algorithm}

When the optimal solution $\Phi_k$ to \eqref{equ:opt_var_reg_smp} is highly non-linear, $\mcS$ in affine methods may not be large enough to approximate $\Phi_k$ well.

\subsection{Kernel models}
In this subsection, we approximate the Wasserstein Newton's direction in kernel models. Specifically, we consider $\mcS$ as the RKHS with an associated kernel function $k(x,y):\mbR^d\times \mbR^d\to \mbR$. Compared to finite-dimensional linear subspace, RKHS can be viewed as with infinitely many feature functions. Detailed description about RKHS and the related norm can be found in \citep{rkhsi}.

To ensure the well-posedness of the optimal solution, we penalize the objective function using the RKHS norm $\|\cdot\|_\mcS$. Hence, we consider a regularized variational problem based on \eqref{equ:opt_var_reg}
\begin{equation}\label{equ:wnewton_k}
\begin{aligned}
\min_{\Phi\in \mcH}&\int \pp{\|\nabla^2\Phi\|_F^2+\|\nabla \Phi\|^2_{\nabla^2 f+\epsilon I}+2\lra{\nabla f, \nabla \Phi+\nabla \log \rho_k}}\rho_k dx+\lambda \|\Phi\|_{\mcS}^2\\
=&\int \pp{\|\nabla^2\Phi\|_F^2+\|\nabla \Phi\|^2_{\nabla^2 f+\epsilon I}+2\lra{\nabla f, \nabla \Phi}-2\Delta \Phi}\rho_k dx+\lambda \|\Phi\|_{\mcS}^2.
\end{aligned}
\end{equation}
In terms of samples, this varitional problem becomes
\begin{equation}\label{equ:wnewton_smp}
\begin{aligned}
\min_{\Phi\in \mcH}\frac{1}{N}\sum_{n=1}^N\Big(&\|\nabla^2\Phi(x_k^n)\|_F^2+\|\nabla \Phi(x_k^n)\|^2_{\nabla^2 f(x_k^n)+\epsilon I}\\
&+2\lra{\nabla f(x_k^n), \nabla \Phi(x_k^n)}-2\Delta \Phi(x_k^n)\Big)+\lambda \|\Phi\|_{\mcS}^2.
\end{aligned}
\end{equation}
From the general representation theorem \citep{agrt}, the minimizer of \eqref{equ:wnewton_smp} can take the form
\begin{equation}\label{kernel_phi}
\Phi(x)=\sum_{n=1}^N\pp{\sum_{i=1}^d\alpha_{i,n} \p_i k(x_k^n,x) +\sum_{j_1=1}^d\sum_{j_2=1}^d\beta_{j_1,j_2,n}\p_{j_1,j_2} k(x_k^n,x) }.
\end{equation}
\begin{proposition}\label{prop:ker}
Let $\Phi$ take the form \eqref{kernel_phi}. Then, \eqref{equ:wnewton_smp} is equivalent to
\begin{equation}\label{equ:kwnewton}
\begin{aligned}
\inf_{\alpha\in \mbR^{Nd},\beta\in \mbR^{Nd^2}} &\bmbm{\alpha\\\beta}^T\bmbm{K^{1,2}\\K^{2,2}}\bmbm{K^{1,2}\\K^{2,2}}^T\bmbm{\alpha\\\beta}+\bmbm{\alpha\\\beta}^T\bmbm{K^{1,1}\\K^{2,1}}H\bmbm{K^{1,1}\\K^{2,1}}^T\bmbm{\alpha\\\beta}\\
&+N\lambda \bmbm{\alpha\\\beta}^T\bmbm{K^{1,1}&K^{1,2}\\K^{2,1}&K^{2,2}}\bmbm{\alpha\\\beta}-2\bmbm{\alpha\\\beta}^T\bmbm{K^{1,1}&K^{1,2}\\K^{2,1}&K^{2,2}}\bmbm{v\\e}.
\end{aligned}
\end{equation}
Here we denote 
$$
v = \bmbm{-\nabla f(x_k^1)\\\vdots\\-\nabla f(x_k^N)}\in \mbR^{Nd},\quad e = \bmbm{\text{vec}(I_d)\\\vdots\\\text{vec}(I_d)}\in \mbR^{Nd^2}, 
$$
$$
H = \bmbm{\nabla^2 f(x_k^1)+\epsilon I&0&\dots&0\\
0&\nabla^2 f(x_k^2)+\epsilon I&\ddots&\vdots\\
\vdots&\ddots&\ddots&0\\
0&\dots&0&\nabla^2 f(x_k^N)+\epsilon I}\in \mbR^{Nd\times Nd},
$$
and
$$
K^{p,q} = \bmbm{K^{p,q}_{1,1}&\dots&K^{p,q}_{1,N}\\
\vdots&\ddots&\vdots\\
K^{p,q}_{N,1}&\dots&K^{p,q}_{N,N}},\quad p,q\in\{1,2\}.
$$
Each $K_{n,n'}^{p,q}$ are defined by
$$
\begin{aligned}
&\pp{K^{1,1}_{n,n'}}_{i,j} = \p_{i,j+d}k(x_k^n,x_k^{n'}),\quad K^{1,1}_{n,n'}\in \mbR^{d\times d},\\
&\pp{K^{1,2}_{n,n'}}_{i,(j_1-1)d+j_2} = \p_{i,j_1+d,j_2+d} k(x_k^n,x_k^{n'}),\quad K^{1,2}_{n,n'}\in \mbR^{d\times d^2},\\
&\pp{K^{2,1}_{n,n'}}_{(j_1-1)d+j_2,i} = \p_{j_1,j_2,i+d} k(x_k^n,x_k^{n'}),\quad K^{2,1}_{n,n'}\in \mbR^{d^2\times d},\\
&\pp{K^{2,2}_{n,n'}}_{(i_1-1)d+i_2,(j_1-1)d+j_2} = \p_{i_1,i_2,j_1+d,j_2+d} k(x_k^n,x_k^{n'}),\quad K^{2,2}_{n,n'}\in \mbR^{d^2\times d^2}.
\end{aligned}
$$
Here we use the notation $\p_{i}k(x,y)=\p_{x_i} k(x,y)$ and $\p_{j+d}=\p_{y_j}k(x,y)$.
The optimal solution follows
$$
\bmbm{\alpha\\\beta} = \pp{\bmbm{K^{1,2}\\K^{2,2}}\bmbm{K^{1,2}\\K^{2,2}}^T+\bmbm{K^{1,1}\\K^{2,1}}H\bmbm{K^{1,1}\\K^{2,1}}^T+N\lambda \bmbm{\alpha\\\beta}^T\bmbm{K^{1,1}&K^{1,2}\\K^{2,1}&K^{2,2}}}^\dagger \bmbm{K^{1,1}&K^{1,2}\\K^{2,1}&K^{2,2}}\bmbm{v\\e}.
$$
Here $\dagger$ denotes the Moore pseudo-inverse. Hence the approximated solution $\hat \Phi_k$ satisfies 
$$
\bmbm{\nabla \hat \Phi_k(x_k^1)\\\vdots\\\nabla \hat \Phi_k(x_k^N)}=K^{1,1}\alpha+K^{1,2}\beta.
$$
\end{proposition}

To solve \eqref{equ:kwnewton} is equivalent to solve a $N(d+d^2)\times N(d+d^2)$ linear system. Moreover, this linear system is potentially to be ill-posed, especially for large $N$ and $d$. Hence, we further restrict $\beta=0$ in \eqref{equ:kwnewton} (this is equivalent to choose a smaller basis in representing $\Phi(x)$). Then,  \eqref{equ:kwnewton} reduces to
\begin{equation}\label{equ:kwnewton_reduce}
\begin{aligned}
\inf_{\alpha\in \mbR^{Nd}} \quad \alpha^T K^{1,2}K^{2,1}\alpha+\alpha^TK^{1,1}HK^{1,1}\alpha+N\lambda \alpha^TK^{1,1}\alpha-2\alpha^T\bmbm{K^{1,1}&K^{1,2}}\bmbm{v\\e}.
\end{aligned}
\end{equation}
The optimal solution follows
$$
\alpha = (K^{1,2}K^{2,1}+K^{1,1}HK^{1,1}+N\lambda K^{1,1})^{-1} \bmbm{K^{1,1}&K^{1,2}}\bmbm{v\\e}.
$$
Denote $\mfC=K^{1,2}K^{2,1}+K^{1,1}HK^{1,1}+N\lambda K^{1,1}$. Hence, the approximate solution $\hat \Phi_k(x_k^n)$ satisfies
\begin{equation}\label{wnf:kernel}
\begin{aligned}
\bmbm{\nabla \hat \Phi_k(x_k^1)\\\vdots\\\nabla \hat \Phi_k(x_k^N)}=K^{1,1}\alpha=K^{1,1} \mfC^{-1}(K^{1,1}v+K^{1,2}e).
\end{aligned}
\end{equation}

In practice, when $N,d$ are large, the computation cost of $K^{1,2}K^{2,1}$ is quite heavy, which is of order $O(N^3d^4)$. Hence, we consider a block-diagonal approximation $\mfC_\text{bd}$ of $\mfC$, which is defined by
$$
\mfC_\text{bd} = \bmbm{C_{1,1}&0&\dots&0\\
0&C_{2,2}&\ddots&\vdots\\
\vdots&\ddots&\ddots&0\\
0&\dots&0&C_{N,N}}.
$$
Here each block $C_{i,i}\in \mbR^{d\times d}$ can be computed by
$$
C_{i,i} =N\lambda K^{1,1}_{i,i}+ \sum_{j=1}^N \pp{K^{1,2}_{i,j} K^{2,1}_{j,i}+ K^{1,1}_{i,j}\nabla^2 f(x_k^j)K^{1,1}_{j,i} }.
$$
The computational cost of $\mfC_\text{bd}$ is $O(N^2d^4)$. We also note that for Gaussian kernel, with $\lambda>0$, $C_{i,i}$ is invertible. Hence, we can compute the approximate solution $\hat \Phi_k(x_k^n)$ by
\begin{equation}\label{wnf:kernel_blk}
\begin{aligned}
\bmbm{\nabla \hat \Phi_k(x_k^1)\\\vdots\\\nabla \hat \Phi_k(x_k^N)}=K^{1,1} \mfC_\text{bd}^{-1}(K^{1,1}v+K^{1,2}e).
\end{aligned}
\end{equation}

The overall algorithm is summarized in Algorithm \ref{alg:ker_newton}.

\begin{algorithm}[ht]
\caption{Wasserstein Newton's method with RKHS. }\label{alg:ker_newton}
\begin{algorithmic}[1]
\REQUIRE initial positions $\{x_0^n\}_{n=1}^N$, $\epsilon\geq 0$, step sizes $\alpha_k$, maximum iteration $K$.
\STATE Set $k=0$.
\WHILE{$k< K$ and the convergence criterion is not met}
\STATE Calculate $H,v,e, K^{1,1}$, $K^{1,2}$ and $K^{2,1}$ in Proposition \ref{prop:ker} based on $\{x_k^n\}_{n=1}^N$. 
\STATE Formulate $\hat \Phi_k(x_k^n)$ via \eqref{wnf:kernel} or \eqref{wnf:kernel_blk}.
\STATE Update particle positions by $$x_{k+1}^n = x_k^n+\alpha_k \nabla \hat \Phi_k(x_k^n).$$
\STATE Set $k=k+1$.
\ENDWHILE
\end{algorithmic}
\end{algorithm}

Besides, we can use a sparse kernel approximation \citep{kwng,ipsof} to further reduce the computational cost. Namely, we assume that $\Phi(x)$ takes the form
\begin{equation}
\Phi(x)=\sum_{m=1}^M\sum_{i=1}^d\alpha_{i,m} \p_i k(z^m,x).
\end{equation}
Here $M\ll N$ and $\{z^m\}_{m=1}^M$ are randomly sampled from $\{x_k^n\}_{n=1}^N$. This can reduce the computational cost to $O(MN^2d^4)$ (or $O(MNd^4)$ if we apply the block-diagonal approximation). 

\begin{remark}
In future works, we expect to find efficient methods to approximate the solution to \eqref{equ:kwnewton} with low computational cost in terms of $N$ and $d$.
\end{remark}
\begin{remark}
We notice that our Wasserstein Newton's method with RKHS is related to Stein variational Newton's method (SVN) \citep{asvnm}. Here SVN restricts the Newton's direction of general transformation map in RKHS, while our method restricts the potential function of gradient transportation map in RKHS. See details in the appendix. We also provide detailed numerical comparison of these methods in section \ref{sec:num}.
\end{remark}

\section{Convergence analysis of Information Newton's method}\label{sec:conv}
In this section, we introduce general update rules of information Newton's method in terms of probability densities and analyze their convergence rates in both distance and objective function value. 

We briefly review the Riemannian structure of probability space as follows. Given a metric tensor $\mcG(\rho)$ and two probability densities $\rho_0,\rho_1\in\mcP(\Omega)$, we denote the distance $\mcD(\rho_0,\rho_1)$ as follows
$$
\mcD(\rho_0,\rho_1)^2 = \inf_{\hat \rho_s,s\in[0,1]}\left\{\int_0^1\int\p_s\hat  \rho_s\mcG(\hat \rho_s)^{-1}\p_s\hat \rho_sdxds:\hat\rho_s|_{s=0}=\rho_0,\hat \rho_s|_{s=1}=\rho_1\right\}.
$$
For the Wasserstein metric, $\mcD(\rho_0,\rho_1)$ is the Wasserstein-2 distance between $\rho_0$ and $\rho_1$. 
Denote the inner product on cotangent space $T_\rho^*\mcP(\Omega)$ by
$$
\lra{\Phi_1,\Phi_2}_\rho = \int \Phi_1\mcG(\rho)^{-1}\Phi_2 dx,\quad \Phi_1,\Phi_2\in T_\rho^*\mcP(\Omega),
$$
and $\|\Phi\|_\rho^2= \lra{\Phi,\Phi}_{\rho}$. And we introduce the definition of the parallelism.
\begin{definition}[Parallelism]
We say that $\tau:T_{\rho_0}\mcP(\Omega)\to T_{\rho_1}\mcP(\Omega)$ is a parallelism from $\rho_0$ to $\rho_1$, if for all $\Phi_1,\Phi_2\in T_{\rho_0}\mcP(\Omega)$, it follows
$$
\lra{\Phi_1,\Phi_2}_{\rho_0} = \lra{\tau \Phi_1,\tau \Phi_2}_{\rho_1}.
$$
\end{definition}

To analyze the convergence rate, we introduce $\nabla^nE(\rho)$. This is a $n$-form on the cotangent space $T_\rho^* \mcP(\Omega)$, which is recursively defined by
$$
\nabla^n E(\rho)(\Phi_1,\dots,\Phi_n) = \left.\frac{\p}{\p s} \nabla^{n-1} E(\Exp_\rho( s\Phi_n))(\tau_s\Phi_1,\dots,\tau_s\Phi_{n-1})\right|_{s=0},
$$
where $\tau_s$ is the parallelism from $\rho$ to $\Exp_\rho( s\Phi_n)$.


\subsection{Convergence analysis in distance}
The general update rule of the information Newton's method follows
\begin{equation}
    \rho_{k+1} = \Exp_{\rho_k}(\alpha_k\Phi_k),\quad \mcH_E(\rho_k)\Phi_k + \mcG(\rho_k)^{-1}\frac{\delta E}{\delta {\rho_k}}=0.
\end{equation}
Here $\alpha_k>0$ is a step size and $\Exp_{\rho_k}(\cdot)$ is the exponential map at $\rho_k$. 

Recall that in the convergence proof of Euclidean Newton methods, it is assumed that $\nabla^2f(x)$ is positive definite around a small neighbour of the optimal solution $x^*$. In the probability space, we assume that the following assumption holds analogously.
\begin{assumption}\label{asmp:1}
Assume that there exists $\zeta, \delta_1,\delta_2,\delta_3>0$, such that for all $\rho$ satisfying $\mcD(\rho,\rho^*)<\zeta$ and $\Phi_1,\Phi_2\in T^*_\rho \mcP(\Omega)$, the following statements hold.
\begin{equation}\label{ass:A1}
\nabla^2E(\rho)(\Phi_1,\Phi_1)\geq \delta_1\|\Phi_1\|_{\rho}^2,\tag{A1}
\end{equation}
\begin{equation}\label{ass:A2}
\nabla^2E(\rho)(\Phi_1,\Phi_1)\leq \delta_2\|\Phi_1\|_{\rho}^2,\tag{A2}
\end{equation}
\begin{equation}\label{ass:A3}
|\nabla^3E(\rho)(\Phi_1,\Phi_1,\Phi_2)|\leq \delta_3\|\Phi_1\|_{\rho}^2\|\Phi_2\|_{\rho}.\tag{A3}
\end{equation}
\end{assumption}

Relying on Assumption \ref{asmp:1}, Theorem \ref{thm:newton} shows the quadratic convergence rate of the Newton's method in the probability space. 

\begin{theorem}\label{thm:newton}
Suppose that Assumption \ref{asmp:1} holds, $\rho_k$ satisfies $\mcD(\rho_k,\rho^*)<\zeta$ and the step size $\alpha_k=1$. Then, we have
$$
\mcD(\rho_{k+1},\rho^*) = O(\mcD(\rho_{k},\rho^*)^2).
$$
\end{theorem}
We present a sketch of the proof. For simplicity, we denote $T_k=\Exp_{\rho_k}^{-1}(\rho^*)$.
\begin{proposition}\label{prop:est0}
Suppose that Assumption \ref{asmp:1} holds. Let $\tau$ be the parallelism from $\rho_k$ to $\rho_{k+1}$. There exists a unique $R_k\in T_{\rho_k}^*\mcP(\Omega)$ such that
$$
T_k=\tau^{-1}T_{k+1} +\Phi_k+R_k.
$$
Then, we have
$$
\|T_{k+1}\|_{\rho_{k+1}}\leq \frac{\delta_3}{\delta_1}\|T_k\|_{\rho_k}^2+\frac{\delta_2}{\delta_1}\|R_k\|_{\rho_k}.
$$
\end{proposition}
In order to provide an estimation on $\|R_k\|_{\rho_k}$, we introduce Lemma \ref{lem:est}.

\begin{lemma}\label{lem:est}
For all $\Psi\in\mcT_{\rho_k}^*\mcP(\Omega)$, it follows
$$
\int \Psi\mcG(\rho_k)^{-1}R_kdx = O(\|\Psi\|_{\rho_k}\|T_k\|_{\rho_k}^2).
$$
\end{lemma}

Taking $\Psi=R_k$ in Lemma \ref{lem:est} yields $\|R_k\|_{\rho_k}=O(\|T_k\|_{\rho_k}^2)$. Because the geodesic curve has constant speed \citep{aitdm}, $\|T_k\|_{\rho_k}^2=\mcD(\rho_k,\rho^*)^2$. As a result, we have
$$
 \mcD(\rho_{k+1},\rho^*)\leq \frac{\delta_2}{\delta_1}\mcD(\rho_{k},\rho^*)^2+\frac{\delta_3}{\delta_1}\|R_k\|_{\rho_k}=O(\mcD(\rho_k,\rho^*)^2).
$$


\subsection{Convergence analysis in objective function value}
We next analyze the convergence rate based on our approximation methods in section \ref{sec:par}. In practice, we use the approximated solution $ \Phi_k$ to update $\rho_k$. Here $ \Phi_k$ is the solution to the variational problem
\begin{equation}\label{var_s}
\inf_{\Phi\in \mcS} \int \Phi \mcH_E(\rho_k) \Phi  dx+2\int \Phi \mcG(\rho_k)^{-1}\frac{\delta E}{\delta {\rho_k}}dx+\lambda \int \Phi \mcR_\mcS \Phi dx.
\end{equation}
Here $\mcH$ is a linear subspace of $\mcF(\Omega)$, $\lambda\geq 0$ is a regularization parameter and $ \int \Phi \mcR_\mcS \Phi dx$ is a regularization term in $\mcS$. For instance, if $\mcS$ is an RKHS, then $ \int \Phi \mcR_\mcH \Phi dx$ can be the squared norm of RKHS, i.e., $\|\Phi\|_{\mcS}^2$. 

Suppose that $P:T_{\rho_k}^*\mcP(\Omega)\to \mcS$ is a projection operator from $T_{\rho_k}^*\mcP(\Omega)$ to $\mcS$ and $P^*:\mcS\to T_{\rho_k}\mcP(\Omega)$ is its adjoint operator. Then, we can write $\hat \Phi_k$ in the closed-form formulation:
\begin{equation}\label{equ:hat_phik}
 \Phi_k =- P(P^*\mcH_E(\rho_k) P+\mcR_\mcS)^{-1} P^* \mcG(\rho_k)^{-1}\frac{\delta E}{\delta {\rho_k}}.
\end{equation}
For simplicity, we use the following notations.
\begin{equation}
g_k=\mcG(\rho_k)^{-1}\frac{\delta E}{\delta {\rho_k}},\quad \mcH_{E,P}=P(P^*\mcH_E(\rho_k) P+\mcR_\mcS)^{-1} P^*.
\end{equation}
For the subspace $\mcS$ and the regularization term $\lambda \int \Phi \mcR_\mcS \Phi dx$, we further assume that the following three statements hold. 
\begin{assumption}\label{asmp:2}
There exists $\epsilon_1\geq 0$, for all $\rho_k$ satisfying $\mcD(\rho_k,\rho^*)<\zeta$, such that
\begin{equation}\label{ass:A4}
\begin{aligned}
&\left|\int g_k(\mcH_{E,P} -\mcH_{E}(\rho_k)^{-1})g_k dx \right|
\leq& \epsilon_1\int g_k\mcH_{E}(\rho_k)^{-1} g_k dx.
\end{aligned}\tag{A4}
\end{equation}

There exists $\epsilon_2\geq 0$, for all $\rho_k$ satisfying $\mcD(\rho_k,\rho^*)<\zeta$, such that
\begin{equation}\label{ass:A5}
\begin{aligned}
&\left|\int g_k(\mcH_{E,P} \mcH_E(\rho_k)\mcH_{E,P}-\mcH_{E,P})g_kdx \right|
\leq& \epsilon_2\int g_k \mcH_{E,P} g_k dx.
\end{aligned}\tag{A5}
\end{equation}

There exists $\delta_4\geq 0$, for all $\rho_k$ satisfying $\mcD(\rho_k,\rho^*)<\zeta$, such that
\begin{equation}\label{ass:A6}
\norm{ \mcH_{E,P} \mcG(\rho)^{-1}\Phi }_{\rho_k}\leq \delta_4 \norm{ \Phi }_{\rho_k}.\tag{A6}
\end{equation}

\end{assumption}
The update rule in terms of density follows
$$
\rho_{k+1} = \Exp_{\rho_k}(\alpha_k \Phi_k).
$$
\begin{theorem}\label{thm:conv_e}
Under Assumption \ref{asmp:1} and \ref{asmp:2}, for $\rho_k$ satisfying $\mcD(\rho_k,\rho^*)<\zeta$, with $\alpha_k=1$, we have the linear convergence rate
\begin{equation*}
E(\rho_{k+1}) - E(\rho^*)\leq (\epsilon_1+\epsilon_2+\epsilon_1\epsilon_2)(E(\rho_{k}) - E(\rho^*))+\mcO((E(\rho_{k}) - E(\rho^*))^{3/2}).
\end{equation*}
\end{theorem}

From Theorem \ref{thm:conv_e}, we note that if the linear subspace $\mcS$ is appropriately chosen such that $\mcH_{E,P}$ is close to $\mcH_{E}(\rho_k)^{-1}$ in the sense of \eqref{ass:A4} and \eqref{ass:A5}, then $\epsilon_1,\epsilon_2$ will be close to $0$. This yields a sharp asymptotic convergence rate in terms of optimality gap, i.e., $E(\rho_{k}) - E(\rho^*)$. 

\begin{remark}
We note that $\epsilon_2=\mcO(\lambda)$. This comes from the following identity.
\begin{equation}\label{equ:hep}
\begin{aligned}
&\mcH_{E,P} \mcH_E(\rho_k)\mcH_{E,P}-\mcH_{E,P}\\
=&P(P^*\mcH_E(\rho_k) P+\lambda \mcR_\mcH)^{-1} P^* \mcH_E(\rho_k)P(P^*\mcH_E(\rho_k) P+\lambda \mcR_\mcH)^{-1} P^*\\
&-P(P^*\mcH_E(\rho_k) P+\lambda \mcR_\mcH)^{-1} P^*\\
=&\lambda P(P^*\mcH_E(\rho_k) P+\lambda \mcR_\mcH)^{-1} \mcR_\mcS(P^*\mcH_E(\rho_k) P+\lambda \mcR_\mcH)^{-1} P^*.
\end{aligned}
\end{equation}
\end{remark}

\subsection{Convergence analysis in terms of samples}
In practice, we replace $\rho_k$ in the variational problem \eqref{var_s} by $\hat\rho_k(x)=\frac{1}{N}\sum_{n=1}^N{\delta(x-x_k^n)}$ to solve $\hat \Phi_k$. Here $x_k^n\sim \rho_k$. A natural question arises: with increasing sample numbers $N$, does $\hat \Phi_k$ from samples converge to $\Phi_k$ from distribution? Under further assumptions, the answer is yes and we postpone the justification in the appendix. 

To establish the convergence rate, we further assume that the following statements hold.
\begin{assumption}\label{asmp:3}
There exists $\epsilon_3\geq 0$, for all $\rho_k$ satisfying $\mcD(\rho_k,\rho^*)<\zeta$, such that
\begin{equation}\label{ass:A7}
\begin{aligned}
&\left| \int (\hat \Phi_k-\Phi_k)g_k dx -\frac{1}{2}\int (\hat \Phi_k-\Phi_k)(\mcH_{E,P}\mcH_E(\rho_k)+\mcH_E(\rho_k)\mcH_{E,P})g_k dx \right|\\
\leq& \frac{\epsilon_3}{2}\int g_k\mcH_{E}(\rho_k)^{-1} g_k dx.
\end{aligned}\tag{A7}
\end{equation}

There exists $\epsilon_4\geq 0$, for all $\rho_k$ satisfying $\mcD(\rho_k,\rho^*)<\zeta$, such that
\begin{equation}\label{ass:A8}
\begin{aligned}
&\left| \int (\hat \Phi_k-\Phi_k) \mcH_E(\rho_k)(\hat \Phi_k-\Phi_k) dx \right|
\leq& \epsilon_4\int g_k\mcH_{E}(\rho_k)^{-1} g_k dx.
\end{aligned}\tag{A8}
\end{equation}
\end{assumption}
The update rule in terms of density follows
$$
\rho_{k+1} = \Exp_{\rho_k}(\alpha_k \hat \Phi_k).
$$
\begin{theorem}\label{thm:conv_s}
Under Assumption \ref{asmp:1}, \ref{asmp:2} and \ref{asmp:3}, for $\rho_k$ satisfying $\mcD(\rho_k,\rho^*)<\zeta$, with $\alpha_k=1$, we have the linear convergence rate
\begin{equation*}
E(\rho_{k+1}) - E(\rho^*)\leq (\epsilon_1+\epsilon_2+\epsilon_1\epsilon_2+\epsilon_3+\epsilon_4)(E(\rho_{k}) - E(\rho^*))+\mcO((E(\rho_{k}) - E(\rho^*))^{3/2}).
\end{equation*}
\end{theorem}

\section{Numerical experiments}\label{sec:num}
In this section, we present numerical experiments to demonstrate the strength of information Newton's methods.

\subsection{Toy examples}\label{ssec:num_toy}
We compare particle implementations among Wasserstein Newton's methods with affine models \ref{alg:aff_newton}/RKHS \ref{alg:ker_newton}  (WNewton-a/WNewton-k), Wasserstein gradient flow (WGF), Hessian Approximated Lagrangian Langevin dynamics (HALLD) and Stein variational Newton's method with the scaled Hessian kernel (SVN-H) \citep{asvnm}. We note that the update rule of WGF satisfies
$$
x_{k+1}^n=x_k^n-\alpha_k(\nabla f(x_k^n)+\xi_k(x_k^n)).
$$
The update rule of HALLD follows
$$
x_{k+1}^n = x_k^n-\alpha_k\nabla^2f(x_k^n)^{-1}(\nabla f(x_k^n)+\xi_k(x_k^n)).
$$
We note that the density evolution of HALLD and HAMCMC are identical to each other. 
In other words, we replace the Brownian motion in HAMCMC by $\xi_k$ in HALLD. Here $\xi_k$ is an approximation of $\nabla\log \rho_k$. For all compared methods, we use constant step sizes. For the calculation of $\xi_k$, we apply KDE with Gaussian kernels and the kernel bandwidth is selected by the Brownian Motion method \citep{aigf}[section 5.1]. This method adaptively learns the bandwidth from samples generated by Brownian motions.

We first consider a $1$D target density $\rho^*(x)\propto\exp\pp{-f(x)}$, where $f(x)=\frac{1}{2}(x^2-1)^2$. For WGF, we set  $\alpha_k=0.01$. For SVGD, we set $\alpha_1=0.1$ and adjust the step size by Adagrad \citep{Adagrad}. For WNewton-a and WNewton-k, we let $\alpha_k=1$, $\epsilon=0$ and $\gamma=0$. Namely, we do not apply the hybrid update. For HALLD and SVN-H, we set $\alpha_k=1$. 

The sample number follows $N=100$. The initial distribution follows $\mcN(0,0.01)$. We plot the distribution after $2,5,10,20$ iterations in Figure \ref{fig:toy1}. Although we use affine/kernel approximations to compute the Newton's direction, WNewton-a and WNewton-k tend to converge to the target density and they are faster than WGF. SVGD has similar performance with WGF. HALLD and SVN-H have some particle which tend to diverge. This may result from that the target density is not log-concave. 

\begin{figure}[!htbp]
\centering
\begin{minipage}[t]{\textwidth}
\centering
\includegraphics[width=\linewidth]{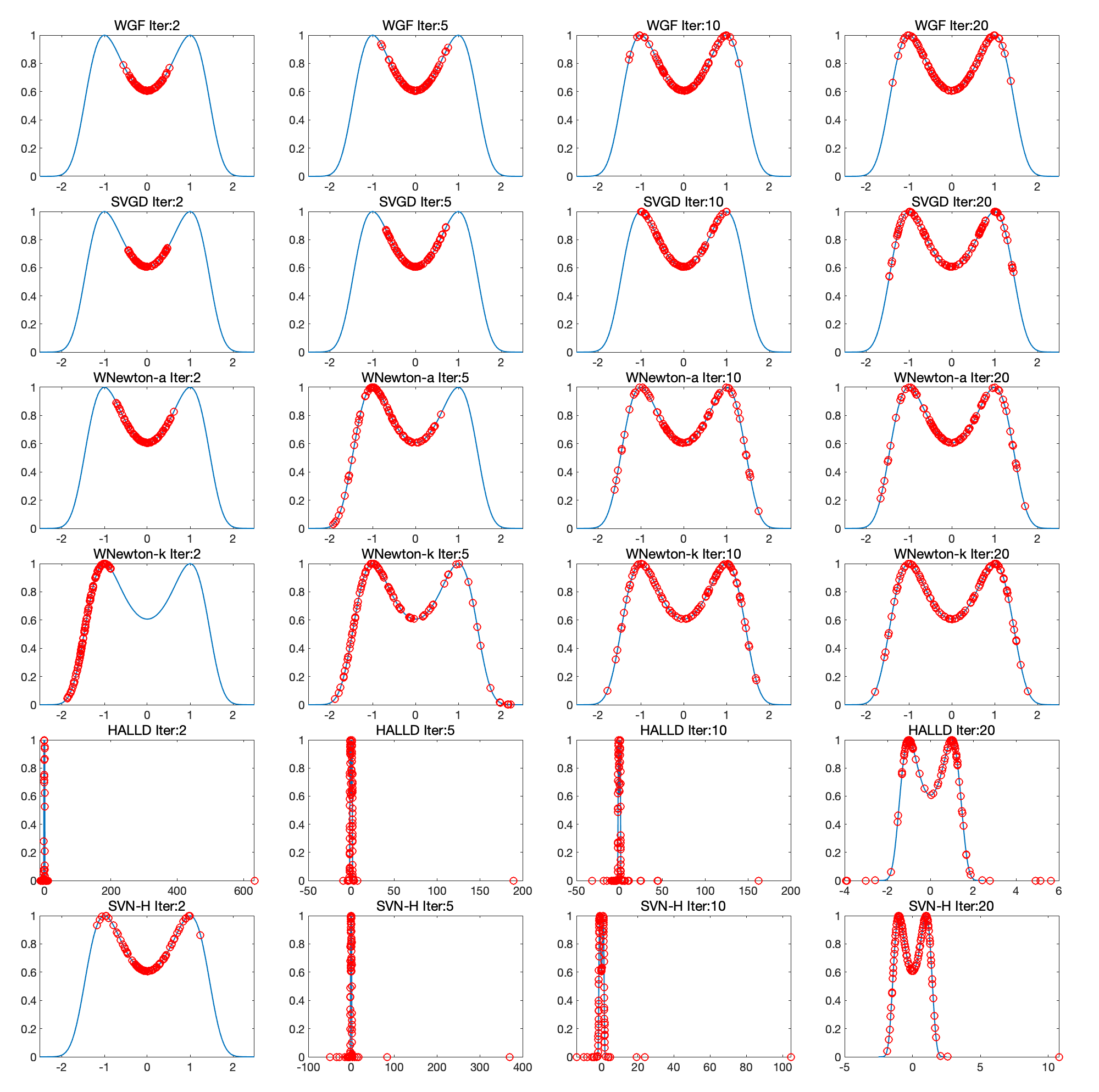}
\end{minipage}
\caption{Comparison among WGF, SVGD, WNewton-a, WNewton-k, HALLD and SVN-H in 1D toy example. Left to right: sample distribution after $2,5,10,20$ iterations.}\label{fig:toy1}
\end{figure}

Then, we let the target density $\rho^*$ to be a $2$D bimodal distribution \citep{viwnf}. For WGF, we set $\alpha_k=0.1$. For SVGD, we set $\alpha_1=1$ and adjust the step size via Adagrad.  For WNewton-a, we apply the hybrid update and set $\alpha_k=0.2,\epsilon=0$ and $\gamma=0.5$. For WNewton-k, we set $\alpha_k=1, \epsilon=0, \gamma=0$. For HALLD, we set $\alpha_k=0.2$. For SVN-H, we set $\alpha=1$. 

The initial distribution follows $\mcN([0,10]',I)$. We plot the distribution after $2,5,10,20$ iterations with $N=100$ samples in Figure \ref{fig:toy}. WNewton-k converges rapidly toward the target density. HALLD fails to converge because $\nabla^2 f$ becomes singular on certain sample points. SVN-H barely moves because the initial distribution is not close enough to the target distribution. SVGD converges slower than WGF. The Wasserstein Newton's direction helps samples to converge faster towards the target density with robustness.

\begin{figure}[!htbp]
\centering
\begin{minipage}[t]{\textwidth}
\centering
\includegraphics[width=\linewidth]{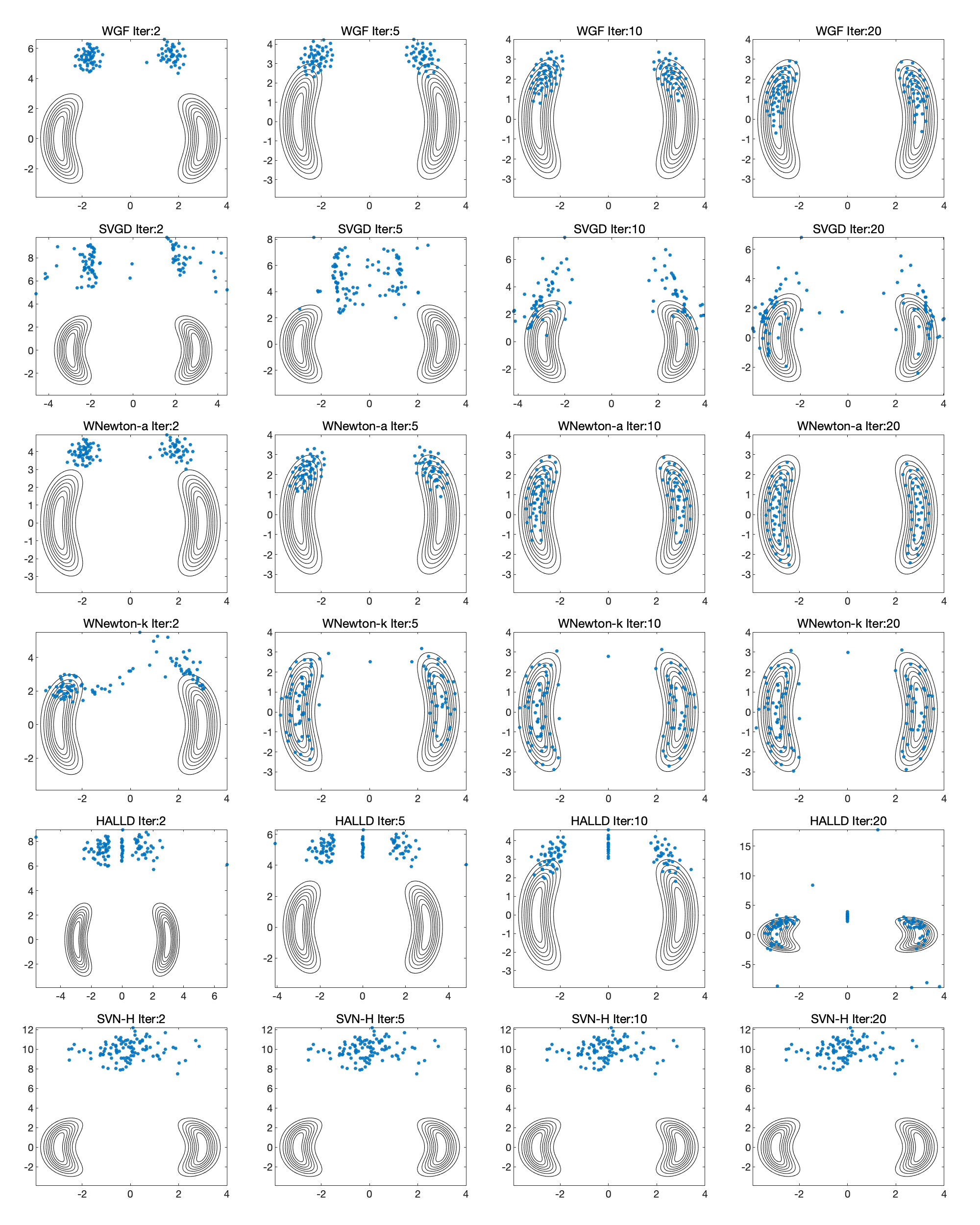}
\end{minipage}
\caption{Comparison among WGF, SVGD, WNewton-a, WNewton-k, HALLD and SVN-H in 2D toy example. Left to right: sample distribution after $2,5,10,20$ iterations.}\label{fig:toy}
\end{figure}

Next, we present numerical results on a 2D double-banana shape posterior density in \citep{asvnm}. For WGF, we set $\alpha_k=0.002$. For SVGD, we set $\alpha_1=0.1$ and adjust the step size via Adagrad. For WNewton-a, we apply the hybrid update and set $\alpha_k=0.2,\epsilon=0$ and $\gamma=0.001$. For WNewton-k, we set $\alpha_k=1, \epsilon=0, \gamma=0$. For HALLD and SVN-H, we set $\alpha_k=1$. 

Similarly, we plot the distribution after $2,5,10,20$ iterations with $N=100$ samples in Figure \ref{fig:double_banana}. WNewton-k and SVN-H converges toward the posterior distribution in no more than 5 iterations. WNewton-a collapses around the center of the lower banana. WGF and SVGD take nearly 20 iterations to converge. HALLD converges rapidly but it diverge at iteration 20. Here we notice that WNewton does not require heavy tunes of step sizes. The step size $\alpha_k=1$ usually leads to robust performance.

\begin{figure}[!htbp]
\centering
\begin{minipage}[t]{0.95\textwidth}
\centering
\includegraphics[width=\linewidth]{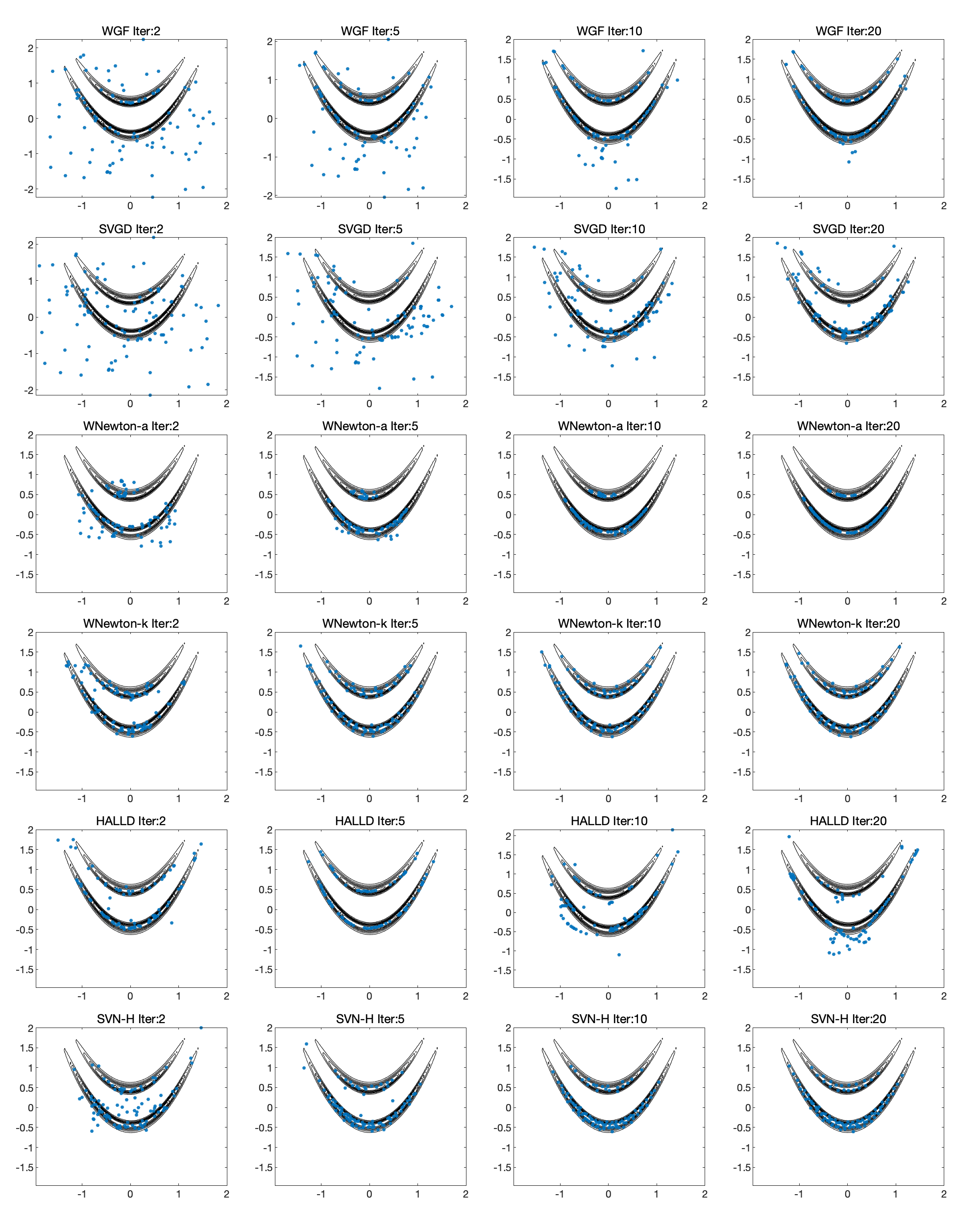}
\end{minipage}
\caption{Comparison among WGF, SVGD, WNewton-a, WNewton-k, HALLD and SVN-H in 2D double banana example. Left to right: sample distribution after $2,5,10,20$ iterations.}\label{fig:double_banana}
\end{figure}

\subsection{Conditioned diffusion}
The conditioned diffusion example is a $100$-dimensional model from a Langevin SDE, with state $u_t:[0,T]\to \mbR$ and dynamics give by
$$
du_t = \frac{\beta u(1-u^2)}{1+u^2} dt+dx_t,\quad u_0=0.
$$
Here $x=(x_t)_{t\geq 0}$ is the standard Brownian motion. The goal is to infer the driving process $x_t$ and its pushfoward to the state $u$. Detailed setup of this test case can be found in \citep{asvnm}. 

We compare WNewton-a with WGF, SVGD, SVN-H and HALLD. We do not compare WNewton-k because per-iteration computation cost in the current implementation is too heavy on this test case with $N=1000$ and $d=100$. For WGF, we set $\alpha_k=0.01$. For SVGD, we set $\alpha_1=0.1$ and adjust step sizes via Adagrad. For WNewton-a, SVN-H and HALLD, we set $\alpha_k=1$. From Figure \ref{fig:cond_diff}, we note that the posterior mean (which captures the trends of true path) from WNewton-a, SVN-H and HALLD almost converge in approximately 10 iterations. Meanwhile, the posterior mean from WGF and SVGD takes 50-100 iterations to converge. Compared to SVN-H, WNewton-a tends to have narrower credible interval. The credible interval of HALLD in $[0,0.5]$ after 100 iterations has larger fluctuation. 

\begin{figure}[!htbp]
\centering
\begin{minipage}[t]{0.85\textwidth}
\centering
\includegraphics[width=\linewidth]{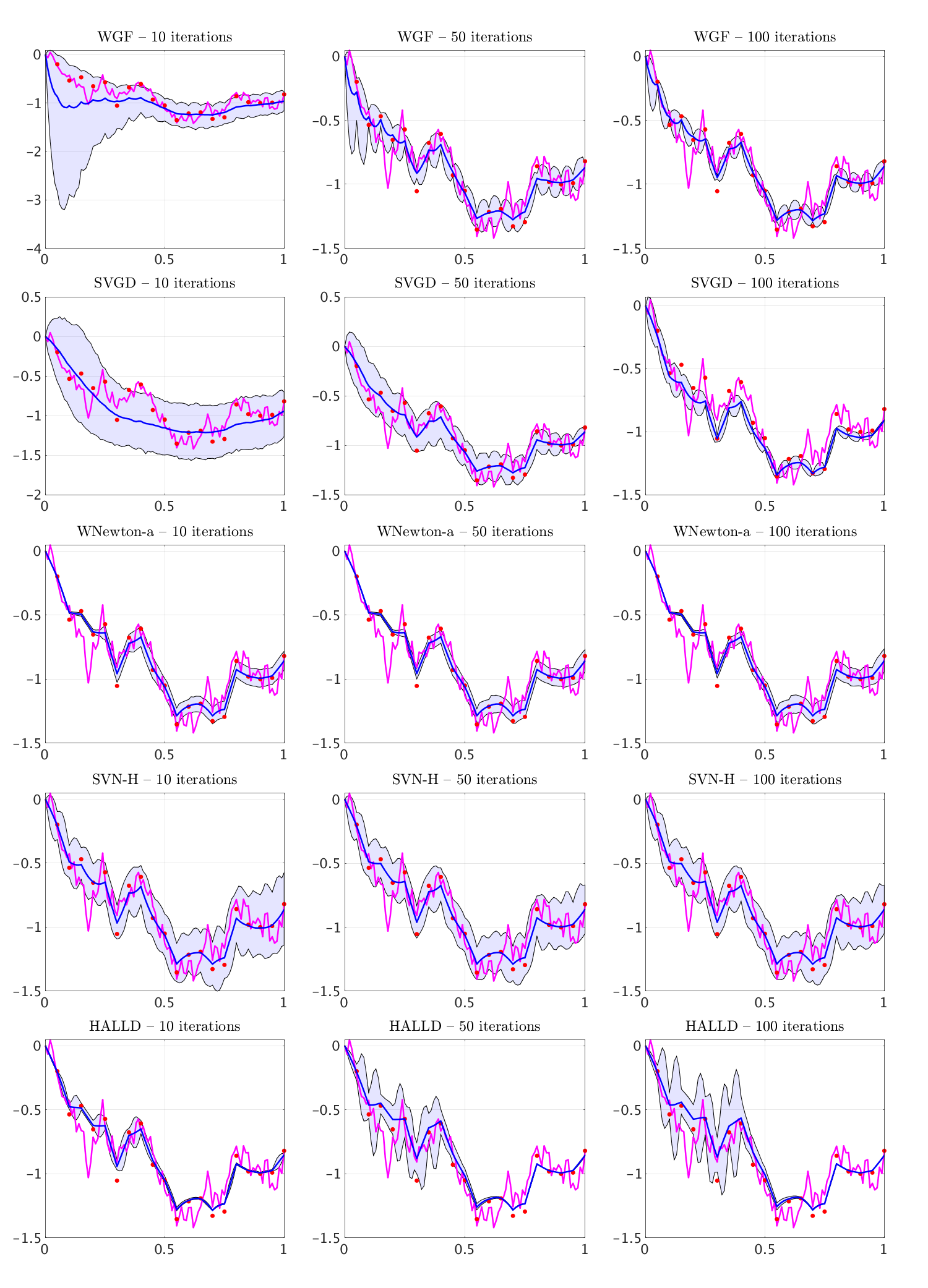}
\end{minipage}
\caption{Comparison among WGF, SVGD, WNewton-a, HALLD and SVN-H in 100D conditioned diffusion example. Left to right: sample distribution after $10,50,100$ iterations. Red dots: noisy observations. Purple line: ground truth. Blue line: posterior mean. Shaded area: $90\%$ credible interval.}\label{fig:cond_diff}
\end{figure}

\subsection{Bayesian logistic regression}
We perform the standard Bayesian logistic regression experiment on the Covertype dataset, following the settings in \citep{SVGD}. We compare WNewton-a and WNewton-k with MCMC, SVGD \citep{SVGD}, and WGF. The performances of SVN-H and HALLD on this test example are not ideal. For the calculation of $\xi_k$ in WGF and WNewton-a, we use KDE with Gaussian kernel and the bandwidth is selected by the median method, which is the same as \citep{SVGD}. The sample number follows $N=50$. The mini-batch size for stochastic gradient and Hessian evaluations in each iteration is $100$.

We first discuss the choice of step sizes. The initial step sizes for the compared methods are given in Table \ref{tab:step}. Except for SVGD, the initial step sizes are selected from $\{i\cdot 10^j|i\in\{1,2,5\},j\in\{-3,\dots,-7\}\}$ to ensure the best performance. For SVGD, we use the initial step size in \citep{SVGD} and adjust step sizes by Adagrad. For MCMC, WGF and WNewton-k, the step size is multiplied by $0.9$ every $100$ iterations. For WNewton-a, the step size is multiplied by $0.82$ every $100$ iterations. 

\begin{table}[!htbp]
    \centering
    \begin{tabular}{|c|c|c|c|c|c|c|}
    \hline
         Method &MCMC & SVGD&WGF&WNewton-a&WNewton-k  \\\hline
         Step size $\alpha_1$& 1e-5& 0.05&1e-5&2e-3&2e-3\\ \hline
    \end{tabular}
    \caption{Initial step sizes for algorithms in comparison.}
    \label{tab:step}
\end{table}
We then elaborate on the implementation details of compared methods. For WNewton-k, we apply the block-diagonal approximation to accelerate the computation. For WNewton-a and WNewton-k, we set $\epsilon=1$ and use the hybrid update with $\gamma=5\times 10^{-3}$ and $\gamma= 10^{-3}$ respectively.

From Figure \ref{fig:blr}, we observe that WNewton-k has the best performance in terms of test accuracy and test log-likelihood and it converges much faster compared to other methods. Namely, WNewton-k has ideal performance on test test tests in less than 200 iterations. WNewton-a and WNewton-k achieves higher test log-likelihood. This indicates that the approximated Wasserstein Newton's direction leads to better generalization on the test set. 

\begin{figure}[!htbp]
\centering
\begin{minipage}[t]{0.48\textwidth}
\centering
\includegraphics[width=\linewidth]{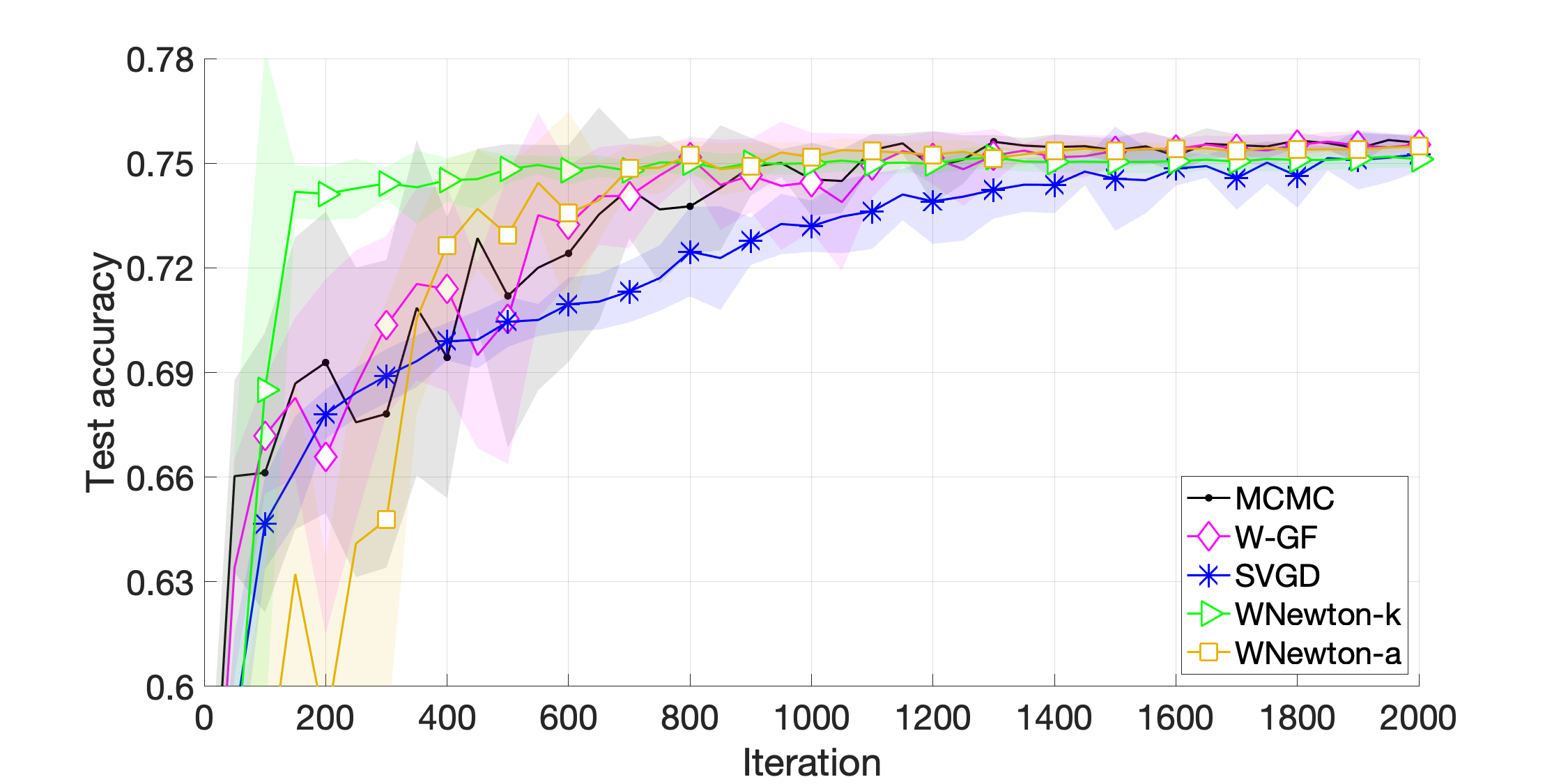}
\end{minipage}
\begin{minipage}[t]{0.48\textwidth}
\centering
\includegraphics[width=\linewidth]{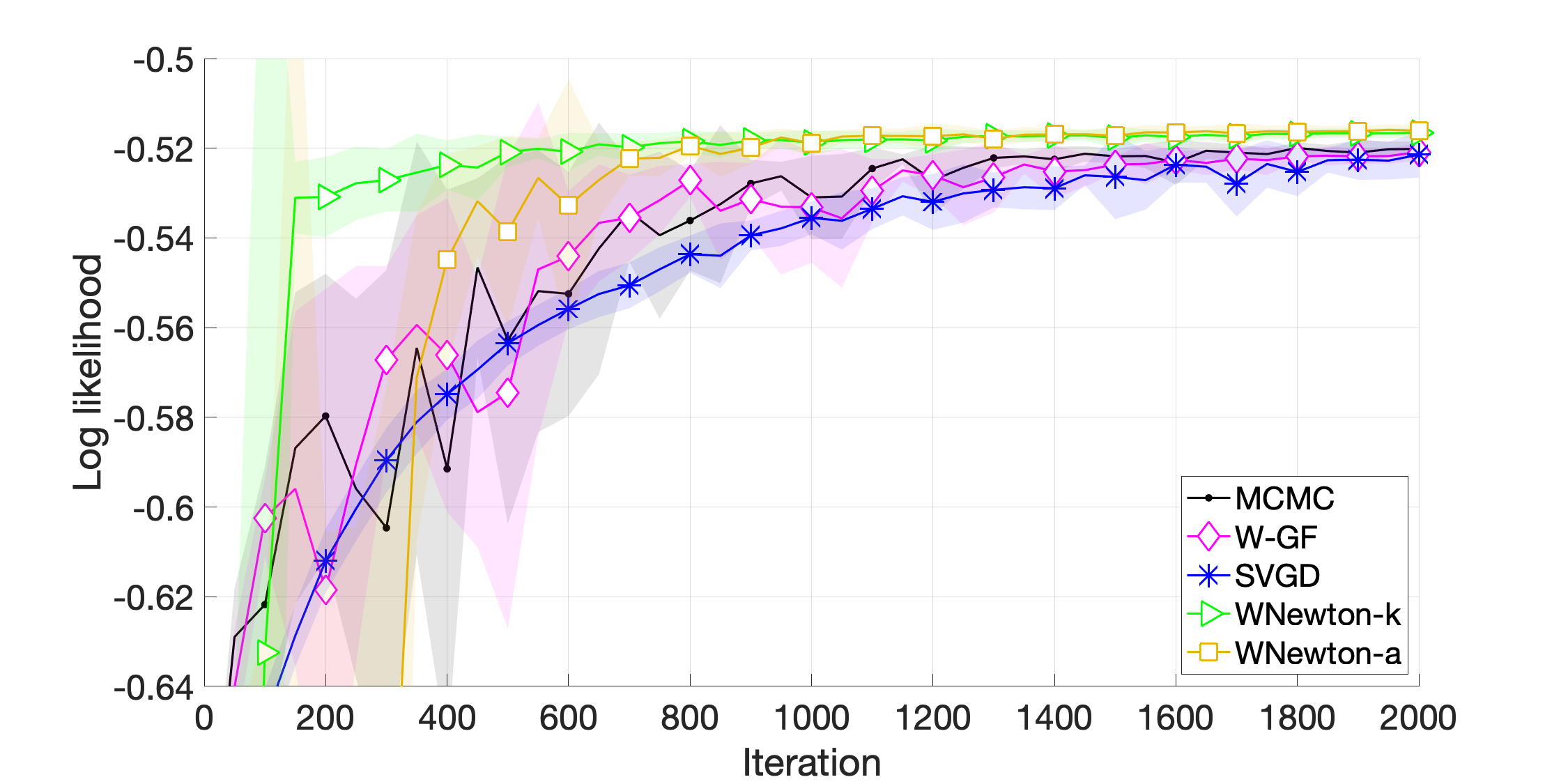}
\end{minipage}
\caption{Comparison of different methods on Bayesian logistic regression, averaged over $10$ independent trials. The shaded areas show the variance over $10$ trials. Left: Test accuracy; Right: Test log-likelihood.}\label{fig:blr}
\end{figure}

\section{Conclusion}
In this paper, we introduce information Newton's flows (second-order optimization methods) for optimization problems in probability space arising from Bayesian statistics, inverse problems, and machine learning. Here two information metrics, such as Fisher-Rao metric and Wasserstein-2 metric, are considered. Several examples and convergence analysis of the proposed second-order methods are provided. Following the fact that the Wasserstein gradient flow of KL divergence formulates the Langevin dynamics, we derive the Wasserstein Newton's flow of KL divergence as Newton's Langevin dynamics. Focusing on Newton's Langevin dynamics, we study analytical examples in one-dimensional sample space and Gaussian families. We further propose practical sampling efficient algorithms, in affine models and RKHS, to implement Newton's Langevin dynamics. We show the convergence rate of information Newton's method with approximated solutions. The numerical examples in Bayesian sampling problems demonstrate the effectiveness of the proposed method.  


\newpage

\appendix

\vskip 0.2in

\section{Definitions and notations}
In this section, we present several definitions and notations used in this paper. We briefly review the concept of self-adjoint operator.
\begin{definition}[Self-adjoint]
Suppose that $V$ is a Hilbert space and let $\mcH:V\to V^*$ be a linear operator. $V^*$ is the adjoint space of $V$, which consists of all linear functionals on $V$. Let $(f,v)=(v,f)=f(v)$ denote the coupling of $v\in V$ and $f\in V^*$. The adjoint operator of $\mcH$ is the unique linear operator $\mcH^*:V\to V^*$, which satisfies
$$
(\mcH v_1,v_2) = (v_1,\mcH^*v_2),\quad \forall v_1,v_2\in V.
$$
We say that $\mcH$ is self-adjoint if $\mcH=\mcH^*$.
\end{definition}
\begin{remark}
If $V=\mbR^d$ is the Euclidean space, then the linear operator $\mcH$ can be viewed as a matrix in $\mbR^{d\times d}$. Then, to say that $\mcH$ is self-adjoint operator is equivalent to say that $\mcH$ is a symmetric matrix.
\end{remark}
We define positive definite operators as follows.
\begin{definition}
Suppose that $V$ is a Hilbert space and let $\mcH:V\to V^*$ be a self-adjoint linear operator. We say that $\mcH$ is positive definite, if $(\mcH v, v)>0$ for all $v\in V$, $v\neq 0$. 
\end{definition}

\section{Proofs in section \ref{section3}}
In this section, we present details and proofs for propositions in section \ref{section3}.
Proposition \ref{prop:unique} provides a sufficient condition to ensure that the Hessian operator is injective (invertible). 
\begin{proposition}\label{prop:unique}
Suppose that $g_\rho(\Hess E(\rho) \sigma, \sigma)>0$ for all $\sigma\neq 0,\sigma\in T_\rho\mcP(\Omega)$. Namely, $\mcH_E(\rho)$ is positive definite. Then, $\Hess E(\rho)$ is injective.
\end{proposition}
\begin{proof}
If there exist $\sigma_1,\sigma_2\in T_\rho\mcP(\Omega)$ such that $\Hess E(\rho)\sigma_1=\Hess E(\rho)\sigma_2$. Then,
$$
g_\rho((\sigma_1-\sigma_2),\Hess E(\rho)(\sigma_1-\sigma_2))=\int (\sigma_1-\sigma_2) G(\rho)^{-1} \Hess E(\rho)(\sigma_1-\sigma_2) dx=0.
$$
By our assumption  $g_\rho(\Hess E(\rho) \sigma, \sigma)>0$ for all $\sigma\neq 0$, we have $\sigma_1=\sigma_2$.
\end{proof}

\subsection{Proof of Proposition \ref{prop:php}}
The geodesic curve $\hat \rho_s$ satisfies geodesic equation
\begin{equation}\label{equ:geo}
\left\{
\begin{aligned}
&\p_s\hat\rho_s-\mcG(\hat\rho_s)^{-1}\Phi_s=0,\\
&\p_s\Phi_s+\frac{1}{2}\frac{\delta}{\delta \hat\rho_s}\lp \int  \Phi_s\mcG(\hat\rho_s)^{-1}\Phi_sdx\rp=0,
\end{aligned}
\right.
\end{equation}
with initial values $\hat\rho_s|_{s=0} = \rho$ and $\Phi_s|_{s=0} = \Phi$. For the first-order derivative, it follows 
\begin{equation*}
    \frac{d}{ds} E(\hat\rho_s) = \int \p_s\hat\rho_s \frac{\delta E}{\delta \hat\rho_s} dx = \int \Phi_s\mcG(\hat\rho_s)^{-1}\frac{\delta E}{\delta \hat\rho_s} dx,
\end{equation*}
where we utilize the fact that $\mcG(\hat\rho_s)$ is self-adjoint. For the second-order derivative,
\begin{equation*}
\begin{aligned}
 \frac{d^2}{ds^2} E(\hat\rho_s) = &\int \p_s\Phi_s\mcG(\hat\rho_s)^{-1}\frac{\delta E}{\delta \hat\rho_s} dx+\int \p_s \hat\rho_s \frac{\delta}{\delta \hat\rho_s} \pp{\frac{d}{ds} E(\hat\rho_s)}dx\\
 =&-\frac{1}{2}\int \mcA(\hat\rho_s)(\Phi_s,\Phi_s)\mcG(\hat\rho_s)^{-1}\frac{\delta E}{\delta \hat\rho_s} dx+\int   \mcA(\hat\rho_s)\pp{\Phi_s,\frac{\delta E}{\delta \hat\rho_s}} \mcG(\hat\rho_s)^{-1}\Phi_s dx\\
 &+\int \int \pp{\mcG(\hat\rho_s)^{-1}\Phi_s}(y)\frac{\delta^2 E}{\delta \hat\rho_s^2}(x,y) \pp{\mcG(\hat\rho_s)^{-1}\Phi_s}(x) dxdy.
\end{aligned}
\end{equation*}
Based on the definition of $\mcH_E(\rho)$, \eqref{equ:php} is proved by setting $s=0$ in the above formula. To prove \eqref{equ:hehess}, we introduce Lemma \ref{lemma:unique}.
\begin{lemma}\label{lemma:unique}
Let $\mcH$ be a self-adjoint linear operator from $T^*_\rho \mcP(\Omega)\to T_\rho \mcP(\Omega)$. Namely $\mcH^* = \mcH$. Suppose that $\int \Phi\mcH\Phi dx=0$ for all $\Phi\in T_\rho^*\mcP(\Omega)$. Then, $\mcH = 0$.
\end{lemma}
\begin{proof}
Because $\mcH$ is self-adjoint and linear, for any $\Phi\in T_\rho^*\mcP(\Omega)$, it follows
\begin{equation*}
\mcH\Phi = \frac{1}{2}\frac{\delta}{\delta \Phi}\int \Phi\mcH\Phi dx=0.
\end{equation*}
This completes the proof. 
\end{proof} 
Note that $\Hess E(\rho)$ is self-adjoint w.r.t. the metric tensor $G(\rho)$, namely
\begin{equation*}
    (\Hess E(\rho))^*\mcG(\rho) = \mcG(\rho) \Hess E(\rho),\quad \mcG(\rho)^{-1}(\Hess E(\rho))^*=\Hess E(\rho)\mcG(\rho)^{-1}.
\end{equation*}
where $(\Hess E(\rho))^*$ is the adjoint operator of $\Hess E(\rho)$. This tells that $\Hess E(\rho)\mcG(\rho)^{-1}$ is self-adjoint.  We have the following relationship.
 \begin{equation*}
 \int \Phi \mcH_E(\rho) \Phi dx = g_\rho(\Hess E(\rho) \sigma,\sigma) = \int \Phi \mcG(\rho)^{-1} \Hess E(\rho)\Phi dx.
 \end{equation*}
As a direct result of Proposition \ref{prop:unique}, it follows $\mcH_E(\rho)=\Hess E(\rho)\mcG(\rho)^{-1}$. 

\subsection{Newton's flows under Fisher-Rao metric}
\label{ssec:fnf}
For Fisher-Rao metric, the geodesic curve $\hat \rho_s$ satisfies
\begin{equation*}
\left\{
\begin{aligned}
&\p_s\hat\rho_s-\rho_s\lp \Phi_s-\int   \Phi_s\hat\rho_s dy\rp=0,\\
&\p_s\Phi_s+\frac{1}{2}\Phi_s^2-\lp\int \Phi_s \hat\rho_s dy\rp\Phi_s=0.
\end{aligned}
\right.
\end{equation*}
And the bi-linear operator $\mcA^F(\rho)$ follows
\begin{equation}
\mcA^F(\rho)(\Phi_1,\Phi_2) = \Phi_1\Phi_2-\pp{\int \Phi_2\rho dy}\Phi_1 -\pp{\int \Phi_1 \rho dy}\Phi_2 . 
\end{equation}
For simplicity, we let $\mbE_{\rho}[\Phi] = \int \Phi \rho dx$, where $\Phi\in T_\rho^*\mcP(\Omega)$. 

\begin{proposition}[Fisher-Rao Newton's flow]
\label{prop:fnf}
For an objective function $E:\mcP(\Omega)\to \mbR$, the Fisher-Rao Newton's flow follows
\begin{equation}
    \lbb{
&\p_t\rho_t-\rho_t(\Phi_t-\mbE_{\rho_t}[\Phi_t]) = 0,\\
&\mcH_E^F(\rho_t)\Phi_t - \rho_t\pp{\frac{\delta E}{\delta {\rho_t}}-\mbE_{\rho_t}\bb{\frac{\delta E}{\delta {\rho_t}}}}=0,
}
\end{equation}
where $\mcH_E^F(\rho):T_\rho^*\mcP(\Omega)\to T_\rho\mcP(\Omega)$ defines a bi-linear form: for $\Phi\in T_\rho^*\mbP(\Omega)$,
\begin{equation}
\begin{aligned}
\int \Phi \mcH_E^F(\rho)\Phi dx = &\frac{1}{2}\int   \mcA^F(\rho)\pp{\Phi,\frac{\delta E}{\delta \rho}}(\Phi-\mbE_\rho[\Phi])\rho dx\\
 &+\int \int \rho(y)(\Phi(y)-\mbE_\rho[\Phi])\frac{\delta^2 E}{\delta \rho^2}(x,y)dy \rho(x)(\Phi(x)-\mbE_\rho[\Phi]) dx.\\
\end{aligned}
\end{equation}
\end{proposition}
\begin{proof}
Based on Proposition \ref{prop:php}, we only need to prove that
\begin{equation*}
    \int \mcA^F(\rho)(\Phi,\Phi)\mcG^F(\rho)^{-1}\frac{\delta E}{\delta \rho} dx=\int   \mcA^F(\rho)\pp{\Phi,\frac{\delta E}{\delta \rho}} \mcG^F(\rho)^{-1}\Phi dx.
\end{equation*}
The left hand side follows
\begin{equation*}
    \begin{aligned}
    &\int \mcA^F(\rho)(\Phi,\Phi)\mcG^F(\rho)^{-1}\frac{\delta E}{\delta \rho} dx\\
    =&\int \pp{\Phi^2-2\mbE_\rho[\Phi]\Phi}\pp{\frac{\delta E}{\delta \rho}-\mbE_\rho\bb{\frac{\delta E}{\delta \rho}}}\rho dx\\
    =&\int \pp{\Phi-\mbE_\rho[\Phi]}\pp{\frac{\delta E}{\delta \rho}-\mbE_\rho\bb{\frac{\delta E}{\delta \rho}}}\Phi\rho dx-\mbE_\rho[\Phi]\int \pp{\frac{\delta E}{\delta \rho}-\mbE_\rho\bb{\frac{\delta E}{\delta \rho}}}\Phi\rho dx.
    \end{aligned}
\end{equation*}
The right hand side satisfies
\begin{equation*}
    \begin{aligned}
    &\int   \mcA^F(\rho)\pp{\Phi,\frac{\delta E}{\delta \rho}} \mcG^F(\rho)^{-1}\Phi dx\\
    =&\int \pp{\Phi\frac{\delta E}{\delta \rho}-\mbE_\rho\bb{\frac{\delta E}{\delta \rho}}\Phi-\mbE_\rho[\Phi]\frac{\delta E}{\delta \rho} }\pp{\Phi-\mbE_\rho[\Phi]}\rho dx\\
    =&\int \pp{\Phi-\mbE_\rho[\Phi]}\pp{\frac{\delta E}{\delta \rho}-\mbE_\rho\bb{\frac{\delta E}{\delta \rho}}}\Phi\rho dx-\mbE_\rho[\Phi]\int \frac{\delta E}{\delta \rho }\pp{\Phi-\mbE_\rho[\Phi]}\rho dx.
    \end{aligned}
\end{equation*}
We also observe that
\begin{equation*}
    \begin{aligned}
    &\int \frac{\delta E}{\delta \rho }\pp{\Phi-\mbE_\rho[\Phi]}\rho dx=\mbE_\rho \bb{\Phi \frac{\delta E}{\delta \rho }}-\mbE_\rho[\Phi]\mbE_\rho\bb{\frac{\delta E}{\delta \rho }}=\int \pp{\frac{\delta E}{\delta \rho}-\mbE_\rho\bb{\frac{\delta E}{\delta \rho}}}\Phi\rho dx.
    \end{aligned}
\end{equation*}
Hence, the left hand side is equal to the right hand side.
\end{proof}

\begin{example}[Fisher-Rao Newton's flow of KL divergence]
Suppose that $E(\rho)$ evaluates the KL divergence from $\rho$ to $\rho^*\sim\exp(-f)$. This objective functional also writes
\begin{equation*}
E(\rho) = \int (\rho\log \rho+f\rho )dx.
\end{equation*}
We derive that
\begin{equation*}
\frac{\delta E}{\delta \rho}(x) = \log \rho(x)+f+1,\quad \frac{\delta^2 E}{\delta \rho ^2}(x,y) = \frac{\delta (x-y)}{\rho(y)}.
\end{equation*}
Based on Proposition \ref{prop:fnf}, we can compute that \eqref{equ:php} is equivalent to
\begin{equation*}
\begin{aligned}
\int \Phi \mcH_E(\rho ) \Phi  dx=&\frac{1}{2} \int \pp{\Phi ^2-2\mbE_{\rho }[\Phi ]\Phi }\pp{\log \rho  +f -\mbE_{\rho }[\log \rho +f]}\rho  dx\\
&+\int \pp{\Phi (x)-\mbE_{\rho }[\Phi ]}\rho (x)\int \frac{\delta(y-x)}{\rho (y)}\pp{\Phi (y)-\mbE_{\rho }[\Phi ]}\rho (y)dy dx\\
=&\frac{1}{2}\int \pp{\log \rho  +f -\mbE_{\rho }[\log \rho +f]}\Phi ^2\rho dx\\
&- \mbE_{\rho}[\Phi] \int\pp{\log \rho  +f -\mbE_{\rho }[\log \rho +f]}\Phi \rho dx\\
&+ \int \Phi ^2\rho dx-\pp{\int \Phi \rho dx}^2.
\end{aligned}
\end{equation*}
Hence, the operator $\mcH^F_E(\rho)$ follows
\begin{equation*}
\begin{aligned}
\mcH^F_E(\rho)\Phi = &\frac{1}{2}\pp{\log \rho +f -\mbE_{\rho}[\log \rho+f]}\Phi\rho-\frac{1}{2}\pp{\int\pp{\log \rho +f -\mbE_{\rho}[\log \rho+f]}\Phi\rho dy}\rho\\
&-\frac{1}{2}\mbE_{\rho}[\Phi]\pp{\log \rho +f -\mbE_{\rho}[\log \rho+f]}\rho+\Phi\rho-\mbE_{\rho}[\Phi]\rho\\
=&\frac{1}{2}\pp{2+\log \rho +f -\mbE_{\rho}[\log \rho+f]}\pp{\Phi-\mbE_{\rho}[\Phi]}\rho\\
&-\frac{1}{2}\pp{\mbE_{\rho}[\Phi(\log\rho +f)]-\mbE_{\rho}[\Phi]\mbE_{\rho}[(\log\rho +f)]}\rho.\\
\end{aligned}
\end{equation*}
\end{example}

\begin{example}[Fisher-Rao Newton's flow of interaction energy]
Consider an interaction energy
\begin{equation*}
E(\rho) = \frac{1}{2}\int\int \rho(x)W(x,y)\rho(y)dxdy,
\end{equation*}
where $W(x,y)=W(y,x)$ is a kernel function. The interaction energy also formulates the MMD, see details in \citep{aktst}. We can compute that
\begin{equation*}
    \frac{\delta E}{\delta \rho}(x) = \int W(x,y)\rho(y)dy,\quad \frac{\delta^2 E}{\delta \rho^2}(x,y) = W(x,y).
\end{equation*}
We denote $(W*\rho)(x)=\int W(x,y)\rho(y)dx$. Based on Proposition \ref{prop:fnf}, it follows
\begin{equation*}
\begin{aligned}
&\int \Phi \mcH^F_E(\rho)\Phi dx \\
= &\frac{1}{2}\int (\Phi^2-2\mbE_\rho[\Phi]\Phi)(W*\rho -\mbE_\rho[W*\rho])\rho dx\\
&+\int\int  (\Phi(y)-\mbE_\rho[\Phi])W(x,y)\rho(y)\rho(x)(\Phi(x)-\mbE_\rho[\Phi]) dydx\\
=&\frac{1}{2}\int \Phi^2(W*\rho -\mbE_\rho[W*\rho])\rho dx-\mbE_\rho[\Phi]\pp{\int \Phi(W*\rho -\mbE_\rho[W*\rho])\rho dx}\\ 
&+\int\int \Phi(x)\rho(x)W(x,y)\Phi(y)\rho(y)dxdy+\pp{\mbE_\rho[\Phi]}^2\pp{\int\int \rho(x)W(x,y)\rho(y)dxdy}\\
&-2\mbE_\rho[\Phi]\pp{\int\int\rho(x)W(x,y)\Phi(x)\rho(y)dxdy}.
\end{aligned}
\end{equation*}
Hence, the operator $\mcH^F_E(\rho)$ satisfies
\begin{equation*}
\begin{aligned}
\mcH^F_E(\rho)\Phi(x) = &\frac{1}{2}(W*\rho -\mbE_\rho[W*\rho])\rho\Phi - \frac{1}{2}\pp{\int \Phi(W*\rho -\mbE_\rho[W*\rho])\rho dy}\rho \\
&-\frac{1}{2}\mbE_\rho[\Phi](W*\rho -\mbE_\rho[W*\rho])\rho+(W*(\rho\Phi))\rho\\
&+\mbE_{\rho}[W*\rho]\mbE_\rho[\Phi]\rho-\mbE_{\rho}[W*(\rho\Phi)]\rho-\mbE_\rho[\Phi](W*\rho)\rho\\
=&\frac{1}{2}(W*\rho -\mbE_\rho[W*\rho])(\Phi-\mbE_\rho[\Phi])\rho-\frac{1}{2}\pp{\mbE_\rho[\Phi (W*\rho)]-\mbE_\rho[\Phi]\mbE_\rho[W*\rho]}\rho\\
&+(W*(\rho\Phi)-\mbE_{\rho}[W*(\rho\Phi)])\rho-\mbE_\rho[\Phi]\pp{(W*\rho)-\mbE_\rho[W*\rho]}\rho.
\end{aligned}
\end{equation*}
\end{example}

\begin{example}[Fisher-Rao Newton's flow of cross entropy]
Suppose that $E(\rho)$ is the cross entropy, i.e., reverse KL divergence. It evaluates the KL divergence from a given density $\rho^*$ to $\rho$
\begin{equation*}
E(\rho) =-\int \log \pp{\frac{\rho}{\rho^*}}\rho^* dx = -\int (\log \rho) \rho^* dx +\int (\log \rho^*) \rho^* dx.
\end{equation*}
It is equivalent to optimize $E(\rho) = -\int (\log \rho) \rho^* dx$.  We compute that
\begin{equation*}
    \frac{\delta E}{\delta \rho}(x) = -\frac{\rho^*(x)}{\rho(x)},\quad \frac{\delta^2 E}{\delta \rho^2}(x,y) = \frac{\rho^*(y)}{\rho^2(y)}\delta(x-y).
\end{equation*}
Proposition \ref{prop:fnf} indicates that
\begin{equation*}
\begin{aligned}
\int \Phi \mcH^F_E(\rho)\Phi dx=&\frac{1}{2}\int (\Phi^2-2\mbE_\rho[\Phi]\Phi)(-\rho^*/\rho +\mbE_\rho[\rho^*/\rho])\rho dx\\
&+\int (\Phi(x)-\mbE_\rho[\Phi])\rho(x)\int \frac{\rho^*(y)}{\rho^2(y)}\delta(x-y)(\Phi(y)-\mbE_\rho[\Phi]) \rho(y)dy dx\\
=&\frac{1}{2}\int \Phi^2(\rho-\rho^*) dx-\mbE_\rho[\Phi]\int \Phi(\rho-\rho^*) dx+\int (\Phi-\mbE_\rho[\Phi])^2\rho^* dx\\
=&\frac{1}{2}(\mbE_\rho[\Phi^2]-\mbE_{\rho^*}[\Phi^2])-\mbE_\rho[\Phi](\mbE_\rho[\Phi]-\mbE_{\rho^*}[\Phi])\\
&+\mbE_{\rho^*}[\Phi^2]-2\mbE_\rho[\Phi]\mbE_{\rho^*}[\Phi]+\pp{\mbE_\rho[\Phi]}^2\\
=&\frac{1}{2}(\mbE_\rho[\Phi^2]+\mbE_{\rho^*}[\Phi^2])-\mbE_\rho[\Phi]\mbE_{\rho^*}[\Phi].
\end{aligned}
\end{equation*}
Hence, the operator $\mcH^F_E(\rho)$ follows
\begin{equation*}
    \mcH^F_E(\rho)\Phi = \frac{1}{2}((\Phi-\mbE_{\rho^*}[\Phi])\rho+(\Phi-\mbE_\rho[\Phi])\rho^*).
\end{equation*}
\end{example}

\subsection{Newton's flows under Wasserstein metric}
\label{ssec:wnf}
For Wasserstein metric, the geodesic curve $\hat \rho_s$ satisfies
\begin{equation*}
\lbb{&\p_s\hat\rho_s+\nabla\cdot(\hat\rho_s\nabla \Phi_s)= 0,\\
&\p_s\Phi_s+\frac{1}{2}\|\nabla\Phi_s\|^2=0.}
\end{equation*}
The bi-linear operator $\mcA^W(\rho)$ follows
\begin{equation*}
\mcA^W( \rho)(\Phi_1,\Phi_2) = \la \nabla \Phi_1, \nabla \Phi_2\ra . 
\end{equation*}

\begin{proposition}[Wasserstein Newton's flow]\label{prop:wnf}
For an objective functional $E\colon\mathcal{P}(\Omega)\rightarrow \mathbb{R}$, the Wasserstein Newton's flow follows
\begin{equation}
\lbb{
&\p_t\rho_t+\nabla\cdot(\rho\nabla\Phi_t) = 0,\\
&\mcH_E^W(\rho_t)\Phi_t - \nabla\cdot\pp{\rho_t\nabla\frac{\delta E}{\delta {\rho_t}}}=0.
}
\end{equation}
Here $\mcH_E(\rho):T_{\rho}^*\mcP(\Omega)\to T_\rho\mcP(\Omega)$ defines a bi-linear form: for $\Phi\in T_\rho^*\mcP(\Omega)$,
\begin{equation}
\begin{split}
    \int \Phi\mcH_E^W(\rho)\Phi dx=&\int\int \lra{\nabla\Phi(x), \nabla_x\nabla_y\frac{\delta^2 E}{\delta\rho^2}(x,y)\nabla\Phi(y)} \rho(x)\rho(y)dxdy\\
&+\int \lra{\nabla \Phi, \nabla^2\frac{\delta E}{\delta  \rho}\nabla \Phi} \rho dx.
\end{split}
\end{equation}
\end{proposition}
\begin{proof}
Based on integration by parts, we observe that
\begin{equation*}
\begin{aligned}
&\int \mcA^W( \rho)(\Phi,\Phi)\mcG^W( \rho)^{-1}\frac{\delta E}{\delta  \rho} dx\\
=&-\int \|\nabla\Phi\|^2\nabla\cdot\pp{\rho \nabla \frac{\delta E}{\delta  \rho}} dx\\
=&\int \lra{\nabla \frac{\delta E}{\delta  \rho}, \nabla\| \nabla \Phi\|^2} \rho dx\\
=&2\int \lra{\nabla \frac{\delta E}{\delta  \rho}, \nabla^2\Phi \nabla \Phi} \rho dx,
\end{aligned}
\end{equation*}
and 
\begin{equation*}
\begin{aligned}
&\int   \mcA^W( \rho)\pp{\Phi,\frac{\delta E}{\delta  \rho}} \mcG^W( \rho)^{-1}\Phi dx\\
=&-\int \lra {\nabla \Phi, \nabla \frac{\delta E}{\delta  \rho}} \nabla \cdot( \rho\nabla \Phi)dx\\
=&\int \lra{\nabla\lra {\nabla \Phi, \nabla \frac{\delta E}{\delta  \rho}}, \nabla \Phi} \rho dx\\
=&\int \lra{\nabla\Phi, \nabla^2\Phi \nabla \frac{\delta E}{\delta  \rho}} \rho dx+\int \lra{\nabla \Phi, \nabla^2\frac{\delta E}{\delta  \rho}\nabla \Phi} \rho dx.
\end{aligned}
\end{equation*}
Combining above two observations with Proposition \ref{prop:php}, we derive 
\begin{equation*}
\begin{aligned}
    \int \Phi\mcH_E^W( \rho)\Phi dx =& \int \lra{\nabla \Phi, \nabla^2\frac{\delta E}{\delta  \rho}\nabla \Phi} \rho dx \\
    +&\int\int \nabla\cdot( \rho\nabla \Phi)(y)\frac{\delta^2 E}{\delta  \rho^2}(x,y) \nabla\cdot(\rho\nabla \Phi)(x) dxdy\\
    =&\int \int\lra{\nabla \Phi(x),\nabla_x\nabla_y\frac{\delta^2 E}{\delta  \rho^2}(x,y) \nabla \Phi(y)} \rho(x)\rho(y) dxdy\\
    &+\int \lra{\nabla \Phi, \nabla^2\frac{\delta E}{\delta  \rho}\nabla \Phi} \rho dx.
\end{aligned}
\end{equation*}
This proves Proposition \ref{prop:wnf}. 
\end{proof}

\begin{example}[Wasserstein Newton's flow of interaction energy]
Consider an interaction energy
\begin{equation*}
E(\rho) = \frac{1}{2}\int\int \rho(x)W(x,y)\rho(y)dxdy.
\end{equation*}
Combining with previous computations, Proposition \ref{prop:wnf} yields that
\begin{equation*}
\begin{aligned}
\int \Phi \mcH^W_E(\rho)\Phi dx = & \int \int\lra{\nabla \Phi(x),\nabla_x\nabla_yW(x,y) \nabla \Phi(y)} \rho(x)\rho(y) dxdy\\
    &+\int \lra{\nabla \Phi(x), \int \nabla_x^2W(x,y)\rho(y)dy \nabla \Phi(x)} \rho(x) dx\\
=&\frac{1}{2}\mbE_{x,y\sim\rho}\bmbm{\nabla\Phi(x)\\\nabla\Phi(y)}^T\bmbm{\nabla^2_{xx}W(x,y)&\nabla^2_{xy}W(x,y)\\\nabla^2_{yx}W(x,y)&\nabla^2_{yy}W(x,y)}\bmbm{\nabla\Phi(x)\\\nabla\Phi(y)}.
\end{aligned}
\end{equation*}
Based on integration by parts, the operator $\mcH_E^W(\rho)$ is given by
\begin{equation*}
\mcH^W_E(\rho)\Phi= -\nabla\cdot(\rho (\nabla_{xy}^2W*(\rho\nabla\Phi)))-\nabla\cdot(\rho (\nabla^2_{xx} W*\rho) \nabla\Phi).
\end{equation*}
\end{example}

\begin{example}[Wasserstein Newton's flow of cross entropy]
Suppose that $E(\rho)$ evaluates the KL divergence from a given density $\rho^*$ to $\rho$
\begin{equation*}
E(\rho) =-\int \log \pp{\frac{\rho}{\rho^*}}\rho^* dx = -\int (\log \rho) \rho^* dx +\int (\log \rho^*) \rho^* dx.
\end{equation*}
It is equivalent to optimize $E(\rho) = -\int (\log \rho) \rho^* dx$. 
Proposition \ref{prop:wnf} yields
\begin{equation*}
\begin{aligned}
\int \Phi \mcH^W_E(\rho)\Phi dx = & \int \nabla\cdot(\rho(x)\nabla \Phi(x))\int \frac{\rho^*(y)}{\rho^2(y)}\delta(x-y) \nabla\cdot(\rho(y)\nabla \Phi(y))dydx\\
    &-\int \lra{\nabla \Phi(x),  \nabla_x^2\pp{\frac{\rho^*(x)}{\rho(x)}} \nabla \Phi(x)} \rho(x) dx\\
=&\int (\rho^{-1}\nabla\cdot(\rho\nabla \Phi))^2\rho^* dx-\int \lra{\nabla \Phi, \nabla^2\pp{\frac{\rho^*}{\rho}} \nabla \Phi} \rho dx.
\end{aligned}
\end{equation*}
Hence, the operator $\mcH^W_E(\rho)$ satisfies
\begin{equation*}
\mcH_E^W(\rho)\Phi = \nabla\cdot \pp{\rho\nabla\pp{\frac{\rho^*}{\rho^2}\nabla\cdot(\rho\nabla \Phi)}}+\nabla\cdot\pp{\rho\nabla^2\pp{\frac{\rho^*}{\rho}}\nabla\Phi}.
\end{equation*}
\end{example}
\begin{remark}
For simplicity of presentations, we only present the Hessian formulas for Fisher-Rao and Wasserstein information metrics. In fact, there are many interesting generalized Hessian formulas in \cite{DBLP:journals/corr/abs-1907-12546} from Hessian transport metrics. We leave systematic studies of Newton's flows for general metrics in future works. 
\end{remark}

We summarize formulations of Hessian-related operators $\mcH_E(\rho)$ under both Fisher-Rao metric and Wasserstein metric. 
\begin{table}[htbp]
    \centering
    \begin{tabular}{|c|c|}
    \hline
        Objective functional $E(\rho)$&  $\mcH_E^F(\rho)\Phi$ \\\hline
        \makecell*[c]{KL divergence:\\$\int (\rho\log\rho+ f\rho )dx.$}
        &  $\begin{aligned}
&\frac{1}{2}\pp{2+\log \rho +f -\mbE_{\rho}[\log \rho+f]}\pp{\Phi-\mbE_{\rho}[\Phi]}\rho\\
-&\frac{1}{2}\pp{\mbE_{\rho}[\Phi(\log\rho +f)]-\mbE_{\rho}[\Phi]\mbE_{\rho}[(\log\rho +f)]}\rho.
\end{aligned}$
\\\hline
\makecell*[c]{Interaction energy:\\
$\frac{1}{2}\int\int \rho(x)W(x,y)\rho(y)dxdy$}
&$\begin{aligned}
&\frac{1}{2}(W*\rho -\mbE_\rho[W*\rho])(\Phi-\mbE_\rho[\Phi])\rho\\
-&\frac{1}{2}\pp{\mbE_\rho[\Phi (W*\rho)]-\mbE_\rho[\Phi]\mbE_\rho[W*\rho]}\rho\\
+&(W*(\rho\Phi)-\mbE_{\rho}[W*(\rho\Phi)])\rho\\
-&\mbE_\rho[\Phi]\pp{(W*\rho)-\mbE_\rho[W*\rho]}\rho.\end{aligned}$
\\\hline
\makecell*[c]{Reverse KL divergence:\\$\int (\log{\rho^*}-\log \rho) {\rho^*} dx$}
&$\frac{1}{2}(\Phi-\mbE_{\rho^*}[\Phi])\rho+\frac{1}{2}(\Phi-\mbE_\rho[\Phi]){\rho^*}.$
\\\hline
    \end{tabular}
    \caption{The formulation of $\mcH_E^F(\rho)$ under the Fisher-Rao metric.}
    \label{tab:fr_n}
\end{table}

\begin{table}[htbp]
    \centering
    \begin{tabular}{|c|c|}
    \hline
         Objective functional $E(\rho)$& $\mcH_E^W (\rho)\Phi$ \\\hline
         \makecell*[c]{KL divergence:\\$\int (\rho\log\rho+ f\rho )dx.$}
         &$ \nabla^2:(\rho\nabla^2\Phi)-\nabla\cdot(\rho\nabla^2f\nabla \Phi).$
\\\hline
 \makecell*[c]{Interaction energy:\\
$\frac{1}{2}\int\int \rho(x)W(x,y)\rho(y)dxdy$}
&$-\nabla\cdot(  (\nabla_{xy}^2W*(\nabla\Phi\rho)) \rho)-\nabla\cdot( (\nabla^2_{xx} W*\rho_t)\rho \nabla\Phi).$
\\\hline
\makecell*[c]{Reverse KL divergence:\\$\int (\log{\rho^*}-\log \rho) \mu dx$}
&$\nabla\cdot \pp{\rho\nabla\pp{\frac{{\rho^*}}{\rho^2}\nabla\cdot(\rho\nabla \Phi)}}+\nabla\cdot\pp{\rho\nabla^2\pp{\frac{{\rho^*}}{\rho}}\nabla\Phi}.$
\\\hline
    \end{tabular}
    \caption{The formulation of $\mcH_E^W(\rho)$ under the Wasserstein metric.}
    \label{tab:w_n}
\end{table}

\subsection{Wasserstein Newton's flows in Gaussian families}
In this subsection, we study information Newton's flows in Gaussian families with respect to Wasserstein metric. We leave the proof of Proposition \ref{prop:exist} in next subsection. Let $\mbP^n$ and $\mbS^n$ represent the space of symmetric positive definite matrices and symmetric matrices with size $n\times n$ respectively.

We let $\mcN_n^0$ denote multivariate Gaussian densities with zero means. Each $\rho\in \mcN_n^0$ is uniquely determined by its covariance matrix $\Sigma\in \mbP^n$. So we can view $\mcN_n^0\simeq \mbP^n$.   
The Wasserstein metric $\mcG^W(\rho)$ on $\mcP(\mbR^d)$ induces the Wasserstein metric $\mcG^W(\Sigma)$ on $\mbP^n$, 
see \citep{owgot,gomds,wrgop}. For $\Sigma\in \mbP^n$, tangent space and cotangent space follow $$T_{\Sigma}\mbP^n\simeq T^*_{\Sigma}\mbP^n\simeq \mbS^n.$$

\begin{definition}[Wasserstein metric in Gaussian families]
Given $\Sigma\in \mbP^n$, the Wasserstein metric tensor $\mcG^W(\Sigma):\mbS^n\to \mbS^n$ is defined by $$\mcG^W(\Sigma)^{-1}S = 2(\Sigma S+S\Sigma).$$ 
It defines an inner product on the tangent space $T_\Sigma\mbP^n$. Namely, for $A_1,A_2\in T_\Sigma\mbP^n\simeq \mbS^n$
$$
g^W_{\Sigma}(A_1,A_2) = \tr(A_1\mcG^W(\Sigma)A_2)=\tr(S_1\mcG^W(\Sigma)^{-1}S_2)= 4\tr(S_1\Sigma S_2).
$$
Here $S_i\in T_\Sigma^*\mbP^n\simeq \mbS^n$ is the solution to discrete Lyapunov equation
\begin{equation*}
    A_i = 2(\Sigma S_i+S_i\Sigma), \quad i=1,2.
\end{equation*}
\end{definition}
For $\Sigma\in \mbP^n$, there exits a unique solution to discrete Lyapunov equation. Again, we focus on the case where the objective functional $E(\Sigma)$ evaluates the KL divergence from $\rho$ with covariance matrix $\Sigma$ to a target Gaussian density $\rho^*$ with covariance matrix $\Sigma^*$. Then, $E(\Sigma)$ satisfies \eqref{esigma}.
\begin{proposition}[Gradient and Hessian operators in $\mbP^n$]
The gradient operator follows
$$
\grad^W E(\Sigma) = \mcG^W(\Sigma)^{-1}\nabla E(\Sigma) = \Sigma (\Sigma^*)^{-1}+(\Sigma^*)^{-1}\Sigma-2I.
$$
And the Hessian operator satisfies that for all $A\in\mbS^n$,
$$
g_\Sigma^W(A,\Hess^W E(\Sigma) A) = 4\tr(S\Sigma S(\Sigma^*)^{-1})+4\tr(S^2),
$$
where $S$ is the unique solution to $A = 2(\Sigma S+S\Sigma)$.
\end{proposition}
Given $A \in \mbS^n$, the geodesic curve $\hat \Sigma_s$ with $\hat\Sigma_s|_{s=0}=\Sigma$ and $\p_s\hat\Sigma_s|_{s=0}=A$ follows $\hat\Sigma_s = (I+2sS)\Sigma(I+2sS)$, where $S=\mcG(\Sigma)^{-1}A$ is the solution to $A = 2(\Sigma S+S\Sigma)$. We can compute that 
$$
E(\hat\Sigma_s)=\frac{1}{2}(\tr((I+2sS)\Sigma (I+2sS)(\Sigma^*)^{-1})-n-\log\det((I+2sS)\Sigma (I+2sS)(\Sigma^*)^{-1})).$$
The Taylor expansion of $\log\det(I+sS)$ w.r.t. $s$ satisfies
$$
\log\det(I+sS) = s\tr(S)-\frac{s^2}{2}\tr(S^2)+o(s^2).
$$
Hence, the first-order derivative follows
$$
\begin{aligned}
\left.\frac{\p }{\p s} E(\Sigma(s)) \right|_{s=0}= &\tr(S\Sigma (\Sigma^*)^{-1})+\tr(\Sigma S(\Sigma^*)^{-1})-2\tr(S)\\
=&\tr\pp{S\pp{\Sigma (\Sigma^*)^{-1}+(\Sigma^*)^{-1}\Sigma-2I}}.
\end{aligned}
$$
By the definition $\left.\frac{\p }{\p s} E(\Sigma(s)) \right|_{s=0}=\tr(S\grad E(\Sigma))$, this yields $\grad E(\Sigma) = \Sigma (\Sigma^*)^{-1}+(\Sigma^*)^{-1}\Sigma-2I$ and the second-order derivative follows
$$
\left.\frac{\p^2 }{\p s^2} E(\Sigma(s)) \right|_{s=0} = 4\tr(S\Sigma S(\Sigma^*)^{-1})+4\tr(S^2).
$$
This completes the proof.

Similarly, let us consider the linear self-adjoint operator $\mcH_E^W(\Sigma):\mbS^n\to \mbS^n$, which defines a bi-linear form
$$
\tr(S\mcH_E^W(\Sigma)S) = g_\Sigma^W(A,\Hess^WE(\Sigma) A) =4\tr(S\Sigma S(\Sigma^*)^{-1})+4\tr(S^2).
$$
We can compute that $\mcH_E(\Sigma)$ is uniquely defined by
$$
\mcH_E(\Sigma)S = 2\Sigma S(\Sigma^*)^{-1}+2(\Sigma^*)^{-1}S\Sigma+4S,\quad \forall S\in \mbS^n.
$$
Because $\tr(S\mcH_E(\Sigma)S) =  4\tr(S\Sigma S(\Sigma^*)^{-1})+4\tr(S^2)>0$ for $S\neq 0, S\in \mbS^n$, $\mcH_E$ is injective and invertible. Now, we are ready to present the Newton's flow of KL divergence in Gaussian families.
\begin{proposition}
The Newton's flow of KL divergence in Gaussian families follows
\begin{equation}\label{equ:wnf_gauss}
\left\{  \begin{aligned}
    &\dot \Sigma_t -2(S\Sigma_t+\Sigma S_t)=0,\\
    &2\Sigma_t S_t(\Sigma^*)^{-1}+2(\Sigma^*)^{-1}S_t\Sigma_t+4S_t = -(\Sigma_t (\Sigma^*)^{-1}+(\Sigma^*)^{-1}\Sigma_t-2I).
    \end{aligned}\right.
\end{equation}
\end{proposition} 
\begin{proof}
The Newton's flow follows
$$
\dot \Sigma_t-(\Hess^W E(\Sigma_t))^{-1}\grad^W E(\Sigma_t)=0. 
$$
We note that $\Hess^W E(\Sigma) \mcG^W(\Sigma)^{-1} = \mcH_E^W(\Sigma)$, which implies
$$
(\Hess^W E(\Sigma))^{-1} = \mcG^W(\Sigma)^{-1}\mcH_E^W(\Sigma)^{-1}.
$$
Hence, we can reformulate the Newton's flow by
\begin{equation*}
\left\{  \begin{aligned}
    &\dot \Sigma_t -\mcG^W(\Sigma_t)^{-1}S_t=0,\\
    &\mcH_E^W(\Sigma_t)S_t=-\grad^W E(\Sigma_t).
    \end{aligned}\right.
\end{equation*}
From the formulations of $\mcG(\Sigma)^{-1}$, $\grad^W E(\Sigma)$ and $\mcH_E(\Sigma)$, we obtain \eqref{equ:wnf_gauss}.
\end{proof}
\begin{example}
In one dimensional case, the second equation in \eqref{equ:wnf_gauss} has an explicit solution $S_t=-\frac{(\Sigma^*)^{-1}\Sigma_t-1}{2((\Sigma^*)^{-1}\Sigma_t+1)}$. Let $\Sigma_t = Y_t^2$, where $Y_t>0$. Then, the first equation in \eqref{equ:wnf_gauss} turns to
$$
2Y_t\dot Y_t+4Y_t^2\frac{(\Sigma^*)^{-1}Y_t^2-1}{2((\Sigma^*)^{-1}Y_t^2+1)}=0,
$$
or equivalently,
\begin{equation}\label{equ:new_gauss_1d}
    \dot Y_t +\frac{(\Sigma^*)^{-1}Y_t-Y_t^{-1}}{(\Sigma^*)^{-1}+Y_t^{-2}}=0.
\end{equation}
Let $f(Y) = \frac{1}{2}((\Sigma^*)^{-1}Y^2-1-\log((\Sigma^*)^{-1}Y^2))$. Then, we have $\nabla f(Y) = (\Sigma^*)^{-1}Y-Y^{-1}$ and $\nabla^2 f(Y) = (\Sigma^*)^{-1}+Y^{-2}$. Hence, the Newton's flow \eqref{equ:new_gauss_1d} coincides with Newton's flow of $f(X)$ in Euclidean space. We also note that \eqref{equ:new_gauss_1d} is identical to the evolution of $\Sigma_t$ in Proposition \ref{prop:gauss1d} by substituting $\Sigma_t = Y_t^2$.
\end{example}

\subsection{Proof of Proposition \ref{prop:exist}}
We first prove that $\Sigma_t$ is positive definite. We formulate that
$$
\begin{aligned}
&\p_tE(\Sigma_t) = \tr(\p_t\Sigma_t\nabla E(\Sigma_t))=2\tr(S_t\Sigma_t((\Sigma^*)^{-1}-\Sigma_t^{-1}))\\
=&\tr(S_t(\Sigma_t (\Sigma^*)^{-1}+(\Sigma^*)^{-1}\Sigma_t-2I))=-\tr(S(2\Sigma S_t(\Sigma^*)^{-1}+2(\Sigma^*)^{-1}S_t\Sigma_t+4S_t))\\
=&-4\tr(S_t\Sigma_t S_t(\Sigma^*)^{-1})-4\tr(S_t^2)\leq 0.
\end{aligned}
$$
As a result, $E(\Sigma_t)$ is non-increasing. Applying the idea of proof in \citep[Theorem 1]{aigf}, we can establish that $\Sigma_t$ is positive definite. Then, we examine that $\Phi_t$ satisfies \eqref{equ:wnf_kl}. We observe that
$$
\begin{aligned}
&\nabla^2:(\rho_t\nabla^2\Phi_t)-\nabla\cdot(\rho_t\nabla^2f\nabla \Phi_t) - \nabla\cdot(\rho_t\nabla f)-\Delta \rho_t\\
=&2\nabla^2:(S_t\rho_t(x))-2\nabla\cdot(\rho_t(x)(\Sigma^*)^{-1}S_tx)-\nabla\cdot(\rho_t(x)(\Sigma^*)^{-1}x)-\Delta \rho_t. 
\end{aligned}
$$
We note that $\nabla\rho_t(x)=-\Sigma_t^{-1}x\rho_t(x)$ and $\nabla^2\rho_t(x)=-\Sigma_t^{-1}\rho_t(x)+\Sigma_t^{-1}xx^T\Sigma_t^{-1}\rho_t(x)$.
Hence, we derive all four terms in the above equation as follows. First, it is easy to observe that
$$
\begin{aligned}
&\nabla^2:(\rho_t(x)S_t)=\tr(S_t\nabla^2\rho_t(x)),\quad -\Delta \rho_t=-\nabla^2:(\rho_t I) = -\tr(\nabla^2\rho_t(x)).
\end{aligned}
$$
We can also compute that
$$
\begin{aligned}
&-\nabla\cdot(\rho_t(\Sigma^*)^{-1}S_tx)\\
= & -\sum_{i=1}^n \p_i (\rho_t(x) WS_tx)_i\\
=&-\sum_{i=1}^n\lb\rho_t(x) \p_i((\Sigma^*)^{-1}S_tx)_i+(WS_tx)_i\p_i \rho_t(x)\rb\\
=&-\rho_t(x) \lb\tr((\Sigma^*)^{-1}S_t)+((\Sigma^*)^{-1}S_tx)^T(-\Sigma_t^{-1}x)\rb\\
=&-\rho_t(x) \tr(S_t(\Sigma^*)^{-1}(I-\Sigma_t^{-1}xx^T))\\
=&\frac{1}{2}\tr((\Sigma_tS_t(\Sigma^*)^{-1}+(\Sigma^*)^{-1}S_t\Sigma_t)\nabla^2\rho_t(x)).
\end{aligned}
$$
Taking $S_t=I$ into the above equation yields
$$
-\nabla\cdot(\rho_t(\Sigma^*)^{-1}x) = \frac{1}{2}\tr((\Sigma_t(\Sigma^*)^{-1}+(\Sigma^*)^{-1}\Sigma_t)\nabla^2\rho_t(x)).
$$
Because $(\Sigma_t,S_t)$ satisfies \eqref{equ:wnf_gauss}, we have
$$
\begin{aligned}
&2\nabla^2:(S_t\rho_t(x))-2\nabla\cdot(\rho_t(x)(\Sigma^*)^{-1}S_tx)-\nabla\cdot(\rho_t(x)(\Sigma^*)^{-1}x)-\Delta \rho_t\\
=&\tr((2S_t+\Sigma_tS_t(\Sigma^*)^{-1}+(\Sigma^*)^{-1}S_t\Sigma_t+\Sigma_t(\Sigma^*)^{-1}+(\Sigma^*)^{-1}\Sigma_t-2I)\nabla^2\rho_t(x))\\
=&0.
\end{aligned}
$$
This completes the proof.

\section{Details in section \ref{sec:nld}}
In this section, we present detailed discussion of Wasserstein Newton's flow and Newton's Langevin dynamics with particular examples.
\subsection{Connections and differences with HAMCMC}
\label{ssec:cmp_hamcmc}
HAMCMC approximates the dynamics of 
$$
dX_t = -(\nabla^2 f(X_t))^{-1}\pp{\nabla f(X_t)+\Gamma(X_t)}dt+\sqrt{2\nabla^2 f(X_t)^{-1}}dB_t,
$$
where $\Gamma_i(x) =\sum_{j=1}\frac{\p}{\p x_j} \pp{\pp{\nabla^2 f(X_t)}^{-1}}_{i,j} $. Here $\Gamma(x)$ is a correction term to ensure that $\rho_t$ converges to $\rho^*$. 
The evolution of $\rho_t$ follows
$$
\p_t\rho_t = \nabla\cdot\pp{\pp{(\nabla^2 f)^{-1}\nabla f+\Gamma}\rho_t}+\nabla^2:\pp{(\nabla^2 f)^{-1}\rho_t}.
$$
We formulate the above equation as
\begin{equation}\label{equ:sqn}
\p_t\rho_t = \nabla\cdot(\rho_t(\nabla^2 f)^{-1}(\nabla f+\nabla\log \rho_t))=\nabla\cdot(\rho_t\mfv_t),
\end{equation}
where we denote $\mfv_t=(\nabla^2 f)^{-1}(\nabla f+\nabla\log \rho_t)$. 
Moreover, $\mfv_t$ satisfies
$$
-\nabla\cdot(\rho_t\nabla^2f\mfv_t) - \nabla\cdot(\rho_t\nabla f)-\Delta \rho_t=0.
$$
On the other hand, replacing $\nabla\Phi_t^{\operatorname{Newton}}$ by $\mfv_t^{\operatorname{Newton}}$ in the Newton's direction equation \eqref{equ:wnf_kl} yields
$$
\nabla^2:(\rho_t\nabla \mfv_t^{\operatorname{Newton}})-\nabla\cdot(\rho_t\nabla^2f\mfv_t^{\operatorname{Newton}}) - \nabla\cdot(\rho_t\nabla f)-\Delta \rho_t=0.
$$
Hence, $\mfv_t\neq \mfv_t^{\operatorname{Newton}}$. Then, \eqref{equ:sqn} is different from information Newton's flow because the term $\nabla^2:(\rho_t\nabla \mfv_t^{\operatorname{Newton}})$ is not considered. 

\subsection{Connections and differences with Newton's flows in Euclidean space}


We recall that the density evolution of particle's gradient flow in Euclidean space corresponds to the Wasserstein gradient flow \citep{otoan}. We notice that this relationship does not hold for the Wasserstein Newton's flow.

Consider an objective function:
$$E(\rho) = \int \rho(x) f(x)dx,$$
where $f(x)$ is a given smooth function. Here we notice that minimize $\rho$ for $E(\rho)$ in probability space is equivalent to minimize $x$ for $f(x)$ in Euclidean space. Namely, the support of the optimal solution $\rho$ contains all global minimizers of $f(x)$. 
The gradient flow in Euclidean space of each particle follows
$$
dX_t=-\nabla f(X_t)dt,
$$
A known fact is that the density evolution of particles satisfies the following continuity equation
$$
\p_t\rho_t= \nabla \cdot(\rho_t \nabla f)=-\grad^WE(\rho_t),
$$
which is the Wasserstein gradient flow of $E(\rho)$ in probability space.

We next show that Newton's flow in Euclidean space of each particle does not coincide with the Wasserstein Newton's flow in probability space.
For simplicity, we assume that $f(x)$ is strictly convex so $\nabla^2 f(x)$ is invertible for all $x$. 
Here, the Euclidean Newton's flow of each particle follows
$$
dX_t = -(\nabla^2f(X_t))^{-1}\nabla f(X_t)dt.
$$
The density evolution of particles satisfies the continuity equation
\begin{equation}\label{equ:newt_e}
    \p_t\rho_t= \nabla \cdot(\rho_t (\nabla^2 f)^{-1} \nabla f).
\end{equation}
On the other hand, the Wasserstein Newton's flow writes
\begin{equation}\label{equ:newt_w}
    \p_t\rho_t+\nabla\cdot(\rho_t\nabla \Phi_t^{\operatorname{Newton}})=0,
\end{equation} 
where $\Phi_t^{\operatorname{Newton}}$ is the unique solution to  
\begin{equation}\label{equ:wnf_beta}
    -\nabla\cdot(\rho_t\nabla^2f\nabla \Phi) - \nabla\cdot(\rho_t\nabla f)=0.
\end{equation}

We note that in general equation \eqref{equ:newt_e} can be different from equation \eqref{equ:newt_w}. Later on in Lemma \ref{lem:decomp}, we formulate the following Hodge decomposition of the Euclidean Newton's direction 
$$
-(\nabla^2f)^{-1}\nabla f=\nabla\Phi_t^{\operatorname{Newton}}+\bm{\xi}_t,
$$ 
where $\nabla\cdot(\rho_t \nabla^2 f \bm{\xi}_t)=0$. Here, the constraint on $\bm{\xi}_t$ does not necessarily ensure that $\nabla\cdot(\rho_t \bm{\xi}_t)=0$. Hence, equation \eqref{equ:newt_e} can be different from equation \eqref{equ:newt_w}. 

\begin{remark}\label{rmk:xi}
In one dimensional case or $f$ is a quadratic function, there exists $\Phi^{\operatorname{Newton}}$, such that $-(\nabla^2 f)^{-1}\nabla f=\nabla\Phi^{\operatorname{Newton}}$. Hence
equation \eqref{equ:newt_e} is same as equation \eqref{equ:newt_w}. We also show an example of $\bm{\xi}\neq0$. Let $\Omega=\mbR^2$ and we define
$$
f(x) = \log(\exp(x_1)+\exp(x_2))+\frac{\lambda}{2} (x_1^2+x_2^2),
$$
where $\lambda>0$ is a parameter. For simplicity, we denote $p_1=\exp(x_1)/(\exp(x_1)+\exp(x_2))$ and $p_2=\exp(x_1)/(\exp(x_1)+\exp(x_2))$. Then, we can compute that the gradient and Hessian of $f(x)$ follows
$$
\nabla f(x) = \bmbm{p_1+\lambda x_1\\p_2+\lambda x_2},\quad \nabla^2 f(x)=\bmbm{p_1p_2+\lambda&-p_1p_2\\-p_1p_2&p_1p_2+\lambda}.
$$
Because $p_1p_2+\lambda>0$ and $\det (\nabla^2f(x))=\lambda^2+2\lambda p_1p_2>0$, $\nabla^2 f(x)$ is positive definite. We note that
$$
\begin{aligned}
(\nabla^2f(x))^{-1}\nabla f(x) = &\frac{1}{\lambda^2+2\lambda p_1p_2}\bmbm{p_1p_2+\lambda&p_1p_2\\p_1p_2&p_1p_2+\lambda}\bmbm{p_1+\lambda x_1\\p_2+\lambda x_2}\\
=&\frac{1}{\lambda^2+2\lambda p_1p_2}\bmbm{p_1p_2(1+\lambda(x_1+x_2))+\lambda(p_1+\lambda x_1)\\p_1p_2(1+\lambda(x_1+x_2))+\lambda(p_2+\lambda x_2)}\\
=&:\bmbm{F_1(x)\\F_2(x)}.
\end{aligned}
$$
If $(\nabla^2f(x))^{-1}\nabla f(x) $ is a gradient vector field, we shall have
$$
\p_{x_2} F_1(x) = \p_{x_1} F_2(x).
$$
However, we can examine that
\small{
$$
\p_{x_2} F_1(x) = \frac{p_1p_2}{\lambda+2p_1p_2}\pp{1+{p_1(1+\lambda(x_1+x_2))}+\frac{2\lambda p_1(\lambda (p_1+\lambda x_1)+p_1p_2(1+\lambda(x_1+x_2)))}{\lambda^2+2\lambda p_1p_2}}.
$$}
\small{
$$
\p_{x_1} F_2(x) = \frac{p_1p_2}{\lambda+2p_1p_2}\pp{1+{p_2(1+\lambda(x_1+x_2))}+\frac{2\lambda p_2(\lambda (p_2+\lambda x_2)+p_1p_2(1+\lambda(x_1+x_2)))}{\lambda^2+2\lambda p_1p_2}}.
$$}
This indicates that $(\nabla^2f(x))^{-1}\nabla f(x) $ is not a gradient vector field. Hence, $\bm{\xi}\neq 0$. 
 \end{remark}

\begin{lemma}\label{lem:decomp}
For given $\rho\in\mcP(\mbR^d)$, there exists a unique $\Phi\in T_\rho^* \mcP(\mbR^d)$ (up to a constant shrift) and a vector field $\bm{\xi}:\mbR^d\to\mbR^d$ satisfying $\nabla\cdot(\rho \nabla^2 f \bm{\xi})=0$ such that
$$-(\nabla^2f(x))^{-1}\nabla f(x)=\nabla\Phi(x)+\bm{\xi}(x).$$
\end{lemma}
\begin{proof}
We first show the existence of $\Phi\in T_\rho^* \mcP(\mbR^d)$ and $\bm{\xi}$. Note that $\Phi$ is the solution to
$$
-\nabla\cdot(\rho\nabla^2f\nabla\Phi)=\nabla\cdot(\rho\nabla f).
$$
Denote $\mcH\Phi = -\nabla\cdot(\rho\nabla^2f\nabla\Phi)$. Then, for $\Phi\neq0$, we have
$$
\int \Phi \mcH\Phi dx = \int \nabla \Phi^T\nabla^2f\nabla\Phi \rho dx>0.
$$
Hence, $\mcH$ is a positive definite operator and it is invertible. Thus $\Phi=\mcH^{-1}\pp{\nabla\cdot(\rho\nabla f)}$ exists. Because $\nabla^2 f\bm{\xi}=\nabla f-\nabla^2f\nabla\Phi$, it follows
$$
\nabla\cdot(\nabla^2 f\bm{\xi}) = \nabla\cdot(\rho\nabla f)-\nabla\cdot(\rho\nabla^2f\nabla\Phi)=0.
$$
Hence, $\bm{\xi}$ also exists. We next prove the uniqueness. Suppose that $\nabla^2f(x)^{-1}\nabla f(x)=\nabla\Phi_1(x)+\bm{\xi}_1(x)=\nabla\Phi_2(x)+\bm{\xi}_2(x)$. Then, we have $\nabla \Phi_1-\nabla\Phi_2=\bm{\xi}_2-\bm{\xi}_1$. Hence
$$
\begin{aligned}
&\int (\Phi_1-\Phi_2)\mcH(\Phi_1-\Phi_2) dx = \int  (\nabla\Phi_1-\nabla\Phi_2)^T\nabla^2f(\nabla\Phi_1-\nabla\Phi_2) \rho dx\\
=&\int (\nabla\Phi_1-\nabla\Phi_2)^T\nabla^2f(\bm{\xi}_2-\bm{\xi}_1)\rho dx=-\int (\Phi_1-\Phi_2) \nabla\cdot(\rho\nabla^2 f(\bm{\xi}_2-\bm{\xi}_1) )dx=0.
\end{aligned}
$$
Because $\mcH$ is positive definite, this yields that $\Phi_1-\Phi_2=0$ (up to a spatial constant).
\end{proof}

\subsection{Newton's Langevin dynamics in one dimensional sample space}\label{NLD1D}
In this subsection, we provide examples of Newton's Langevin dynamics in one dimensional sample space. In particular, similar to the Ornstein–Uhlenbeck (OU) process in classical Langevin dynamics, we derive a closed form solution to Newton's OU process.

Here we assume that $\Omega=\mbR$ and $f$ is strictly convex. The essence of Newton's Langevin dynamics is to compute $\Phi_t^{\operatorname{Newton}}$ from the Wasserstein Newton's direction equation \eqref{equ:wnf_kl}. Proposition \ref{prop:unique} ensures the uniqueness of the solution to \eqref{equ:wnf_kl}. For the simplicity of notations, we neglect the subscript $t$. 

\begin{proposition}\label{prop:new_1d}
Suppose that $\rho>0$ and let $u=\nabla \Phi$. Then, the Newton's direction equation \eqref{equ:wnf_kl} reduces to an ODE
\begin{equation}\label{equ:ode1}
    u''+u'(\log\rho)'-f''u-f'-(\log \rho)'=0.
\end{equation}
\end{proposition}
\begin{proof}
In 1-dimensional case, the equation \eqref{equ:wnf_kl} follows
$$
\nabla^2(\rho\nabla^2\Phi)-\nabla(\rho\nabla^2f\nabla \Phi)-\nabla(\rho\nabla f)-\nabla^2\rho=0.
$$
The above equation is equivalent to
$$
\rho\nabla^3\Phi+\nabla\rho\nabla^2\Phi-\rho\nabla^2f\nabla \Phi-\rho\nabla f-\nabla\rho+C=0,
$$
where $C$ is a constant. Because $\rho\in \mcP(\mbR)\subset L^1(\mbR)$. Hence $\lim_{|x|\to\infty}\rho(x)=0$, which indicates $C=0$. Suppose that $\rho>0$ and let $u=\nabla \Phi$. Dividing both sides by $\rho$, we obtain
$$
u''+u'\rho'/\rho-f''u-f'-\rho'/\rho=0.
$$
By the fact that $\rho'/\rho=(\log\rho)'$, we derive \eqref{equ:ode1}.
\end{proof}

We consider the case where $f'(x)$ and $(\log \rho)'(x)$ are affine functions. Then, ODE \eqref{equ:ode1} has a closed-form solution.  Applying ODE \eqref{equ:ode1}, we obtain the exact formulation of Newton's Langevin dynamics in Proposition \ref{prop:gauss1d}. For the rest of this section, we present the proof of Proposition \ref{prop:gauss1d}.

\begin{proof}
In section \ref{section3} Proposition \ref{prop:exist}, we show that if the evolution of $X_t$ follows NLD, then $X_t$ follows the Gaussian distribution. We first solve the Newton's direction from ODE \eqref{equ:ode1}. Suppose that $(\log\rho)'(x) = \Sigma^{-1}(x-\mu)$. The ODE turns to be
$$
u''-u'\Sigma^{-1}(x-\mu)-(\Sigma^*)^{-1}u-(\Sigma^*)^{-1}(x-\mu^*)+\Sigma^{-1}(x-\mu)=0.
$$
We can examine that the following $u$ is a solution to the above ODE.
$$u(x)=\frac{\Sigma^{-1}-(\Sigma^*)^{-1}}{\Sigma^{-1}+(\Sigma^*)^{-1}}x-\frac{2\Sigma^{-1}}{\Sigma^{-1}+(\Sigma^*)^{-1}}\mu+\mu^*.$$
Hence, we have $\Phi^{\operatorname{Newton}}(x)=\frac{\Sigma^*-\Sigma}{2(\Sigma^*+\Sigma)}x^2-\frac{2\Sigma^*}{\Sigma^*+\Sigma}\mu x+\mu^*x.$
As a result, NLD follows
$$
dX_t = \pp{\frac{\Sigma^*-\Sigma_t}{\Sigma^*+\Sigma_t}X_t-\frac{2\Sigma^*}{\Sigma^*+\Sigma_t}\mu_t+\mu^*}dt.
$$
The dynamics of $\mu_t$ satisfies $$
d\mu_t = d\mbE[X_t]=\mbE[dX_t] = \pp{\frac{\Sigma^*-\Sigma_t}{\Sigma^*+\Sigma_t}\mu_t-\frac{2\Sigma^*}{\Sigma^*+\Sigma_t}\mu_t+\mu^*}dt=(-\mu_t+\mu^*)dt.
$$ 
This indicates that  $\mu_t=\mu^*+e^{-t}(\mu_0-\mu^*)$. The dynamics of $\Sigma_t$ follows
$$
\begin{aligned}
&d\Sigma_t = d(\mbE[X_t^2]- \mu_t^2)=2\mbE[X_tdX_t]-2\mu_td\mu_t\\
=&2\bb{\frac{\Sigma^*-\Sigma_t}{\Sigma^*+\Sigma_t}\pp{\Sigma_t+\mu_t^2}-\frac{2\Sigma^*}{\Sigma^*+\Sigma_t}\mu_t^2+\mu^*\mu_t-\mu_t(-\mu_t+\mu^*)}dt=2\frac{\Sigma^*-\Sigma_t}{\Sigma^*+\Sigma_t}\Sigma_tdt.
\end{aligned}
$$
We can rewrite that
$$
dt = \frac{(\Sigma^*+\Sigma_t)d\Sigma_t}{2(\Sigma^*-\Sigma_t)\Sigma_t}=\pp{\frac{1}{(\Sigma^*-\Sigma_t)}+\frac{1}{2\Sigma_t}}d\Sigma_t.
$$
Integrating both sides of the above equation yields
$$
t-\log |\Sigma^*-\Sigma_0|+\frac{1}{2}\log\Sigma_0=-\log |\Sigma^*-\Sigma_t|+\frac{1}{2}\log\Sigma_t,\quad (\Sigma_t-\Sigma^*)^2 = \frac{(\Sigma_0-\Sigma^*)^2}{\Sigma_0}e^{-2t}\Sigma_t.
$$
Hence, the solution $\Sigma_t$ follows
$$
\Sigma_t=\Sigma^*+\frac{e^{-2t}(\Sigma_0-\Sigma^*)^2}{2\Sigma_0}+(\Sigma_0-\Sigma^*)e^{-t}\sqrt{\frac{e^{-2t}(\Sigma_0-\Sigma^*)^2}{4\Sigma_0^2}+\frac{\Sigma^*}{\Sigma_0}}.
$$
\end{proof}

Now, we are ready to compare the NLD with OLD, LLD and HAMCMC. Here we consider $f(x)=(2\Sigma^*)^{-1}(x-\mu^*)^2$, where $\Sigma^*>0$ and $\mu^*$ are given. The OLD satisfies
$$
dX_t = -(\Sigma^*)^{-1}(X_t-\mu^*)dt+\sqrt{2}dB_t,
$$
which is also known as the Ornstein-Uhlenbeck process. 
And LLD writes
$$
dX_t =  -(\Sigma^*)^{-1}(X_t-\mu^*)dt+\Sigma_t^{-1}(X_t-\mu_t)dt.
$$
The mean $\mu_t$ and variance $\Sigma_t$ of OLD and LLD both satisfy
$$
\mu_t =\mu^*+e^{-(\Sigma^*)^{-1}t}(\mu_0-\mu^*),\quad \Sigma_t=\Sigma^*+e^{-2(\Sigma^*)^{-1}t}(\Sigma_0-\Sigma^*).$$
On the other hand, HAMCMC follows the dynamics
$$
dX_t = -(X_t-\mu^*)dt+\sqrt{2\Sigma^*} dB_t.
$$
For HAMCMC, the evolution of mean $\mu_t$ follows
$$
d\mu_t = d\mbE[X_t] = -(\mu_t-\mu^*)dt,
$$
and the evolution of variance $\Sigma_t$ satisfies
$$
\begin{aligned}
&d\Sigma_t = d(\mbE[X_t^2]- \mu_t^2)=2\mbE[X_tdX_t]-2\mu_td\mu_t\\
=&2\bb{-\pp{\Sigma_t+\mu_t^2}+\mu^*\mu_t+\Sigma^*+\mu_t(\mu_t-\mu^*)}dt=2(\Sigma^*-\Sigma_t)dt.
\end{aligned}
$$
The mean $\mu_t$ and variance $\Sigma_t$ of HAMCMC follows
$$
\mu_t = \mu^*+e^{-t}(\mu_0-\mu^*),\quad \Sigma_t = \Sigma^*+e^{-2t}(\Sigma_0-\Sigma^*).
$$

We summarize our results in Table \ref{tab:my_label}. 
\begin{table}[ht]
    \centering
    \small{\setlength{\tabcolsep}{1mm}{
    \begin{tabular}{|c|c|c|c|}
    \hline
         Dynamics&Particle&Mean and variance  \\\hline
         \multirow{3}*{NLD}&\multirow{3}*{$dX_t = \pp{\frac{\Sigma^*-\Sigma_t}{\Sigma^*+\Sigma_t}X_t-\frac{2\Sigma^*}{\Sigma^*+\Sigma_t}\mu_t+\mu^*}dt$}&$\mu_t = \mu^*+e^{-t}(\mu_0-\mu^*)$\\
         &&{$\frac{\Sigma_t-\Sigma^*}{\Sigma_0-\Sigma^*}=\frac{e^{-2t}(\Sigma_0-\Sigma^*)}{2\Sigma_0}$}\\
         &&{$+e^{-t}\sqrt{\frac{e^{-2t}(\Sigma_0-\Sigma^*)^2}{4\Sigma_0^2}+\frac{\Sigma^*}{\Sigma_0}}$}\\\hline
         OLD & $dX_t = -(\Sigma^*)^{-1}(X_t-\mu^*)dt+\sqrt{2}dB_t$&$\mu_t =\mu^*+e^{-(\Sigma^*)^{-1}t}(\mu_0-\mu^*)$\\\cline{1-2}
         LLD & $dX_t =  -(\Sigma^*)^{-1}(X_t-\mu^*)dt+\Sigma_t^{-1}(X_t-\mu_t)dt$&$\Sigma_t=\Sigma^*+e^{-2(\Sigma^*)^{-1}t}(\Sigma_0-\Sigma^*)$\\\hline
         \multirow{2}*{HAMCMC}&\multirow{2}*{$dX_t = -(X_t-\mu^*)dt+\sqrt{2\Sigma^*} dB_t.$}&$\mu_t = \mu^*+e^{-t}(\mu_0-\mu^*)$\\
         &&$\Sigma_t = \Sigma^*+e^{-2t}(\Sigma_0-\Sigma^*)$\\\hline
    \end{tabular}}
    \caption{Comparison among different Langevin dynamics on 1D Gaussian family.}\label{tab:my_label}}
\end{table}

Compared to OLD and LLD, the exponential convergence rate of $\mu_t$ and $\Sigma_t$ in NLD does not depend on $\Sigma^*$. This fact shows that the NLD is the Newton's flow for both the evolution of mean and variance in Gaussian process. We also note that the convergence rates of mean and variance are different in HAMCMC, while they are same in NLD. In section \ref{sec:num}, we use numerical examples to further demonstrate the differences between NLD and HAMCMC.


\section{Connection with Stein variational Newton's method}
The Stein variational Newton's method (SVN) is also a second-order method for sampling. It aims to minimize $J_\rho[\bphi]$, which evaluates the change of $E(\rho)$ along the transformation map $\bphi:\mbR^d\to \mbR^d$. 
\begin{equation}\label{jphi}
J_{\rho}[\bphi] = E((I+\bphi)\#\rho).
\end{equation}
Here $(I+\bphi)\#\rho$ denotes the pushforward density of $\rho$ along the map $I(x)+\bphi(x)$ and $I(x)$ is the identity map.
In each iteration, SVN solves $\bphi\in \mcS^d$ via the following equation:
\begin{equation}\label{equ:svn}
D^2 J_\rho[0](\bpsi,\bphi)=-D J_\rho[0](\bpsi),\quad \bpsi\in \mcS^d.
\end{equation}
Here $\mcS$ is the RKHS related to a kernel function $k(x,y)$ and  $\mcS^d=\mcS\times\dots\times \mcS$. Besides, $DJ_\rho$ and $D^2J_\rho$ denote the first and second variation of $J_\rho$. 

We note that the following relationships hold
$$
DJ_\rho[0][\bpsi ] = \int \bpsi^T (\nabla f +\nabla \log \rho) \rho dx. 
$$
$$
\begin{aligned}
&D^2 J_\rho[0](\bpsi,\bphi)=\mbE_{x\sim \rho }[\bphi(x)^T\nabla^2f(x) \bpsi(x)+\tr(\nabla \bphi(x)\nabla \bpsi(x))].
\end{aligned}
$$
If we restrict $\bpsi$ and $\bphi$ to be gradient vector fields. Namely, there exists $\Psi(x),\Phi(x):\mbR^d\to \mbR$ such that $\bpsi(x)=\nabla \Psi(x)$ and $\bphi(x)=\nabla \Phi(x)$. Then, we recover the gradient and Hessian operators in probability space with Wasserstein-2 metric.
$$
DJ_\rho[0][\nabla \Psi ] = \int (\lra{\nabla \psi,\nabla f}+\Delta \Psi)\rho dx = \int \Psi\grad^W E(\rho)dx. 
$$
$$ 
\begin{aligned}
D^2 J_\rho[0](\nabla \Psi,\nabla \Phi) =& \int \pp{\lra{\nabla^2 \Psi,\nabla^2\Phi} +\nabla \Psi^T \nabla^2f\nabla \Phi} \rho dx\\
=&\int \Psi \mcH_E^W (\rho) \Phi  dx.
\end{aligned}
$$
On the other hand, the kernelized Wasserstein Newton's method in each step solves $\Phi\in \mcS$ from \eqref{equ:wnf_kl}. Because $\mcS$ is a Hilbert space, this is equivalent to find $\Phi\in \mcS$ such that 
$$
\int \Psi \Hess^W E(\rho)[\Phi] dx = -\int \Psi \grad^W E(\rho)dx,\quad \forall \Psi\in \mcS,
$$
or equivalently,
$$
D^2 J_\rho[0](\nabla \Phi,\nabla \Psi) = -D J_\rho[0](\nabla \Psi),\quad \forall \Psi\in \mcS.
$$
This can be viewed as a restriction on \eqref{equ:svn}. Namely, we solve $D^2 J_\rho[0](\bpsi,\bphi)=-D J_\rho[0](\bpsi)$ in the space $\{\bphi=\nabla \Phi|\Phi\in \mcS\}$ instead of $\mcS^d$. 

\begin{remark}
We notice the differences between Wasserstein Newton and Stein variational Newton in formulations. SVN studies the second order variations w.r.t. transportation maps, while we focus on these variations w.r.t. densities.  Besides, we benefit from the utilization of gradient and Hessian operators in probability space with Wasserstein-2 metric. This allows us to to prove the convergence rate of information Newton's method in the sense of density. 
\end{remark}

\section{Proofs in Section \ref{sec:conv}}
In this section, we provide convergence proofs of information Newton's method with approximated Newton's direction in section \ref{sec:conv}. 

\subsection{Riemannian structure of probability space}
\label{ssec:riem}
We first provide some background knowledge for the Riemannian structure of probability space. For simplicity, we define the exponential map and other Riemannian operators on cotangent space.
\begin{definition}[Exponential map on cotangent space and its inverse]
The exponential map $\Exp_{\rho_0}$ is a mapping from the cotangent space $T_{\rho_0}^*\mcP(\Omega)$ to $\mcP(\Omega)$. Namely, $\Exp_{\rho_0}(\Phi)=\hat \rho_s|_{s=1}$. Here $\hat\rho_s, s\in[0,1]$ is the solution to geodesic equation \eqref{equ:geo} with initial conditions $\hat\rho_s|_{s=0} = \rho_0$, $\Phi_s|_{s=0}=\Phi$.

The inverse of the exponential map $\Exp_{\rho_0}(\rho_1)$ follows $\Exp_{\rho_0}^{-1}(\rho_1)=\mcG(\hat \rho_s)\p_s\hat \rho_s|_{s=0}$. Here $\hat\rho_s, s\in[0,1]$ is the solution to geodesic equation \eqref{equ:geo} with boundary conditions $\hat\rho_s|_{s=0} = \rho_0$ and $\hat\rho_s|_{s=1} = \rho_1$.
\end{definition}
We also denote $\Exp_{\rho}^{\alpha} (\Phi)$ to be the solution at time $t=\alpha$ to the geodesic equation \eqref{equ:geo} with initial values $\hat \rho_0 = \rho$ and $\Phi_0 = \Phi$. 
As a known result of Riemannian geometry, the geodesic curve has constant speed \citep{aitdm}. Namely, for $\Phi\in T_\rho^*\mcP(\Omega)$ and $\alpha>0$, we have
$$\Exp_\rho^\alpha(\Phi)= \Exp_\rho(\alpha\Phi).$$
And for $\rho_0,\rho_1\in\mcP(\Omega)$, it follows
$$ \|\Exp_{\rho_0}^{-1}(\rho_1)\|_{\rho_0}^2 = \mcD(\rho_0,\rho_1)^2.$$

We define high-order derivatives on the cotangent-space in Proposition \ref{prop:high}.
\begin{proposition}\label{prop:high}
For all $\Phi\in T_\rho^* \mcP(\Omega)$, it follows
$$
\begin{aligned}
E(\Exp^s_\rho(\Phi)) = &E(\rho)+s\nabla E(\rho)(\Phi)+\dots \frac{s^{n-1}}{(n-1)!}\nabla^{n-1}E(\rho)(\Phi,\dots, \Phi)\\
&+\frac{s^n}{n!} \nabla^n E(\Exp_\rho(\lambda\Phi))(\tau_\lambda \Phi,\dots,\tau_\lambda \Phi),
\end{aligned}
$$
where $\tau_\lambda$ is the parallelism from $\rho$ to $\Exp^\lambda_\rho(\Phi)$ and $\lambda\in(0,s)$. Here $\nabla^nE(\rho)$ defines a $n$-form on the cotangent space $T_\rho^* \mcP(\Omega)$. Namely, it is recursively defined by
$$
\nabla^n E(\rho)(\Phi_1,\dots,\Phi_n) = \left.\frac{\p}{\p s} \nabla^{n-1} E(\Exp_\rho( s\Phi_n))(\tau_s\Phi_1,\dots,\tau_s\Phi_{n-1})\right|_{s=0},
$$
where $\tau_s$ is the parallelism from $\rho$ to $\Exp_\rho( s\Phi_n)$.
\end{proposition}
\begin{proof}
We first show that
\begin{equation}\label{equ:tay_deri}
    \frac{\p}{\p s} \nabla^{n-1} E(\Exp_\rho^s(\Phi_n))(\tau_s\Phi_1,\dots,\tau_s\Phi_{n-1})=\nabla^n E(\Exp_\rho^s(\Phi_n))(\tau_s\Phi_1,\dots,\tau_s\Phi_n).
\end{equation}
From the definition, it follows that
$$
\begin{aligned}
&\frac{\p}{\p s} \nabla^{n-1} E(\Exp_\rho^s(\Phi_n))(\tau_s\Phi_1,\dots,\tau_s\Phi_{n-1})\\
=&\left.\frac{\p}{\p t} \nabla^{n-1} E(\Exp_\rho^{s+t}(\Phi_n))(\tau_{s+t}\Phi_1,\dots,\tau_{s+t}\Phi_{n-1})\right|_{t=0}\\
=&\left.\frac{\p}{\p t} \nabla^{n-1} E(\Exp_{\Exp_{\rho}^s(\Phi_n)}^{t}(\tau_s\Phi_n))(\tau_{t}\tau_s\Phi_1,\dots,\tau_{t}\tau_s\Phi_{n-1})\right|_{t=0}\\
=&\nabla^n E(\Exp_\rho^s(\Phi_n))(\tau_s\Phi_1,\dots,\tau_s\Phi_n).
\end{aligned}
$$
From \eqref{equ:tay_deri}, we can recursively compute that
$$
\frac{\p^n}{(\p s)^n} E(\Exp_\rho^s(\Phi)) = \nabla^nE(\Exp_\rho^s(\Phi))(\tau_s\Phi,\dots \tau_s\Phi).
$$
The Taylor expansion of $E(\Exp_\rho^s(\Phi))$ w.r.t. $s$ completes the proof.
\end{proof}

\subsection{Cauchy-Schwarz inequality}
To complete proofs in section \ref{sec:conv}, we introduce Lemma \ref{lem:cs}.
\begin{lemma}[Cauchy-Schwarz inequality] \label{lem:cs}
Suppose that $\mcH:T_\rho^*\mcP(\Omega)\to T_\rho\mcP(\Omega)$ is a self-adjoint linear operator and $\mcH$ is positive definite. Then, for $\Phi_1, \Phi_2\in T_\rho^*\mcP(\Omega)$, we have 
$$
\pp{\int \Phi_1\mcH\Phi_2 dx}^2\leq \pp{\int \Phi_1\mcH\Phi_1 dx} \pp{\int \Phi_2\mcH\Phi_2 dx}.
$$
\end{lemma}
\begin{proof}
The proof is quite similar to the Euclidean space. For all $s\in \mbR$, we have
$$
\begin{aligned}
0&\leq\int (\Phi_1+s\Phi_2)\mcH(\Phi_1+s\Phi_2) dx\\
=&s^2 \int \Phi_2\mcH\Phi_2 dx+2s\int \Phi_1\mcH\Phi_2 dx+\int \Phi_1\mcH\Phi_1 dx.
\end{aligned}
$$
Because the arbitrary choice of $s$, it follows that
$$
\pp{2\int \Phi_1\mcH_E(\rho)\Phi_2 dx}^2-4\pp{\int \Phi_1\mcH_E(\rho)\Phi_1 dx} \pp{\int \Phi_2\mcH_E(\rho)\Phi_2 dx}\geq 0.
$$
This completes the proof. 
\end{proof}

\subsection{Proofs of Proposition \ref{prop:est0} and Lemma \ref{lem:est}}
To prove Proposition \ref{prop:est0}, we introduce Lemma \ref{lem:tay2}.
\begin{lemma}\label{lem:tay2}
For all $\Phi\in T_{\rho_k}^*\mcP(\Omega)$, it follows
$$
\nabla E(\rho_k)(\Phi)+\nabla^2 E(\rho_k)(T_k,\Phi) = -\frac{1}{2}\nabla^3E(\Exp^{\lambda}_{\rho_k})(\tau_\lambda T_k,\tau_\lambda T_k,\tau_\lambda\Phi),
$$
where $\tau_\lambda$ is the parallelism from $\rho_k$ to $\Exp_{\rho_k}^\lambda(T_k)$ and $\lambda\in(0,1)$.
\end{lemma}
\begin{proof}
Consider an auxiliary function 
$$
A(s) = \nabla E(\Exp_{\rho_k}^s(T_k))(\tau_s\Phi).
$$
Directly from the definition of high-order derivatives, it follows
$$
\frac{\p}{\p s}A(s)=\nabla^2 E(\Exp_{\rho_k}^s(T_k))(\tau_sT_k,\tau_s\Phi),
$$
$$
\frac{\p^2}{\p s^2}A(s)=\nabla^3 E(\Exp_{\rho_k}^s(T_k))(\tau_sT_k,\tau_sT_k,\tau_s\Phi).
$$
Hence, we can compute the Taylor expansion
$$
\nabla E(\Exp_{\rho_k}^1(T_k))(\tau_1\Phi)= \nabla E(\rho_k)(\Phi)+\nabla^2 E(\rho_k)(T_k,\Phi)+\frac{1}{2}\nabla^3E(\Exp^{\lambda}_{\rho_k})(\tau_\lambda T_k,\tau_\lambda T_k,\tau_\lambda\Phi).
$$
On the other hand, we notice that
$$
\nabla E(\Exp_{\rho_k}^1(T_k))(\tau_1\Phi)=\nabla E(\rho^*)(\tau_1\Phi) = \int \tau_1\Phi \mcG(\rho)^{-1}\frac{\delta E}{\delta \rho^*}dx=0.
$$
This completes the proof.
\end{proof}

Based on Lemma \ref{lem:tay2}, 
Note that $\Phi_k=-\mcH_E(\rho_k)^{-1}\mcG(\rho_k)^{-1}\frac{\delta E}{\delta \rho_k}$.
Hence, it follows
$$
\mcH_E(\rho_k) \tau^{-1}T_{k+1} = \mcH_E(\rho_k)T_k+\mcG(\rho_k)^{-1}\frac{\delta E}{\delta \rho_k}-\mcH_E(\rho_k)R_k.
$$
For arbitrary $\Psi\in T_{\rho_k}^*\mcP(\Omega)$, we have
\begin{equation}\label{equ:tay1}
    \begin{aligned}
&\nabla^2E(\rho_k)(\Psi,\tau^{-1}T_{k+1})\\
=&\int \Psi\mcH_E(\rho_k) \tau^{-1}T_{k+1}dx\\
=&\int \Psi(\mcH_E(\rho_k)T_k+\mcG(\rho_k)^{-1}\frac{\delta E}{\delta \rho_k}-\mcH_E(\rho_k)R_k)dx\\
=&\nabla^2E(\rho_k)(\Psi,T_k)+\nabla E(\rho_k)(\Psi)-\nabla^2E(\rho_k)(\Psi,R_k)\\
=&-\frac{1}{2}\nabla^3E(\Exp^{\lambda}_{\rho_k})(\tau_\lambda\Psi,\tau_\lambda T_k,\tau_\lambda T_k)-\nabla^2E(\rho_k)(\Psi,R_k).
\end{aligned}
\end{equation}
Here the last equality comes from Lemma \ref{lem:tay2}. Based on the definition of parallelism, we notice the fact $$\|\tau_\lambda\Psi\|_{\Exp_{\rho_k}^\lambda(\Phi_k)}=\|\Psi\|_{\rho_k},\quad \forall \Psi\in T_{\rho_k}^*\mcP(\Omega).$$ 
Taking $\Psi=\tau^{-1}T_{k+1}$ in \eqref{equ:tay1}, applying Assumption \ref{asmp:1} and utilizing Lemma \ref{lem:cs} yields
$$
\begin{aligned}
&\delta_1 \|\tau^{-1}T_{k+1}\|_{\rho_{k}}^2\leq \left|\nabla^2E(\rho_k)(\tau^{-1}T_{k+1},\tau^{-1}T_{k+1})\right|\\
\leq&\frac{1}{2}\left|\nabla^3E(\Exp^{\lambda}_{\rho_k})(\tau_\lambda\tau^{-1}T_{k+1},\tau_\lambda T_k,\tau_\lambda T_k)\right|+\left|\nabla^2E(\rho_k)(\tau^{-1}T_{k+1},R_k)\right|\\
\leq&\frac{1}{2}\left|\nabla^3E(\Exp^{\lambda}_{\rho_k})(\tau_\lambda\tau^{-1}T_{k+1},\tau_\lambda T_k,\tau_\lambda T_k)\right|\\
&+\sqrt{\left|\nabla^2E(\rho_k)(R_k,R_k)\right|\left|\nabla^2E(\rho_k)(\tau_\lambda\tau^{-1}T_{k+1},\tau_\lambda\tau^{-1}T_{k+1})\right|}\\
\leq & \delta_3\|\tau_\lambda T_k\|_{\Exp_{\rho_k}^\lambda(\Phi_k)}^2\|\tau_\lambda\tau^{-1}T_{k+1}\|_{\Exp_{\rho_k}^\lambda(\Phi_k)}+\delta_2\|\tau^{-1}T_{k+1}\|_{\rho_k}\|R_k\|_{\rho_k}\\
=& \delta_3\|T_k\|_{\rho_k}^2\|\tau^{-1}T_{k+1}\|_{\rho_k}+\delta_2\|\tau^{-1}T_{k+1}\|_{\rho_k}\|R_k\|_{\rho_k}.
\end{aligned}
$$
Hence, it follows
$$
\|T_{k+1}\|_{\rho_{k+1}}=\|\tau^{-1}T_{k+1}\|_{\rho_{k}}\leq \frac{\delta_3}{\delta_1}\|T_k\|_{\rho_k}^2+\frac{\delta_2}{\delta_1}\|R_k\|_{\rho_k}.
$$
To prove Lemma \ref{lem:est}, we introduce the following Lemma \ref{lem:est1}.
\begin{lemma}\label{lem:est1}
We have following estimations
$$
\|\Phi_k\|_{\rho_k} = O(\|T_k\|_{\rho_k}),\quad \|T_{k+1}\|_{\rho_{k+1}}=O(\|T_k\|_{\rho_k}).
$$
\end{lemma}
\begin{proof}
From Assumption \ref{asmp:1} and Cauchy-Swarz inequality, it follows that
$$
\begin{aligned}
\|\Phi_k\|_{\rho_k}^2=&\int \Phi_k \mcG(\rho_k)^{-1} \Phi_k dx\leq \delta_1^{-1} \int \Phi_k \mcH_E(\rho_k)\Phi_k dx\\
=&\delta_1^{-1} \int \Phi_k \mcG(\rho_k)^{-1} \frac{\delta E}{\delta \rho_k} dx\leq \delta_1^{-1} \|\Phi_k\|_{\rho_k}\norm{\frac{\delta E}{\delta \rho_k}}_{\rho_k}.
\end{aligned}
$$
We also notice that from Lemma \ref{lem:tay2},
$$
\begin{aligned}
&\norm{\frac{\delta E}{\delta \rho_k}}_{\rho_k}^2= \nabla E(\rho_k)\pp{\frac{\delta E}{\delta \rho_k}}\\
=&\nabla^2 E(\rho_k)\pp{T_k,\frac{\delta E}{\delta \rho_k}}+O\pp{\|T_k\|_{\rho_k}^2\norm{\frac{\delta E}{\delta \rho_k}}_{\rho_k}}\\
=&O\pp{\|T_k\|_{\rho_k}\norm{\frac{\delta E}{\delta \rho_k}}_{\rho_k}}.
\end{aligned}
$$
As a result, we have $\|\Phi_k\|_{\rho_k}=O\pp{\norm{\frac{\delta E}{\delta \rho_k}}_{\rho_k}}=O\pp{\|T_k\|_{\rho_k}}$. We also note the triangle inequality
$$
|\|T_k\|_{\rho_k}-\|\Phi_k\|_{\rho_k}|\leq\|T_{k+1}\|_{\rho_{k+1}}\leq \|T_k\|_{\rho_k}+\|\Phi_k\|_{\rho_k}.
$$
This yields $\|T_{k+1}\|_{\rho_{k+1}}=O(\|T_k\|_{\rho_k})$.
\end{proof}
We finally show the estimation of $\|R_k\|_{\rho_k}$.  Based on the first-order approximation of the exponential map and the parallelsim, we have the following estimations
$$
\int \Psi(\rho^*-\rho_{k})dx = \int \Psi\mcG(\rho_k)^{-1} T_kdx+O(\|\Psi\|_{\rho_k}\|T_k\|_{\rho_k}^2),
$$
$$
\begin{aligned}
&\int \Psi(\rho_{k+1}-\rho_k)dx = \int \Psi\mcG(\rho_k)^{-1} \Phi_kdx+O(\|\Psi\|_{\rho_k}\|\Phi_k\|_{\rho_k}^2)\\
=&\int \Psi\mcG(\rho_k)^{-1} \Phi_kdx+O(\|\Psi\|_{\rho_k}\|T_k\|_{\rho_k}^2),
\end{aligned}
$$
and
$$
\begin{aligned}
&\int \Psi(\rho^*-\rho_{k+1})dx = \int \Psi\mcG(\rho_{k+1})^{-1} T_{k+1}dx+O(\|\Psi\|_{\rho_{k+1}}\|T_{k+1}\|_{\rho_{k+1}}^2)\\
=&\int \tau^{-1}\Psi\mcG(\rho_{k})^{-1} \tau^{-1} T_{k+1}dx+O(\|\Psi\|_{\rho_{k}}\|T_{k+1}\|_{\rho_{k+1}}^2+\|\Psi-\tau^{-1}\Psi\|_{\rho_{k}}\|T_{k+1}\|_{\rho_{k+1}}^2)\\
=&\int \tau^{-1}\Psi\mcG(\rho_{k})^{-1} \tau^{-1} T_{k+1}dx+O(\|\Psi\|_{\rho_{k}}\|T_{k+1}\|_{\rho_{k+1}}^2+\|\Psi\|_{\rho_{k}}\|\Phi_k\|_{\rho_k}\|T_{k+1}\|_{\rho_{k+1}}^2)\\
=&\int \Psi\mcG(\rho_{k})^{-1} \tau^{-1} T_{k+1}dx+O(\|\Psi\|_{\rho_{k}}\|T_{k+1}\|_{\rho_{k+1}}^2+\|\Psi-\tau^{-1}\Psi\|_{\rho_{k}}\|\tau^{-1}T_{k+1}\|_{\rho_{k}})\\
=&\int \Psi\mcG(\rho_{k})^{-1} \tau^{-1} T_{k+1}dx+O(\|\Psi\|_{\rho_{k}}\|T_{k+1}\|_{\rho_{k+1}}^2+\|\Psi\|_{\rho_{k}}\|\Phi_k\|_{\rho_k}\|T_{k+1}\|_{\rho_{k+1}})\\
=&\int \Psi\mcG(\rho_{k})^{-1} \tau^{-1} T_{k+1}dx+O(\|\Psi\|_{\rho_{k}}\|T_{k}\|_{\rho_{k}}^2).\\
\end{aligned}
$$
Furthermore, we have $R_k = T_k-\tau^{-1}T_{k+1} -\Phi_k$ and 
$$
\int \Psi(\rho^*-\rho_{k})dx-\int \Psi(\rho^*-\rho_{k+1})dx-\int \Psi(\rho_{k+1}-\rho_k)dx=0.
$$
This completes the proof.

\subsection{Proof of Theorem \ref{thm:conv_e}}
We first notice that 
\begin{equation}\label{e_grad}
\nabla E(\rho)(\Phi)=\int \Phi \mcG(\rho)^{-1} \frac{\delta E}{\delta \rho}dx,\quad \nabla^2E(\rho)(\Phi_1,\Phi_1) = \int \Phi \mcH_E(\rho) \Phi dx.
\end{equation}
By taking \eqref{e_grad} into Lemma \ref{lem:est} and utilizing \eqref{ass:A3}, we note that for $\sigma\in T_{\rho_k}^* \mcP(\Omega)$,
\begin{equation}\label{gsimga}
\int g_k \sigma dx =-\int \mcH_E(\rho_k) T_k\sigma dx+\mcO(\|\sigma\|_{\rho_k}\|T_k\|_{\rho_k}^2).
\end{equation}
Based on the Taylor expansion on the Riemannian manifold with \eqref{ass:A3}, it follows
$$
\begin{aligned}
E(\rho_{k+1}) = &E(\rho_k)+\alpha_k\int  \Phi_k \mcG(\rho_k)^{-1}\frac{\delta E}{\delta {\rho_k}}  dx+\frac{\alpha_k^2}{2}\int  \Phi_k \mcH_E(\rho_k) \Phi_k dx+\mcO(\| \Phi_k\|_{\rho_k}^3).
\end{aligned}
$$
Following \eqref{equ:hat_phik} and \eqref{ass:A5}, this yields
\begin{equation}\label{equ:rhokk}
\begin{aligned}
&E(\rho_{k+1}) - E(\rho_k) \\
=& -\alpha_k \int g_k \mcH_{E,P}g_kdx+\frac{\alpha_k^2}{2} \int g_k\mcH_{E,P} \mcH_E(\rho_k)\mcH_{E,P}g_kdx +\mcO(\| \Phi_k\|_{\rho_k}^3)\\
=&\frac{\alpha_k^2-2\alpha_k }{2}\int g_k \mcH_{E,P}g_kdx+\frac{\alpha_k^2}{2} \int g_k(\mcH_{E,P} \mcH_E(\rho_k)\mcH_{E,P}-\mcH_{E,P})g_kdx +\mcO(\| \Phi_k\|_{\rho_k}^3)\\
\leq&\frac{\alpha_k^2-2\alpha_k }{2}\int g_k \mcH_{E,P}g_kdx+\frac{\epsilon_2\alpha_k^2}{2} \int g_k\mcH_{E,P}(\rho_k) g_k dx+\mcO(\| \Phi_k\|_{\rho_k}^3).
\end{aligned}
\end{equation}
Similarly, by the Taylor expansion along with \eqref{ass:A3}, we have
\begin{equation}\label{equ:rhoks}
\begin{aligned}
&E(\rho^*)  -E(\rho_k)\\
= &\int g_k T_k dx+\frac{1}{2}\int T_k\mcH_E(\rho_k)T_kdx+\mcO(\|T_k\|_{\rho_k}^3)\\
=&-\frac{1}{2}\int  T_k \mcH_E(\rho_k) T_k dx+\mcO(\|T_k\|_{\rho_k}^3).
\end{aligned}
\end{equation}
According to \eqref{ass:A1}, \eqref{ass:A2} and Cauchy-Schwartz inequality, we have
$$
\begin{aligned}
\|\mcH_E(\rho_k)^{-1} g_k\|_{\rho_k}^2 =& \int \mcH_E(\rho_k)^{-1} g_k\mcG(\rho_k)^{-1}\mcH_E(\rho_k)^{-1} g_k dx \\
\leq &\delta_1^{-1}\int\mcH_E(\rho_k)^{-1} g_k\mcH_E(\rho_k)\mcH_E(\rho_k)^{-1} g_k dx\\
=&\delta_1^{-1}\int g_k\mcH_E(\rho_k)^{-1} \mcG(\rho_k)^{-1} \mcG(\rho_k) g_k dx\\
\leq &\delta_1^{-1} \|\mcH_E(\rho_k)^{-1} g_k\|_{\rho_k} \|\mcG(\rho_k) g_k\|_{\rho_k}.
\end{aligned}
$$
Besides, from the proof of Lemma \ref{lem:est1}, we have
$$
\|\mcG(\rho_k) g_k\|_{\rho_k} = \norm{\frac{\delta E}{\delta \rho_k}}_{\rho_k}=O(\|T_k\|_{\rho_k}). 
$$
This tells $\|\mcH_E(\rho_k)^{-1} g_k\|_{\rho_k}= O(\|\mcG(\rho_k) g_k\|_{\rho_k})=O(\|T_k\|_{\rho_k})$. Hence, by utilizing \eqref{gsimga} two times, we have
$$
\begin{aligned}
&\int g_k \mcH_E(\rho_k)^{-1}g_k dx \\
=&-\int \mcH_E(\rho_k)^{-1} T_k\mcH_E(\rho_k) g_k dx+\mcO(\|T_k\|_{\rho_k}^2\|\mcH_E(\rho_k)^{-1} g_k\|_{\rho_k})\\
=&-\int  T_k g_k dx+\mcO(\|T_k\|_{\rho_k}^3)\\
=&\int T_k \mcH_E(\rho_k) T_k dx +\mcO(\|T_k\|_{\rho_k}^3).
\end{aligned}
$$
This indicates 
\begin{equation}\label{equ:rhoks2}
\begin{aligned}
E(\rho^*)  -E(\rho_k)=&-\frac{1}{2}\int  T_k \mcH_E(\rho_k) T_k dx+\mcO(\|T_k\|_{\rho_k}^3)\\
=&-\frac{1}{2 }\int g_k \mcH_E(\rho_k)^{-1}g_k dx+\mcO(\|T_k\|_{\rho_k}^3).
\end{aligned}
\end{equation}
Following \eqref{ass:A6}, we note that
$$
\| \Phi_k \|_{\rho_k}=\norm{\mcH_{E,P} \mcG(\rho_k)^{-1}\frac{\delta E}{\delta \rho_k}}_{\rho_k} = \mcO\pp{\norm{\frac{\delta E}{\delta \rho_k}}_{\rho_k}}=\mcO(\|T_k\|_{\rho_k}).
$$
In summary, combining \eqref{ass:A4}, \eqref{equ:rhokk} and \eqref{equ:rhoks2}, we have
$$
\begin{aligned}
&E(\rho_{k+1}) - E(\rho^*) \\
\leq&E(\rho_k)-E(\rho^*)+\frac{\alpha_k^2-2\alpha_k }{2}\int g_k \mcH_{E,P}g_kdx\\
&+\frac{\epsilon_1\alpha_k^2}{2} \int g_k\mcH_{E,P}(\rho_k) g_k dx+\mcO(\| \Phi_k\|_{\rho_k}^3)\\
\leq& \frac{1}{2 }\int g_k \mcH_E(\rho_k)^{-1}g_k dx+\frac{\alpha_k^2-2\alpha_k }{2}\int g_k \mcH_E(\rho_k) g_kdx\\
&+\frac{|\alpha_k^2-2\alpha_k| \epsilon_1}{2} \int g_k \mcH_E(\rho_k)^{-1}g_k dx +\frac{\epsilon_2(1+\epsilon_1)\alpha_k^2}{2} \int g_k\mcH_{E}(\rho_k)^{-1} g_k dx+\mcO(\|T_k\|_{\rho_k}^3)\\
=&\pp{\frac{(\alpha_k-1)^2}{2}+\frac{|\alpha_k^2-2\alpha_k| \epsilon_1}{2}+\frac{\epsilon_2(1+\epsilon_1)\alpha_k^2}{2} } \int g_k \mcH_E(\rho_k)^{-1}g_k dx +\mcO(\|T_k\|_{\rho_k}^3).\\
\end{aligned}
$$
By taking $\alpha_k=1$ and utilizing \eqref{equ:rhoks2}, we have
$$
\begin{aligned}
E(\rho_{k+1}) - E(\rho^*)\leq& \frac{\epsilon_1+\epsilon_2+\epsilon_1\epsilon_2}{2} \int g_k \mcH_E(\rho_k)^{-1}g_k dx+\mcO(\|T_k\|_{\rho_k}^3)\\
=&(\epsilon_1+\epsilon_2+\epsilon_1\epsilon_2)(E(\rho_{k}) - E(\rho^*))+\mcO((E(\rho_{k}) - E(\rho^*))^{3/2}).
\end{aligned}
$$
The last equality comes from $\|T_k\|_{\rho_k}^2=O\pp{\int T_k \mcH_E(\rho_k) T_k dx}=O(E(\rho_{k}) - E(\rho^*))$. 

\subsection{Proof of Theorem \ref{thm:conv_s}}
For simplicity, denote $p_k = \hat \Phi_k-\Phi_k$. From the previous derivation, with $\alpha_k=1$, we note that
$$
\begin{aligned}
&E(\rho_{k+1}) - E(\rho_k) \\
=& - \int g_k (\mcH_{E,P}g_k+p_k)dx+\frac{1}{2} \int (p_k+\mcH_{E,P} g_k)\mcH_E(\rho_k)(\mcH_{E,P}g_k+p_k)dx +\mcO(\|\hat \Phi_k\|_{\rho_k}^3)\\
=&\frac{1 }{2}\int g_k \mcH_{E,P}g_kdx+\frac{1}{2} \int g_k(\mcH_{E,P} \mcH_E(\rho_k)\mcH_{E,P}-\mcH_{E,P})g_kdx \\
&- \int \pp{g_k p_k -\frac{1}{2} p_k(\mcH_E(\rho_k)\mcH_{E,P}+\mcH_{E,P}\mcH_E(\rho_k))g_k  }dx\\
&+\frac{1}{2}\int p_k \mcH_E(\rho_k) p_k dx +\mcO(\|\hat \Phi_k\|_{\rho_k}^3)\\
\leq&\frac{1}{2}\int g_k \mcH_{E,P}g_kdx+\frac{\epsilon_2}{2} \int g_k\mcH_{E,P}(\rho_k) g_k dx+\frac{\epsilon_3+\epsilon_4}{2}\int g_k\mcH_E^{-1}(\rho_k) g_k dx+\mcO(\|\hat \Phi_k\|_{\rho_k}^3).
\end{aligned}
$$
The last inequality further utilizes \eqref{ass:A7} and \eqref{ass:A8}. We also note that
$$
\|\hat \Phi_k\|_{\rho_k}\leq \|p_k\|_{\rho_k} +\|\Phi_k\|_{\rho_k}.
$$
And we have
$$
\begin{aligned}
\|p_k\|_{\rho_k}^2 = &\int p_k \mcG(\rho_k)^{-1} p_k dx\leq \frac{1}{\delta_1} \int p_k \mcH_E(\rho_k) p_k dx \\
\leq& \frac{\epsilon_4}{\delta_1} \int g_k \mcH_E(\rho_k)^{-1}g_k dx=\mcO(\|T_k\|_{\rho_k}^2).
\end{aligned}
$$
Hence, $\|\hat \Phi_k\|_{\rho_k}=\mcO\mcO(\|T_k\|_{\rho_k})$. As a result, by utilizing \eqref{equ:rhoks2}, we complete the proof. 

\subsection{Justification of Assumption \ref{asmp:3}}
To justify Assumption \ref{asmp:3}, we first introduce some definitions. 

For an energy function $E(\rho)$, we call it \textit{well-defined w.r.t. samples} if $E(\hat \rho)$ is well-defined for $\hat \rho = \frac{1}{N}\sum_{i=1}^N\delta(x-x_i)$, where $\delta$ is the Dirac-delta distribution. We denote 
$$
\hat \mcP(\Omega)=\mcP(\Omega)\cup\left\{\hat \rho= \frac{1}{N}\sum_{i=1}^N\delta(x-x_i)|x_i\sim \rho, \rho \in \mcP(\Omega)\right\}.
$$ 
\begin{remark}
Typical examples of such energy functions include
$$
E(\rho)=\int f(x)\rho(x)dx,
$$
where $f(x)$ is a smooth function. Or  
$$
E(\rho)=\int f(x;\rho) \rho(x)dx.
$$
Here $f(x;\rho)$ is well-defined w.r.t. samples for fixed $x$. For instance, $f(x;\rho)=\int w(x,y)\rho(y)dy$ for some smooth function $w(x,y)$. 
\end{remark}

We say that $\{\hat \rho_n\}\subseteq \hat \mcP(\Omega) $ weakly converges to $\rho\in \mcP(\Omega)$ if for any smooth (test) function $f$,
$$
\lim_{N\to\infty }\int f(x)\hat \rho_N(x) dx = \int f(x) \rho(x) dx.
$$
We say that $E(\rho)$ is \textit{convergent w.r.t. samples} if $E(\rho)$ is well-defined w.r.t. samples and
$$
\lim_{n\to\infty }E(\hat \rho_n) = E(\rho),
$$
for any $\{\hat \rho_n\}\subseteq \hat \mcP(\Omega) $ weakly converges to $\rho\in \mcP(\Omega)$. 

For $\rho\in \hat \mcP(\Omega)$, we define the variational problem
$$
J(\rho,\Phi) = \int \Phi \mcH_E(\rho) \Phi dx+2\int \Phi\mcG(\rho)^{-1}\frac{\delta E}{\delta \rho} dx+\lambda\int \Phi \mcR_\mcS \Phi dx.
$$
Suppose that $\|\Phi\|_{\mcS}$ is a norm in $\mcS$, which is independent of $\rho$. We further assumes that $\|\Phi\|_{\mcS}$ and the regularization term $\int \Phi \mcR_\mcS \Phi dx$ satisfy Assumption \ref{asmp:4}. 
\begin{assumption}\label{asmp:4}
There exists $\delta_5,\delta_6>0$ such that for all $\mcD(\rho,\rho^*)<\zeta$, 
\begin{equation}\label{ass:A9}
\delta_6 \|\Phi\|_\mcS^2\leq\|\Phi\|_{\rho}^2\leq \delta_5 \|\Phi\|_\mcS^2.\tag{A9}
\end{equation}

There exists $\delta_7\geq 0$, such that 
\begin{equation}\label{ass:A10}
\int \Phi \mcR_\mcS \Phi dx\leq \delta_7\|\Phi\|_\mcS^2.\tag{A10}
\end{equation}
\end{assumption}

Suppose that for fixed $\Phi\in \mcS$, $J(\rho, \Phi)$ is convergent w.r.t. samples. Then, for fixed $\rho\in \hat \mcP(\Omega)$, $J(\rho,\Phi)$ is well-defined and we denote $\Phi(\rho)$ as the minimizer of $\min_{\Phi\in \mcS}J(\rho,\Phi)$. Then, $\Phi(\rho)$ is well-defined w.r.t. samples. We then show that $\Phi(\rho)$ is convergent w.r.t. samples. 

For $\rho\in \hat \mcP(\Omega)$ satisfying \eqref{ass:A1}, we note that
$$
J(\rho,\Phi)\geq \delta_1 \|\Phi\|_{\rho}^2+2\int \Phi\mcG(\rho)^{-1}\frac{\delta E}{\delta \rho} dx+\lambda\int \Phi \mcR_\mcS \Phi dx.
$$
As a result, for fixed $\rho$, $J(\rho,\Phi)$ is $\delta_1$-strictly convex in $\Phi$ w.r.t. the norm $\|\cdot\|_{\rho}$, i.e.,
\begin{equation}\label{j_cvx}
J(\rho,\Phi_1)-J(\rho,\Phi_2)\geq \int (\Phi_1-\Phi_2)\left.\frac{\delta J(\rho,\Phi)}{\delta \Phi}\right|_{\Phi=\Phi_2} dx+\delta_1\|\Phi_1-\Phi_2\|_{\rho}^2.
\end{equation}

Similarly, for $\rho\in \hat \mcP(\Omega)$ satisfying \eqref{ass:A2}, we note that
$$
J(\rho,\Phi)\leq \delta_2 \|\Phi\|_{\rho}^2+2\int \Phi\mcG(\rho)^{-1}\frac{\delta E}{\delta \rho} dx+\lambda\delta_7\|\Phi\|_\mcS^2.
$$
Hence, this yields
\begin{equation}\label{j_cvx_l}
J(\rho,\Phi_1)-J(\rho,\Phi_2)\leq \int (\Phi_1-\Phi_2)\left.\frac{\delta J(\rho,\Phi)}{\delta \Phi}\right|_{\Phi=\Phi_2} dx+\delta_2\|\Phi_1-\Phi_2\|_{\rho}^2+\lambda\delta_7\|\Phi_1-\Phi_2\|_{\mcS}^2.
\end{equation}

\begin{lemma}\label{lem:j_opt}
Suppose that $\mcS\subseteq \mcF(\Omega)/\mbR$ is a Hilbert space. $J(\Phi)$ is strictly convex in $\Phi$ w.r.t. some norm. For a variational problem $\min_{\Phi\in \mcS}J(\Phi)$, the unique minimizer $\Phi^*$ satisfies
$$
\int  (\Psi-\Phi^*) \left.\frac{\delta J}{\delta \Phi}\right|_{\Phi=\Phi^*} dx = 0,\quad \forall \Psi\in \mcS.
$$
\end{lemma}
\begin{proof}
The variational problem $\min_{\Phi\in \mcS}J(\Phi)$ is equivalent to
$$
\min_{\Phi\in \mcF(\Omega)/\mbR, \Psi\in \mcS} J(\Phi),\quad\text{s.t.}\quad \Phi=\Psi.
$$
Consider the Lagrangian $\mcL(\Phi,\Psi,\lambda)=J(\Phi)+\int \lambda (\Phi-\Psi)dx$. The KKT conditions include:
$$
\frac{\delta J}{\delta \Phi}+\lambda=0,\quad \Psi=\Phi,\quad \int \lambda \tilde \Psi dx=0,\quad\forall \tilde \Psi\in \mcS.
$$
Here the equality holds up to a spatial-shift. As a result, for the minimizer $\Phi^*$, we have
$$
\int \left.\frac{\delta J}{\delta \Phi}\right|_{\Phi=\Phi^*} \tilde \Psi dx=0,\quad \forall \tilde \Psi\in \mcS. 
$$
Because $\Psi-\Phi^*\in \mcS$, this completes the proof. 
\end{proof}

\begin{proposition}
$\mcS$ is a Hilbert space. 
Suppose that \eqref{ass:A1} and \eqref{ass:A2} in Assumption \ref{asmp:1} further holds for $\rho\in \hat \mcP(\Omega)$. We assume the following statements hold.
\begin{itemize}
\item For $\rho \in \hat \mcP(\Omega)$, $\|\Phi(\rho)\|_{\mcS}$ is bounded. 
\item For fixed $\Phi\in \mcS$, $J(\rho; \Phi)$ is convergent w.r.t. samples.
\item For fixed $\Phi\in \mcS$, $\|\Phi\|_{\rho}^2$ is well-defined w.r.t. samples.
\end{itemize}
Then, under Assumption \ref{asmp:4}, $\Phi(\rho)$ is convergent w.r.t. samples. 
\end{proposition}
\begin{proof}
Suppose that $\Phi(\rho)$ is not convergent w.r.t. samples. Then, there exists $\{\hat \rho_n\}_{n=1}^\infty\subseteq \hat \mcP(\Omega)$ and $\epsilon>0$ such that $\hat \rho_n$ weakly converges to $\rho\in \mcP(\Omega)$, while $\|\Phi(\hat \rho_n)-\Phi(\rho)\|_{\rho}>\epsilon$. We note that
$$
\begin{aligned}
J(\hat \rho_n, \Phi(\hat \rho_n))-J(\rho,\Phi(\rho)) = & J(\hat \rho_n, \Phi(\hat \rho_n))-J(\rho,\Phi(\hat \rho_n))+J(\rho,\Phi(\hat \rho_n))-J(\rho,\Phi(\rho)) \\
=&J(\hat \rho_n, \Phi(\hat \rho_n))-J(\hat \rho_n,\Phi( \rho))+J(\hat \rho_n,\Phi( \rho))-J(\rho,\Phi(\rho)).
\end{aligned}
$$
Because $\Phi(\hat \rho_n)$ is the minimizer of $J(\hat \rho_n, \Phi)$, by applying \eqref{j_cvx} and Lemma \ref{lem:j_opt}, we have
$$
\begin{aligned}
&J(\hat \rho_n, \Phi(\hat \rho_n))-J(\hat \rho_n,\Phi( \rho))\leq -\delta_1 \|\Phi(\hat \rho_n)-\Phi(\rho)\|_{\hat \rho_n}^2\\
&\leq - \delta_1\delta_6\|\Phi(\hat \rho_n)-\Phi(\rho)\|_{\mcS}^2\leq-\frac{\delta_1\delta_6}{\delta_5}\|\Phi(\hat \rho_n)-\Phi(\rho)\|_{\rho}^2\leq-\frac{\delta_1\delta_6\epsilon^2}{\delta_5}.
\end{aligned}
$$
Similarly, because $\Phi( \rho)$ is the minimizer of $J( \rho, \Phi)$, we have
$$
J(\hat \rho_n,\Phi( \rho))-J(\rho,\Phi(\rho))\geq \delta_1\|\Phi(\hat \rho_n)-\Phi(\rho)\|_{\rho}^2\geq \frac{\delta_1\epsilon^2}{2}.
$$
Because $\mcS$ is a Hilbert space and $\{\Phi(\hat \rho_n)\}$ is bounded, according to the Banach-Alaoglu theorem, $\{\Phi(\hat \rho_n)\}$ is weakly sequentially compact. Namely, there exists a weakly convergent subsequent $\{\Phi(\hat \rho_{n_k})\}$ (which is also convergent because $\mcS$ is a Hilbert space). Suppose that this sequence converges to $\Phi^*$.  As a result, 
$$
\lim_{k\to \infty}J(\rho,\Phi(\hat \rho_{n_k}))= J(\rho,\Phi^*).
$$
From \eqref{j_cvx_l} and Assumption \ref{asmp:4}, we have 
$$
\begin{aligned}
J(\hat \rho_{n_k}, \Phi(\hat \rho_{n_k}))-J(\hat \rho_{n_k},\Phi^*)\geq&-\delta_2\|\Phi(\hat \rho_{n_k})-\Phi^*\|_{\hat \rho_{n_k}}^2-\lambda\delta_7 \|\Phi(\hat \rho_{n_k})-\Phi^*\|_{\hat \rho_{n_k}}^2 \\
\geq &-(\delta_2\delta_5+\lambda\delta_7)\|\Phi(\hat \rho_{n_k})-\Phi^*\|_\mcS^2.
\end{aligned}
$$
Hence, we have
$$
\lim_{k\to \infty}J(\hat \rho_{n_k},\Phi(\hat \rho_{n_k}))= J(\rho,\Phi^*).
$$
On the other hand, because $J(\rho, \Phi)$ is convergent w.r.t. samples for fixed $\Phi$, $\lim_{k\to \infty}J(\hat \rho_{n_k},\Phi(\rho))-J(\rho,\Phi(\rho))= 0$. Hence, for sufficiently large $k$, we have
$$
\begin{aligned}
&J(\hat \rho_{n_k}, \Phi(\hat \rho_{n_k}))-J(\rho,\Phi(\rho))\\
= &J(\hat \rho_{n_k}, \Phi(\hat \rho_{n_k}))-J(\rho,\Phi(\hat \rho_{n_k}))+J(\rho,\Phi(\hat \rho_{n_k}))-J(\rho,\Phi(\rho))\leq -\frac{\delta_1\delta_6\epsilon^2}{2\delta_5},
\end{aligned}
$$ 
and 
$$
\begin{aligned}
&J(\hat \rho_{n_k}, \Phi(\hat \rho_{n_k}))-J(\rho,\Phi(\rho))\\
= &J(\hat \rho_{n_k}, \Phi(\hat \rho_{n_k}))-J(\hat \rho_{n_k},\Phi( \rho))+J(\hat \rho_{n_k},\Phi( \rho))-J(\rho,\Phi(\rho))\geq \frac{\delta_1}{2}\epsilon^2.
\end{aligned}
$$
This leads to a contradiction.
\end{proof}

\bibliography{WIG}

\end{document}